\documentclass{amsart}
\pdfoutput=1
\usepackage[usenames,dvipsnames]{xcolor}
\RequirePackage[OT1]{fontenc}
\usepackage[a4paper,inner=3cm,outer=3cm,top=3cm,bottom=3cm]{geometry}
\usepackage{hyperref}
\hypersetup{colorlinks=true,citecolor=blue,urlcolor=blue, linkcolor=blue,linktocpage=true}
\usepackage[foot]{amsaddr}
\usepackage[]{amsmath}
\usepackage{amssymb,amsthm,amsfonts,amsbsy,latexsym,amsxtra}
\usepackage{stmaryrd}
\usepackage[cal=pxtx,scr=boondoxo]{mathalfa}

\makeatletter
\def\set@curr@file#1{%
  \begingroup
    \escapechar\m@ne
    \xdef\@curr@file{\expandafter\string\csname #1\endcsname}%
  \endgroup
}
\def\quote@name#1{"\quote@@name#1\@gobble""}
\def\quote@@name#1"{#1\quote@@name}
\def\unquote@name#1{\quote@@name#1\@gobble"}
\makeatother
\usepackage{graphicx}

\usepackage{todonotes}
\usepackage{appendix}
\usepackage{url}

\usepackage[capitalise]{cleveref}
\crefname{equation}{}{}
\Crefname{equation}{Eqn.~}{Eqns.~}
\crefname{enumi}{}{}
\Crefname{enumi}{}{}
\creflabelformat{enumi}{(#2#1#3)}

\usepackage{bbm,bm}
\usepackage{comment}
\usepackage{mathtools}
\usepackage{placeins}
\usepackage[labelfont={sl}]{caption}
\usepackage[labelfont={normalfont}]{subcaption}

\newtheorem{theorem}{Theorem}[section]
\newtheorem{lemma}[theorem]{Lemma}
\newtheorem{proposition}[theorem]{Proposition}
\newtheorem{corollary}[theorem]{Corollary}

\newtheorem{definition}[theorem]{Definition}
\newtheorem{assumption}[theorem]{Assumption}
\newtheorem{remark}[theorem]{Remark}

\newcommand{\proofends}{\qed}

\usepackage{enumerate}
\newenvironment{enumeratei}{\begin{enumerate}[\upshape (i)]}{\end{enumerate}}

\newenvironment{enumeraten}{\begin{enumerate}[\upshape (1)]}{\end{enumerate}}
\newenvironment{enumerateA}{\begin{enumerate}[\upshape (A)]}{\end{enumerate}}

\numberwithin{equation}{section}

\renewcommand{\P}{\mathbb{P}}
\newcommand{\E}{\mathbb{E}}
\newcommand{\Var}{\mathrm{Var}}

\newcommand{\cond}{\mid}
\newcommand{\bcond}{\;\big\vert\;}
\newcommand{\countin}{\,\vert\,}

\newcommand{\distr}{\sim}

\newcommand{\Unif}{\mathrm{Unif}}

\newcommand{\bal}{\begin{aligned}}
\newcommand{\eal}{\end{aligned}}
\newcommand*{\beq}{\begin{equation}}
\newcommand*{\eeq}{\end{equation}}
\newcommand*{\sss}{\scriptscriptstyle}
\newcommand{\wih}{\widehat}
\newcommand{\wit}{\widetilde}

\def \toinp    {\buildrel {\P}\over{\longrightarrow}}
\def \toindis  {\buildrel \textit{d}\over{\longrightarrow}}

\def \toLWC    {\buildrel \text{\rm loc}\over{\longrightarrow}}
\def \toLWCP   {\buildrel {\P\text{\rm -loc}}\over{\longrightarrow}}

\def \eqindis  {\buildrel \textit{d}\over{=}}

\DeclarePairedDelimiter\abs{\lvert}{\rvert}
\DeclarePairedDelimiter\norm{\lVert}{\rVert}
\DeclarePairedDelimiter\floor{\lfloor}{\rfloor}
\DeclareMathOperator{\supp}{supp}
\DeclareMathOperator{\disjointunion}{\stackrel{\centerdot}{\cup}}
\DeclareMathOperator{\dist}{dist}

\newcommand{\bmatch}{\omega}
\newcommand{\bmatchset}{\Omega}

\newcommand{\comrole}{\mapsfrom}
\newcommand{\comroles}{\stackrel{\otimes}{\longmapsfrom}}
\newcommand{\midmin}{\mathbin{\wedge}}

\newcommand{\vertices}{\SV}
\newcommand{\edges}{\SE}

\newcommand{\boundary}{\partial B}

\newcommand{\leftpar}{\SV^\msl}
\newcommand{\rightpar}{\SV^\msr}

\newcommand{\otreesize}{\beta_r}
\newcommand{\otreetype}[1]{\alpha_r(#1)}
\newcommand{\otreebefore}[1]{\beta_r^{{<}}(#1)}
\newcommand{\otreetypebefore}[1]{\alpha_r^{{<}}(#1)}

\newcommand{\graphs}{\SG}
\newcommand{\comgraphs}{\SH}
\newcommand{\isom}{\simeq}
\newcommand{\orderedisom}{\cong}
\newcommand{\groot}{o}
\newcommand{\otroot}{\vect{0}}
\newcommand{\vect}[1]{\underline{#1}}

\newcommand{\markset}{\SM}
\newcommand{\markfunc}{\Xi}
\newcommand{\nomark}{\varnothing}

\newcommand{\partt}{F}
\newcommand{\partitions}{\SF}

\newcommand{\nth}{\nn}

\newcommand{\ith}{^\text{th}}
\newcommand{\nn}[1]{{#1}^{\sss (n)}}

\newcommand{\RIG}{\mathrm{RIG}}
\newcommand{\RIGC}{\mathrm{RIGC}}
\newcommand{\BCM}{\mathrm{BCM}}

\newcommand{\com}{\mathrm{Com}}
\newcommand{\comvect}{\mathbf{Com}}
\newcommand{\BP}{\mathrm{BP}}

\newcommand{\rCP}{\mathrm{CP}}

\newcommand{\dloc}{d_{\mathrm{loc}}}
\newcommand{\Eudist}{d_{\mathrm{eucl}}}
\newcommand{\rmax}{r_{\mathrm{max}}}

\newcommand{\triproj}{\Delta^\msp}
\newcommand{\tricom}[1]{\Delta_{#1}^\msc}
\newcommand{\tricomrand}[1]{\Lambda_{#1}^\msc}
\newcommand{\clustering}{\mathrm{Cl}}
\newcommand{\overlapset}{\SL}

\newcommand{\intersection}{\mathscr{O}}
\newcommand{\indc}[1]{\ind_{C_{#1}}}
\newcommand{\neighbors}{\SN}

\newcommand{\ldeg}{\msl\text{-}\mathrm{deg}}
\newcommand{\rdeg}{\msr\text{-}\mathrm{deg}}
\newcommand{\bdeg}{\msb\text{-}\mathrm{deg}}
\newcommand{\cdeg}{\msc\text{-}\mathrm{deg}}
\newcommand{\pdeg}{\msp\text{-}\mathrm{deg}}

\newcommand{\ballevent}{\CB}
\newcommand{\goodevent}{\CE}

\newcommand{\dmax}[1]{d_{\mathrm{max}}^{#1}}

\newcommand{\CB}{\mathcal{B}}

\newcommand{\CE}{\mathcal{E}}

\newcommand{\CI}{\mathcal{I}}

\newcommand{\CR}{\mathcal{R}}

\newcommand{\he}{\mathscr{h}}

\newcommand{\msl}{\mathscr{l}}
\newcommand{\msr}{\mathscr{r}}
\newcommand{\msb}{\mathscr{b}}
\newcommand{\msp}{\mathscr{p}}
\newcommand{\mss}{\mathscr{s}}
\newcommand{\msc}{\mathscr{c}}

\newcommand{\SV}{\mathscr{V}}

\newcommand{\SB}{\mathscr{B}}

\newcommand{\SE}{\mathscr{E}}
\newcommand{\SF}{\mathscr{F}}
\newcommand{\SG}{\mathscr{G}}
\newcommand{\SH}{\mathscr{H}}

\newcommand{\SL}{\mathscr{L}}
\newcommand{\SM}{\mathscr{M}}
\newcommand{\SN}{\mathscr{N}}
\newcommand{\SP}{\mathscr{P}}

\newcommand{\bitd}{\boldsymbol{d}}

\newcommand{\bitp}{\boldsymbol{p}}
\newcommand{\bitq}{\boldsymbol{q}}

\newcommand{\R}{\mathbb{R}}
\newcommand{\N}{\mathbb{N}}
\newcommand{\Z}{\mathbb{Z}}

\newcommand{\ind}{\mathbbm{1}}

\newcommand*{\ro}{\varrho}

\newcommand{\eps}{\varepsilon}

\usepackage[backend=bibtex,style=numeric,useprefix=false,sorting=nyt, doi=false,url=false,maxnames=10,abbreviate=false,giveninits=true]{biblatex}

\begin{document}

\title[Description of $\RIGC$]{Random intersection graphs with communities}
\author[R.v.d.Hofstad]{Remco van der Hofstad}
\author[J.Komj\'athy]{J\'ulia Komj\'athy}
\author[V.Vadon]{Vikt\'oria Vadon}
\date{\today}

	\address{Department of Mathematics and
	    Computer Science, Eindhoven University of Technology, P.O.\ Box 513,
	    5600 MB Eindhoven, The Netherlands.}

	\email{r.w.v.d.hofstad@tue.nl, j.komjathy@tue.nl, v.vadon@tue.nl}

\begin{abstract}
Random intersection graphs model networks with communities, assuming an underlying bipartite structure of groups and individuals, where these groups may overlap. Group memberships are generated through the bipartite configuration model. Conditionally on the group memberships, the classical random intersection graph is obtained\hyphenation{ob-tained} by connecting individuals when they are together in at least one group. We generalize this definition, allowing for arbitrary community structures within the groups. 

In our new model, groups might overlap and they have their own internal structure described by a graph, the classical setting corresponding to groups being complete graphs. Our model turns out to be tractable. We analyze the overlapping structure of the communities, derive the asymptotic degree distribution and the local clustering coefficient. These proofs rely on local weak convergence, which also implies that subgraph counts converge. We further exploit the connection to the bipartite configuration model, for which we also prove local weak convergence, and which is interesting in its own right.
\end{abstract}

\maketitle

\section{Introduction}\label{s:introduction}
Communities are local structures that are more densely connected than the network average. They are present in numerous real-life networks \cite{GirNew02}, for example in the Internet, in collaboration networks and in social networks, and offer a possible explanation for the often observed high clustering (transitivity) \cite[Chapter 7.9, 11]{Newm10}.

 There are several possible reasons why communities arise, e.g.\ an underlying geometry or properties shared by the vertices. We focus on networks with an underlying structure of individuals and groups that they are part of. While our terminology and examples are mainly taken from social networks, the model is applicable for any network that builds on some kind of group structure. Such structures exist in many real-life networks \cite{GuiLat04,GuiLat06}, the most evident example being collaboration networks, like the Internet movie database IMDb or the ArXiv. In these examples, the `individuals' are the actors and actresses or the authors, and the `groups' are the movies or articles they collaborate in. We can also consider a social network based on groups, where `groups' can represent families, common interests, workplaces or cities.

Due to the complexity of real-world networks, they are often modeled using \emph{random graphs} \cite{Bol01,Dur07,JanLucRuc00}. Properties and processes of interest, e.g.\ distances, clustering, network evolution and information or epidemic spreading processes, are studied on the random graph models to predict their behavior on real-life networks. An underlying group structure such as mentioned above is modeled using bipartite graphs, where the two partitions correspond to the individuals (people) and the groups (or attributes), and an edge represents a group membership, see \cref{fig:group_memberships}. The historical random graph model for networks with group structure is the random intersection graph ($\RIG$) first introduced in \cite{Sing96}. Over the years, several ways were introduced to generate the (random) bipartite graph of group memberships \cite{BloGodJawKurRyb15}: ranging from independent percolation on the complete bipartite graph (binomial $\RIG$ \cite{FilSchSin00,KarSchSin99,Sing96} or inhomogeneous RIG \cite{BloDam13,DeiKet09}), through pre-assigning the number of group memberships to each individual and connecting them to uniformly chosen groups (uniform $\RIG$ \cite{BalGer2009unifRIG,Ryb2011unifRIG} or generalized $\RIG$ \cite{Bloz10,Bloz13,Bloz17,BloJawKur13,GodeJaw2003}), to pre-assigning the number of group memberships to each individual as well as the number of group members to each group, then matching these ``tokens'' uniformly (i.e., the group memberships are generated via the bipartite configuration model) \cite{CouLel15,New2003configRIG}. What all of these models have in common is that once the group memberships are generated, each two individuals that share a group are connected. As a result, groups (communities) do overlap, while each community is a complete graph, see \cref{fig:RIG}, which may not be a realistic assumption for large communities.

\begin{figure}[htb]
\centering
\begin{subfigure}[c]{0.6\textwidth}
	\centering
	\includegraphics[width=0.8\textwidth]{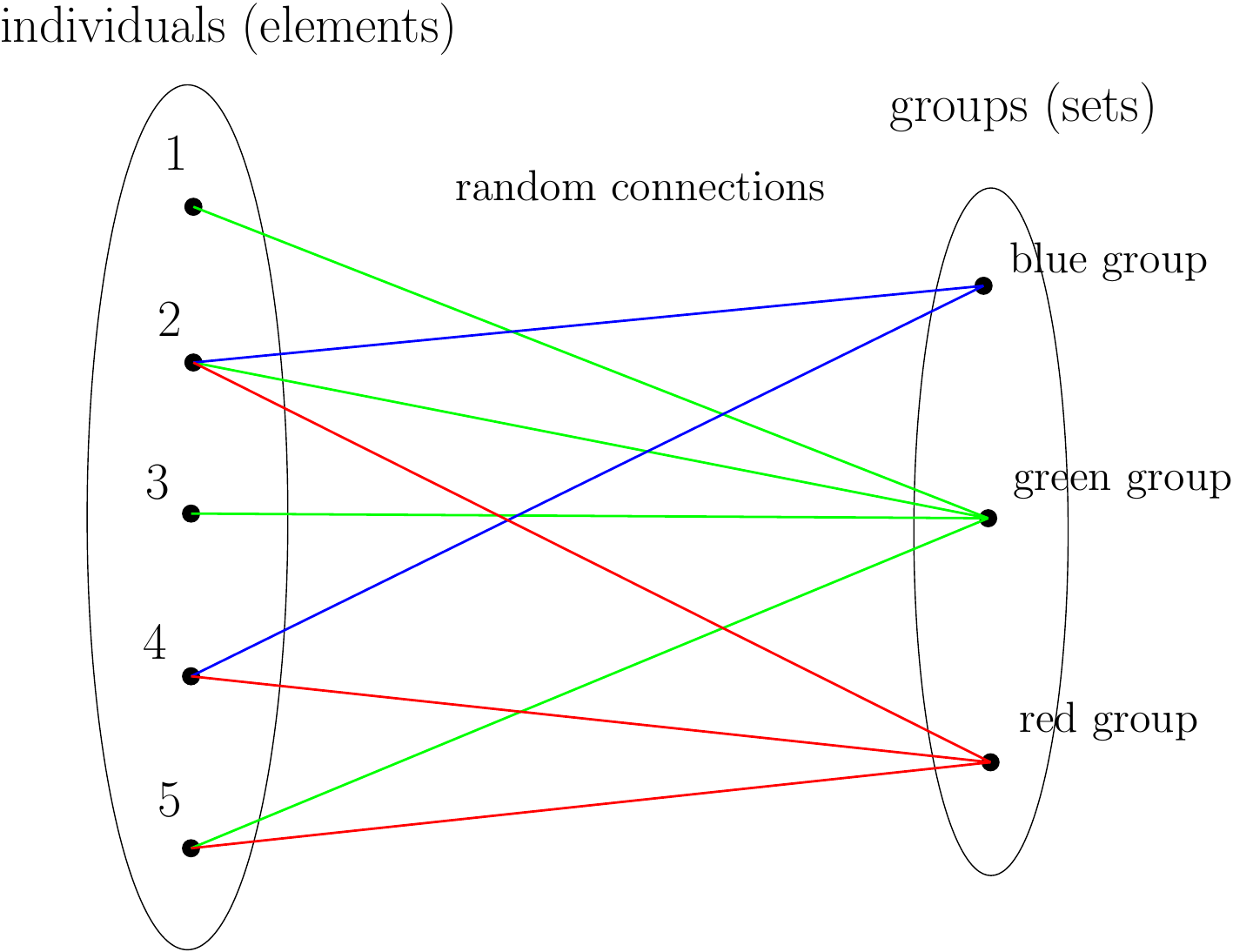}
	\caption{\emph{Modeling group memberships as a (random) bipartite graph}}
	\label{fig:group_memberships}
\end{subfigure}\hspace{5px}
\begin{subfigure}[c]{0.36\textwidth}
\centering
	\begin{subfigure}[t]{\textwidth}
		\centering
		\includegraphics[width=0.6\textwidth]{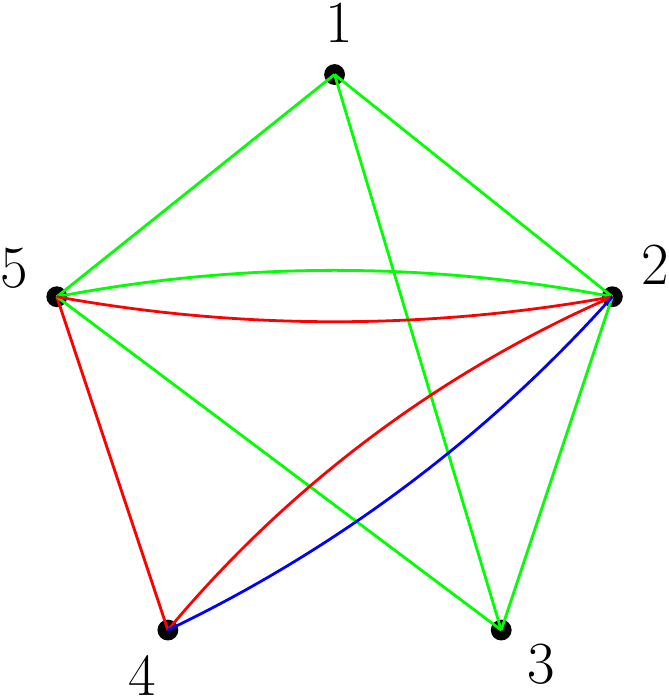}
		\caption{\emph{$\RIG$: union of cliques}}
		\label{fig:RIG}
	\end{subfigure}\\
	\begin{subfigure}[t]{\textwidth}
		\centering
		\includegraphics[width=0.6\textwidth]{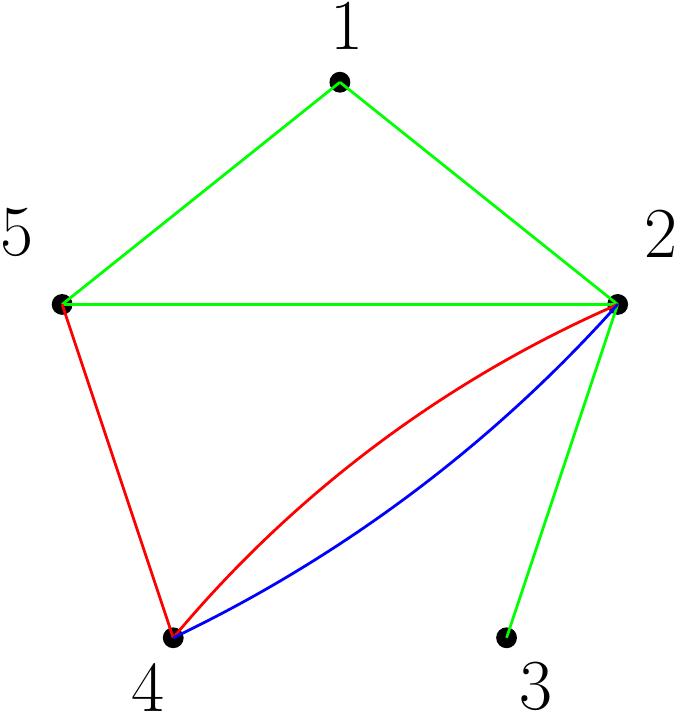}
		\caption{\emph{$\RIGC$: arbitrary communities}}
		\label{fig:RIGC}
	\end{subfigure}
\end{subfigure}
\caption{Two models for overlapping communities: RIG and RIGC}
\label{fig:RIG_vs_RIGC}
\end{figure}

One easy and natural way to go about this is thinning communities \cite{KarLeuwLes18,New2003configRIG}, however this may not give the full generality we desire. The recently introduced hierarchical configuration model (HCM) \cite{SteHofLeuw2015,SteHofLeuw2016powerlaw}, that extends the household model \cite{BallSirlTrap09,BallSirlTrap10}, offers an alternative approach, using arbitrary communities as building blocks with random connections \emph{between} the communities, resulting in non-overlapping communities. In this paper, we aim to bridge the gap: we introduce a new random graph model, the random intersection graph with community structure ($\RIGC$), that accommodates arbitrary, yet at the same time overlapping, communities, see \cref{fig:RIGC}.

The $\RIGC$ model is flexible in terms of the choice of parameters, ranging from i.i.d.\ random variables to data taken from real-life networks, see \cref{ss:RIGC_discussion} for a brief discussion. The model also turns out to be analytically tractable. In this paper, we keep our assumptions as general as possible, and present results on the overlapping structure, local properties of the model (including local weak convergence, degree structure and non-trivial clustering). Its global properties, including the existence and quantification of the so-called giant component (a unique linear-sized connected component), and percolation on the $\RIGC$ model are studied in the companion paper \cite{vdHKomVad19globArxiv}. The proofs rely on the connection to the bipartite configuration model that generates the group memberships. The matching results that we present on the bipartite configuration model are hence both instrumental to the $\RIGC$ and of independent interest.

\smallskip
\paragraph*{\textbf{Outline of the paper}} 

The rest of this paper is organized as follows. In \cref{s:RIGC}, we introduce the random intersection graph with community structure ($\RIGC$), state our results and provide a brief discussion. In \cref{s:BCM}, we introduce the underlying bipartite configuration model ($\BCM$), relate it to the $\RIGC$ model, and prove our main results for the $\BCM$. We provide the proofs for the $\RIGC$ in \cref{s:proof_local}.

\smallskip
\paragraph*{\textbf{Notational conventions}}
We will consider a sequence of graphs and consequently, a sequence of input parameters, both indexed by $n\in\N$. We note that $n$ only serves as the index; it does not necessarily mean the size or any other parameter of the graph, which allows for studying more general (growing) graph sequences. We often omit the dependence on $n$ to keep the notation light, as long as it does not cause confusion. Throughout this paper, we distinguish the set of positive integers as $\Z^+$ and the set of non-negative integers as $\N$. The notions $\toinp$ and $\toindis$ stand for convergence in probability and convergence in distribution (weak convergence), respectively. We write $X\eqindis Y$ to mean that the random variables $X$ and $Y$ have the same distribution. For an $\N$-valued random variable $X$ such that $\E[X]<\infty$, we define its \emph{size-biased} distribution $X^\star$ and the transform $\wit X$ with the following probability mass functions (pmf): for all $k\in\N$,
\beq\label{eq:def_sizebiasing} \P(X^\star = k) = k \,\P(X=k) / \E[X], 
\qquad \P(\wit X = k) = \P(X^\star - 1 = k). \eeq
We say that a sequence of events $(A_n)_{n\in\N}$ occurs with high probability (whp), if $\lim_{n\to\infty}\P(A_n)=1$. For two (possibly) random sequences $(X_n)_{n\in\N}$ and $(Y_n)_{n\in\N}$, we say that $X_n = o_{\sss\P}(Y_n)$ if $X_n/Y_n \toinp 0$ as $n\to\infty$. We denote the set $[n]:=\{1,2,\ldots,n\}$ and the indicator of an event $A$ by $\ind_A$. For a graph $G$, we denote its vertex set by $\vertices(G)$, its size by $\abs{G} = \abs{\vertices(G)}$ and its edge set by $\edges(G)$.

\FloatBarrier
\section{Model and results}
\label{s:RIGC}
In this section, we give a formal definition of the $\RIGC$ model and present our results on its local properties, as well as provide a discussion on its applicability.

\subsection{Definition of the random intersection graph with communities}
\label{ss:RIGC_def}
First, we give a short, intuitive description of the \emph{random intersection graph with communities}, followed by a detailed, formal construction. After introducing the parameters, the construction happens in two steps. 
First, we construct the \emph{community structure}: an underlying bipartite graph that represents the group memberships, where all the randomness arises from. Then we explain how to derive the $\RIGC$ based on the given community structure.

\smallskip
\paragraph*{\textbf{Intuitive model description}}
The aim of the model is to create a network that uses given community graphs as its building blocks, but at the same time allows them to overlap. We achieve this by thinking of vertices in the community graphs as \emph{community roles} that may be taken by the individuals. The individuals are represented as a distinct set of vertices, and we allow them to take on (possibly several) community roles by assigning them membership tokens. Each membership token corresponds to one community role taken, and we match membership tokens with community roles one-to-one, uniformly at random (uar). Then, we identify each individual with all the community roles it takes, ``gluing'' together the community graphs, which introduces overlaps and creates the (much more interconnected) network.

\smallskip
\paragraph*{\textbf{Parameters}}
Intuitively, we think of the individuals being placed on the left-hand side (lhs) and the groups (communities) on the right-hand side (rhs), and consequently we sometimes refer to them as $\msl$-vertices and $\msr$-vertices, respectively. We denote the set of individuals by $\leftpar = [N_n]$, where the number of individuals $N_n$ satisfies $N_n\to\infty$ as $n\to\infty$. Similarly, we denote the set of communities $\rightpar = [M_n]$, where $M_n\to\infty$ is to be defined later.

In this paper, we will encounter three types of relevant degrees, as we work with three different types of graphs: the $\RIGC$ model itself, the bipartite graph used to generate its community memberships, and the community graphs we use as building blocks. The notion ``degree'' is reserved for the most natural concept, namely, the number of connections of the individual in the resulting $\RIGC$; we sometimes refer to this notion of degree as ``projected degree'' ($\msp$-degree) for clarity. 
On the level of the underlying bipartite graph, the role of ``degrees'' is taken by the number of group memberships (for individuals) and the number of community members (for groups). Hence we introduce the concept of $\msl$-degrees and $\msr$-degrees (of $\msl$- and $\msr$-vertices, respectively), that we may collectively refer to as bipartite degrees ($\msb$-degrees). 
Within the community graphs, we will refer to the degree of a community vertex as its community degree ($\msc$-degree). We soon introduce notation for all three types of degrees.

As mentioned above, the number of group memberships of an individual $v\in\leftpar$ is called its $\msl$-degree which we denote by $d_v^\msl = \ldeg(v)$. For a community $a\in\rightpar$, we denote its community graph by $\com_a$ and we suppose that it comes from the set of possible community graphs $\comgraphs$, defined as follows. Let $\comgraphs$ be the set of (non-empty,) simple, finite, connected graphs, and equip each graph with an arbitrary fixed labeling, so that any two isomorphic community graphs are labeled in the exact same way. (We do allow several communities to have the same community graph.) Without loss of generality (wlog), we assume that $H\in\comgraphs$ is labeled by the set $[\abs{H}]$. 
We call the size $\abs{\com_a}$ of the community graph the $\msr$-degree of $a$, denoted by $d_a^\msr = \rdeg(a)$. We collect the $\msl$- and $\msr$-degrees and the community graphs in the vectors $\bitd^\msl := (d_v^\msl)_{v\in\leftpar}$, $\bitd^\msr := (d_a^\msr)_{a\in\rightpar}$ and $\comvect := (\com_a)_{a\in\rightpar}$, respectively.
Wlog we assume that $\bitd^\msl \geq 1$ and $\bitd^\msr \geq 1$ (element-wise) for each $n$, as isolated vertices can simply be excluded by adjusting $N_n$ and $M_n$. 
Also note that $\bitd^\msr$ is derived from $\comvect$, thus the $\RIGC$ is parametrized by the pair $(\bitd^\msl,\comvect)$. 
For a visual representation of the parameters, see \cref{fig:params_matching}.

\begin{figure}[hbt]
\centering
\includegraphics[width=0.8\textwidth]{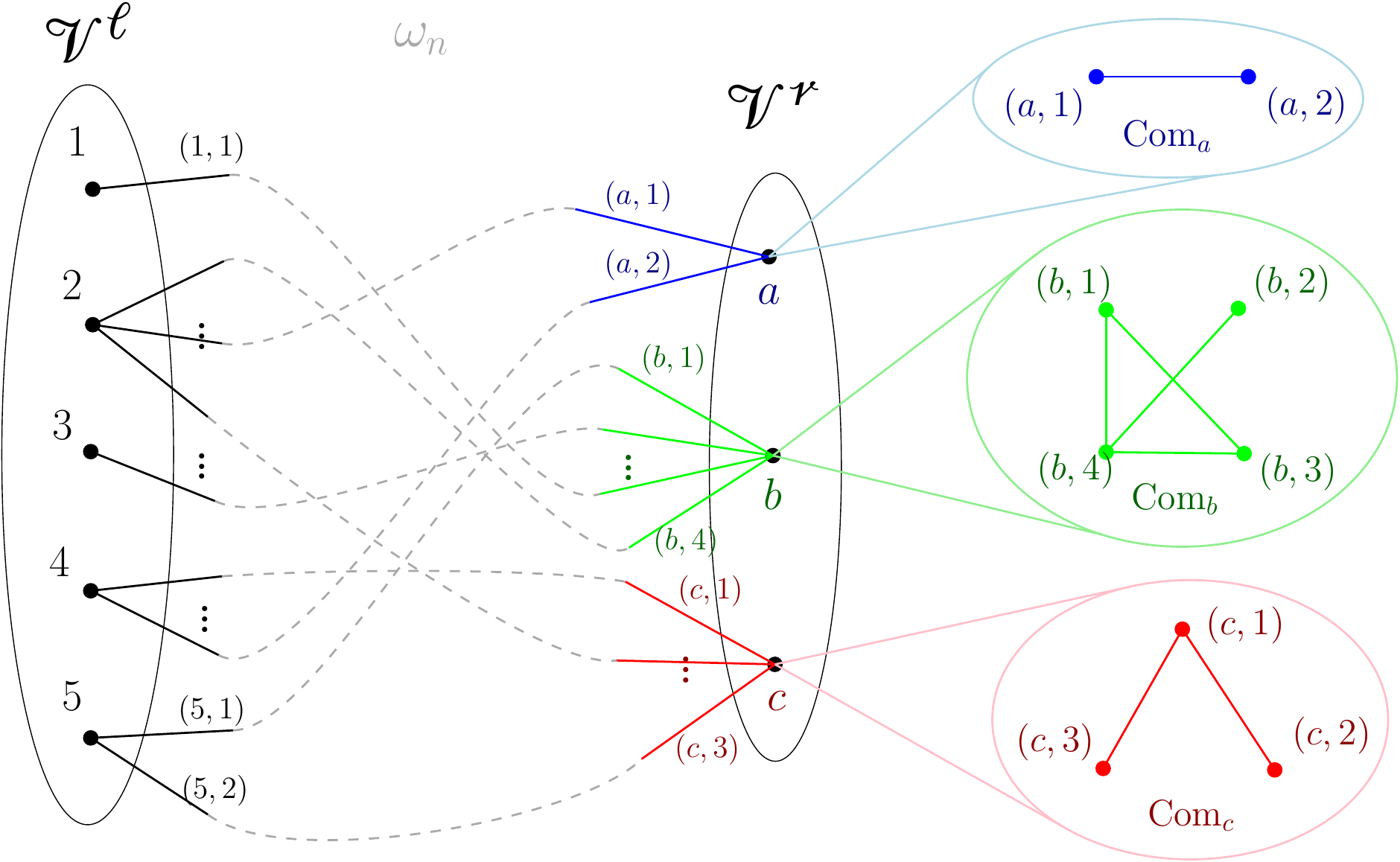}
\caption{An example of the parameters. \\ \small
Individuals form the lhs partition $\leftpar$, and their $\msl$-degree, i.e., the number of group memberships, is represented by outgoing half-edges. 
Communities form the rhs partition $\rightpar$, and each is assigned an arbitrary connected community graph. As before, we represent the $\msr$-degree, i.e., the number of community members, by outgoing half-edges. In fact, each half-edge represents a specific vertex (role) in the community graph, thus they are labeled the same way. 
In the next step, we assign group memberships (community roles) through a (bipartite) matching of the half-edges.
}
\label{fig:params_matching}
\end{figure}

\smallskip
\paragraph*{\textbf{Group memberships}}
Recall that the $\msl$-degree of $v\in\leftpar$ denotes the number of group memberships of $v$, that we intuitively think of as giving $\ldeg(v)$ membership tokens to $v$. We represent them as $\ldeg(v)$ $\msl$-half-edges incident to $v$ and label them by $(v,i)_{i\in[\ldeg(v)]}$. 
Let us denote the union of all vertices in community graphs by $\vertices(\comvect)$, that we call the set of community roles or community vertices. 
For a community vertex $j\in\vertices(\com_{a})$, we can uniquely identify $j$ by the tuple $(a,l)$, where $l$ is the vertex label of $j$ in $\com_{a}$. Now, similarly with individuals, we give each group $a\in\rightpar$ $\rdeg(a)$ community role tokens, represented by $\rdeg(a)$ $\msr$-half-edges incident to $a$ and labeled by $(a,l)_{l\in[\rdeg(a)]}$, so that we can represent $j\in\vertices(\comvect)$ by the $\msr$-half-edge $(a,l)$.

Next, we introduce the random matching of membership tokens and community role tokens. 
To ensure that the half-edges can indeed be matched, we assume and denote
\beq\label{eq:def_halfedges} \he_n := \sum_{v\in\leftpar} d_v^\msl = \sum_{a\in\rightpar} d_a^\msr. \eeq
Let $\bmatchset_n$ denote the set of all possible bijections between the $\msl$-half-edges $(v,i)_{i\in[\ldeg(v)], v\in\leftpar}$ and the $\msr$-half-edges $(a,l)_{l\in[\rdeg(a)],a\in\rightpar}$.\footnote{Equivalently, we can think of $\bmatchset_n$ as bijections between the $\msl$-half-edges and $\vertices(\comvect)$, due to each $\msr$-half-edge $(a,l)$, $l\in[\rdeg(a)], a\in\rightpar$ corresponding to a unique community vertex $j\in\vertices(\comvect)$.} Let the group memberships be determined by a \emph{uniform random bipartite matching} (bipartite configuration) $\bmatch_n \distr \Unif[\bmatchset_n]$. 

\begin{remark}[Algorithmic pairing]
\label{rem:pairingalgo}
We can produce the uniform bipartite matching $\bmatch_n$ sequentially, as follows. In each step, we pick an arbitrary unpaired half-edge, and match it to a uniform unpaired half-edge of the opposite type (so that we always match one $\msl$-half-edge and one $\msr$-half-edge). The arbitrary choices may even depend on the past of the pairing process, as long as we pair them uar with one the remaining half-edges.
\end{remark}

\begin{definition}[The ``underlying BCM'']
\label{def:bcm}
Considering the half-edges as tokens to form edges, the bipartite matching $\bmatch_n$ also determines a bipartite (multi)graph, defined as follows. For each matched pair of an $\msl$-half-edge $(v,i)$ and $\msr$-half-edge $(a,l)$, add an edge with label $(i,l)$ between $v$ and $a$. We call this edge-labeled graph the \emph{underlying bipartite configuration model} ($\BCM$). As the edge labels allow us to reconstruct the paired half-edges, the underlying $\BCM$ is an equivalent representation of the bipartite matching $\bmatch_n$, and thus encodes the group memberships.

Deleting the edge-labels, we obtain a bipartite version of the configuration model, i.e., the \emph{bipartite configuration model} with degree sequences $(\bitd^\msl,\bitd^\msr)$.
\end{definition}

\begin{figure}[hbt]
\centering
\begin{subfigure}[b]{0.68\textwidth}
	\centering
	\begin{subfigure}[b]{\textwidth}
		\centering
		\includegraphics[width=\textwidth]{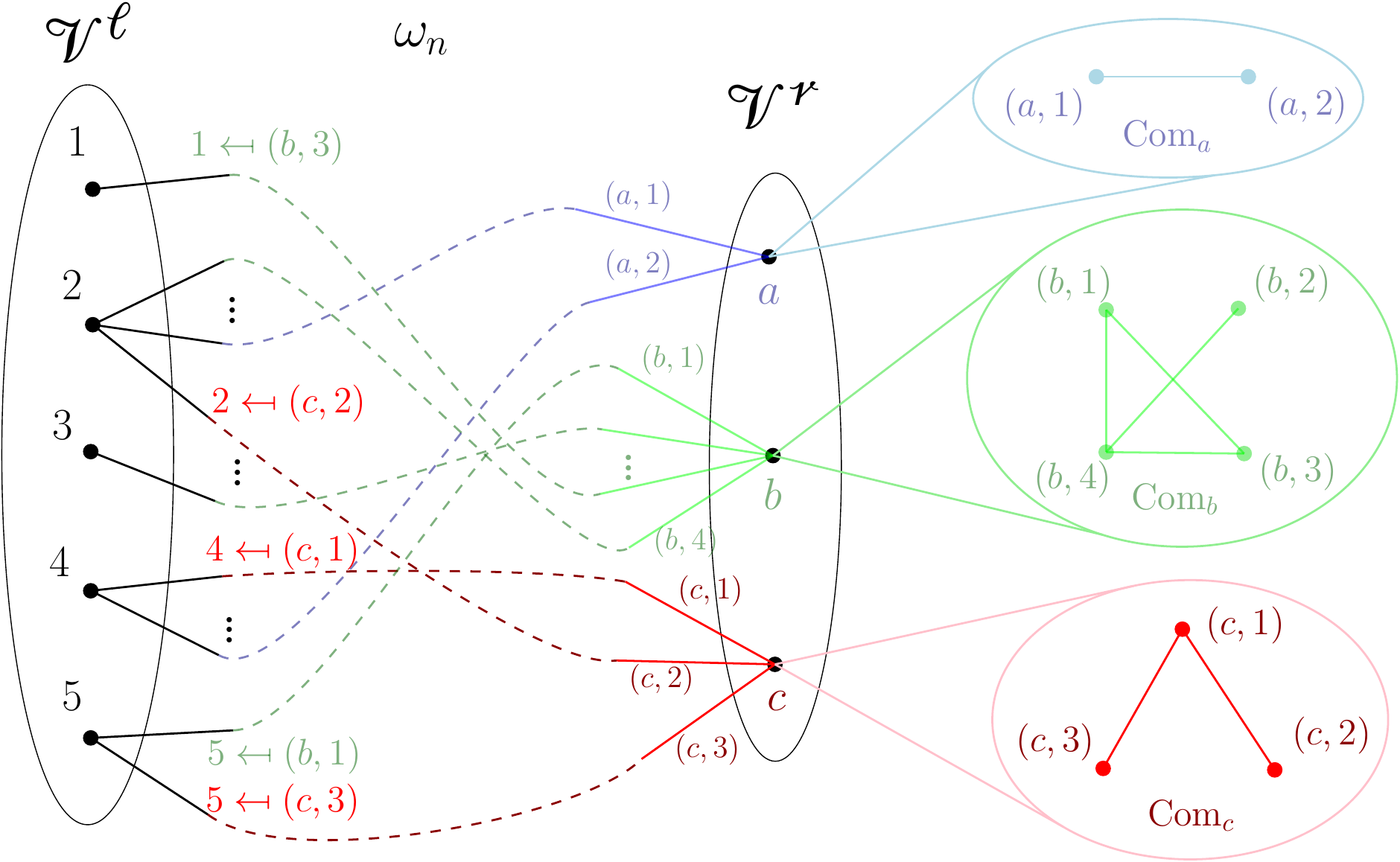}
		\caption{\emph{Community roles assigned by the matching of half-edges}\\ The community roles of $\com_c$ (highlighted) have been assigned to individuals $4$, $2$ and $5$ (in this order).}
	\end{subfigure}\\	
	\begin{subfigure}[b]{\textwidth}
		\centering
		\begin{subfigure}[b]{0.48\textwidth}
			\centering
			\includegraphics[width=0.9\textwidth]{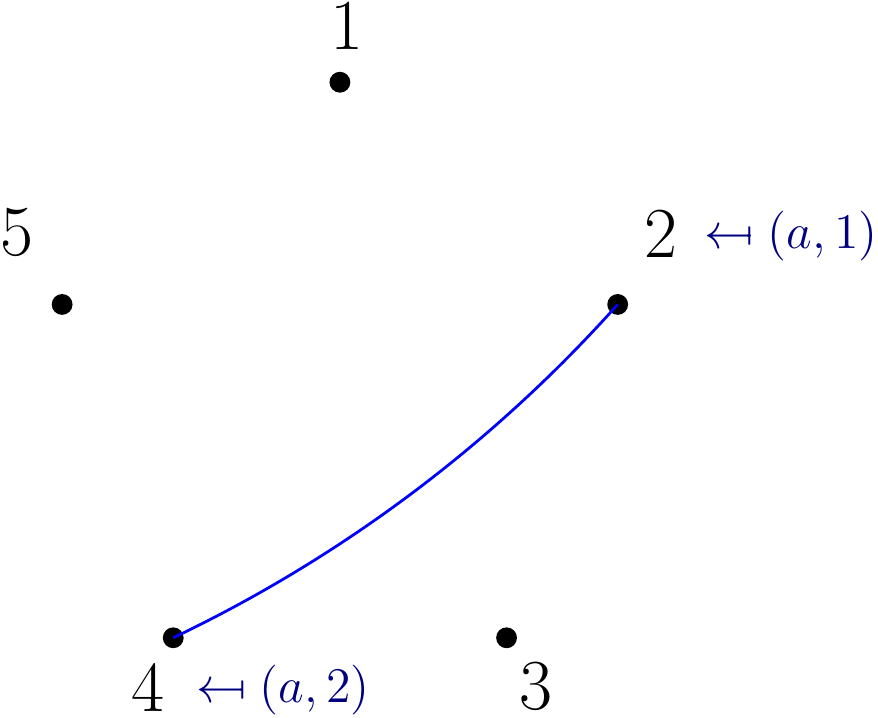}
		\end{subfigure}
		\begin{subfigure}[b]{0.48\textwidth}
			\centering
			\includegraphics[width=0.9\textwidth]{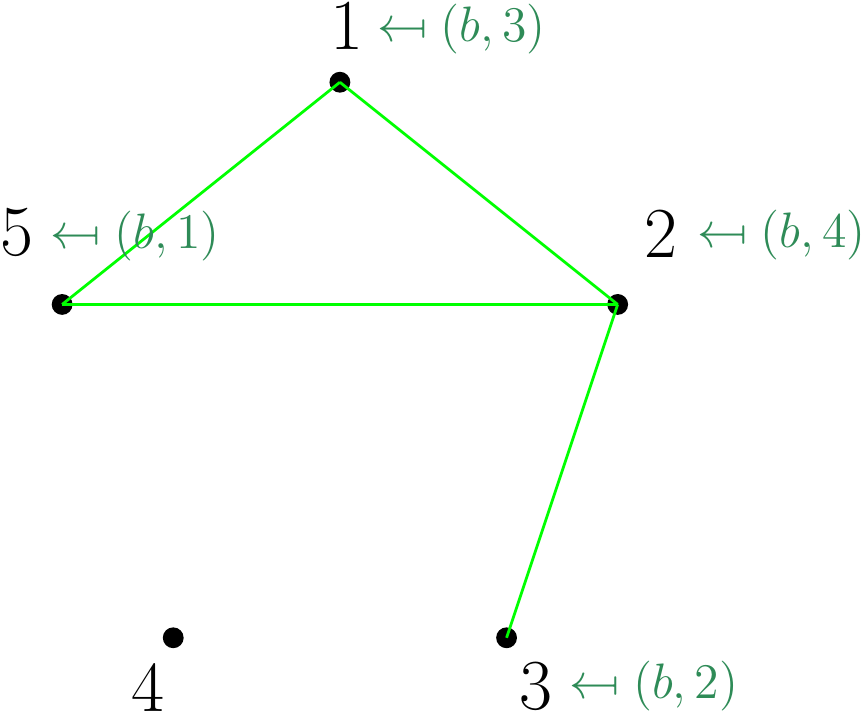}
		\end{subfigure}
		\caption{\emph{Projection of $\com_a$ and $\com_b$}\\ Obtained similarly to $\com_c$.}
	\end{subfigure}
\end{subfigure}
\hfill
\begin{subfigure}[b]{0.3\textwidth}
	\centering
	\begin{subfigure}[b]{\textwidth}
		\centering
		\includegraphics[width=0.95\textwidth]{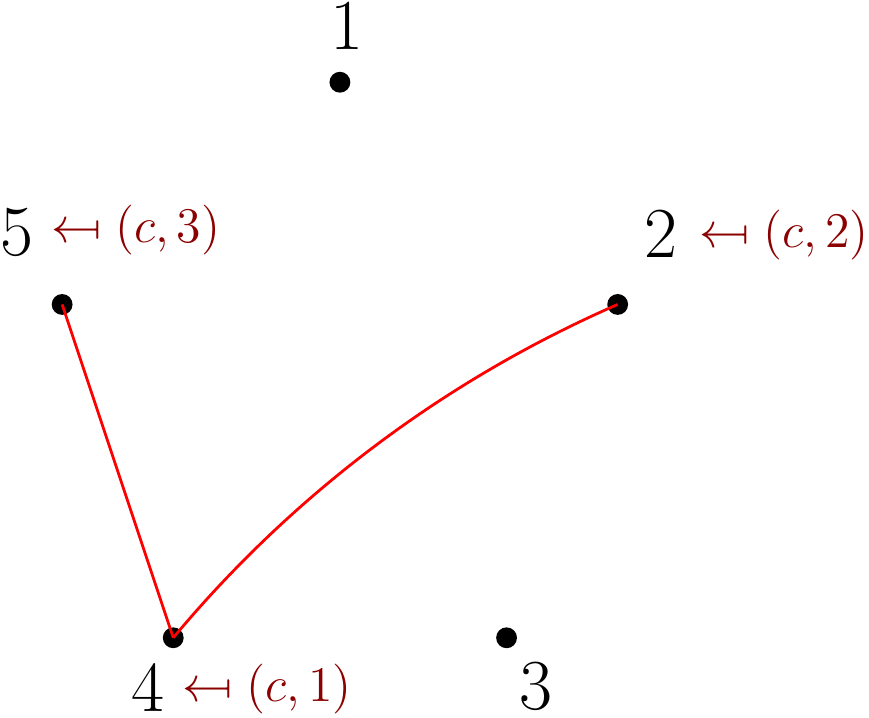}
		\caption{\emph{Projection of $\com_c$}\\ Each edge in $\com_c$ is copied to the corresponding individuals (that are assigned the community roles forming the edge), e.g.\ edge $((c,1),(c,2))$ becomes edge $(4,2)$. We do allow multigraphs.}
	\end{subfigure}\\
	\begin{subfigure}[b]{\textwidth}
		\centering
		\includegraphics[width=\textwidth]{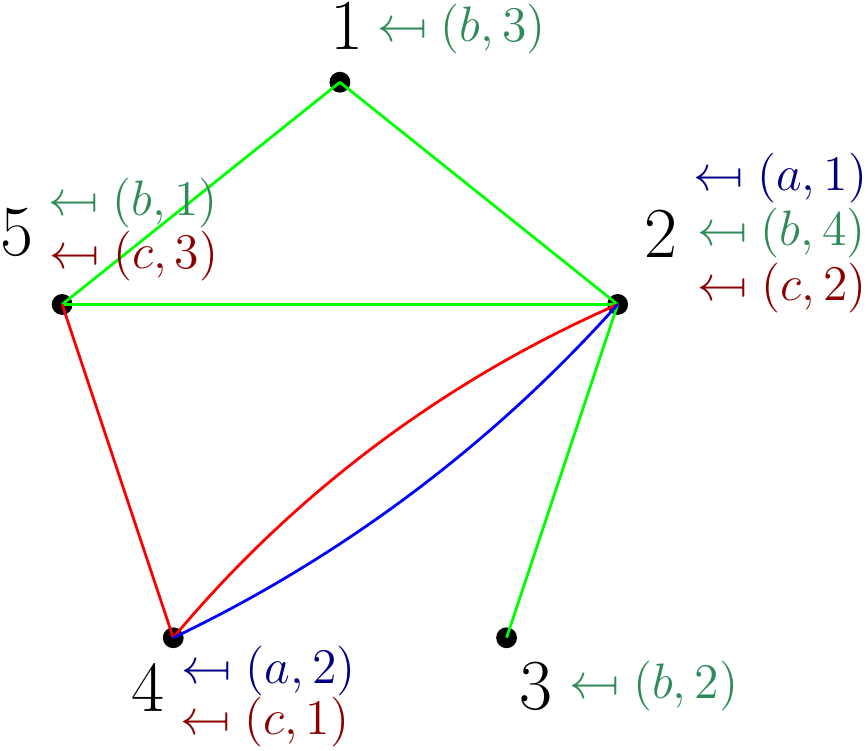}
		\caption{\emph{The resulting $\RIGC$} \\ Obtained by combining the projection of each community. We do allow multigraphs.}
	\end{subfigure}
\end{subfigure}
\caption{The community projection}
\label{fig:memberships_projection}
\end{figure}

\smallskip
\paragraph*{\textbf{The ``community projection''}}
We now introduce the \emph{community projection}, i.e., the method of projecting the community graphs to the individuals and generating the $\RIGC$ model, given the realization of the uniform(ly random) bipartite matching $\bmatch_n$. This procedure is deterministic, and the only randomness of the model comes from the choice of $\bmatch_n$, thus we can think of the community projection as an operator $\SP$ from $\bmatchset_n$ to the space of multigraphs. Alternatively, since the underlying $\BCM$ (see \cref{def:bcm}) provides an equivalent representation of the bipartite matching $\bmatch_n$, we can think of the projection as an operator\footnote{This operator can be further generalized as an operator mapping any bipartite graph, that we may interpret as the graph of group memberships, into a network.} that maps the underlying $\BCM$ into the $\RIGC$. 
We will describe the multigraph $\RIGC$ by its edge multiplicities.

Recall that the $\msr$-half-edge labeled $(a,l)$ represents the community role (community vertex) $j\in\vertices(\com_a)$ with vertex label $l$, and the $\msl$-half-edge $(v,i)$ is one of the membership tokens of $v\in\leftpar$. Then, if $(v,i)$ and $(a,l)$ are matched by $\bmatch_n$, this intuitively means one of the community roles taken by $v$ is $j$. We denote this by $v \comrole j$. Note that each community role $j$ is assigned to a \emph{unique} individual $v$, however each individual $v$ has $\ldeg(v)$ community roles $j$ that are assigned to it. We want to ``identify'' each individual with all community roles taken, and we carry this out by copying each edge between community roles $j_1, j_2 \in \vertices(\com_a)$ (for each community $a$) to the individuals $v \comrole j_1$ and $w \comrole j_2$. We emphasize that each community edge is copied individually, even when $v=w$ or when there is already an edge (or more) between $v$ and $w$; that is, we allow self-loops and multi-edges (see \cref{ss:RIGC_discussion} for a discussion on multigraphs). 

Let us denote the disjoint union of edges in all community graphs by $\edges(\comvect)$, and we refer to this set as community edges. 
Now, we shift perspective to obtain the multiplicity $X(v,w;\bmatch_n)$ of an edge $(v,w)$ (for $v,w\in\leftpar$) for a given bipartite matching $\bmatch_n$. We can do so by counting the number of community edges $(j_1,j_2)$ such that the community roles $j_1$ and $j_2$ are taken by $v$ and $w$ (in some order), formally:
\beq\label{eq:def_edge_multiplicities} X(v,w) = X(v,w;\bmatch_n) := 
\sum_{(j_1,j_2)\in\edges(\comvect)} \ind_{\{ v \comrole j_1, w \comrole j_2 \} \cup \{ v \comrole j_2, w \comrole j_1 \}}. \eeq
The \emph{random intersection graph with communities} $\RIGC(\bitd^\msl,\comvect)$ is the random multigraph given by the edge multiplicities $(X(v,w))_{v,w\in\leftpar}$ determined by the uniform(ly random) bipartite matching $\bmatch_n$.

\FloatBarrier
\subsection{Notation and assumptions}
\label{ss:cond_result}
In this section, we introduce the quantities and assumptions that are crucial throughout the paper.

\smallskip
\paragraph*{\textbf{Bipartite degrees}}
Throughout this paper, we make use of the following description of the $\msb$-degree sequences. Let $V_n^\msl \sim \Unif[\leftpar]$ and $V_n^\msr \sim \Unif[\rightpar]$ denote uniformly chosen $\msl$- and $\msr$-vertices respectively, and define
\beq\label{eq:def_ldeg_rdeg} D_n^\msl := \ldeg\bigl(V_n^\msl\bigr),
\qquad D_n^\msr := \rdeg(V_n^\msr). \eeq
Then the pmf
\begin{subequations}\label{eq:def_pq_pmf}
\beq\label{eq:def_p_k} \nth p_k := \abs{\{v\in\leftpar:\, d^\msl_v = k\}}/N_n, \eeq
for $k\in\Z^+$, describes the distribution of $D_n^\msl$ as well as the empirical distribution of $\bitd^\msl$. Similarly, we can describe $D_n^\msr$ and $\bitd^\msr$ by the pmf
\beq \label{eq:def_q_k} \nth q_k := \abs{\{a\in\rightpar:\, d_a^\msr = k\}}/M_n. \eeq
\end{subequations}
We collect the pmfs in the (infinite-dimensional) probability vectors $\nth\bitp = (\nth p_k)_{k\in\Z^+}$, $\nth\bitq = (\nth q_k)_{k\in\Z^+}$.

\smallskip
\paragraph*{\textbf{The empirical community distribution}}
Recall that $\comgraphs$ denotes the set of possible community graphs: simple, connected, finite graphs, each $H\in\comgraphs$ equipped with an arbitrary, fixed labeling using $[\abs{H}]$ as labels, so that any two community graphs that are isomorphic are labeled in the exact same way. 
For a fixed $H\in\comgraphs$, define
\beq\label{eq:def_VH} \vertices_H^\msr := 
\{a\in\rightpar:\, \com_a=H\}. \eeq
We introduce the pmf
\beq\label{eq:def_mu} \nth \mu_H := M_n^{-1} \abs{\vertices_H^\msr}, \qquad \nth {\bm\mu} = (\nth \mu_H)_{H\in\comgraphs}, \eeq
so that $\nth{\bm\mu}$ describes the empirical pmf of $\comvect$ as well as the pmf of $\com_{V_n^\msr}$, with $V_n^\msr\sim\Unif[\rightpar]$. 
For $k\in\Z^+$, define the (finite) set
\beq\label{eq:def_Hk} \comgraphs_k := \bigl\{ H\in\comgraphs:\, \abs{H}=k \bigr\}. \eeq
Note that since $d_a^\msr = \abs{\com_a}$, $\nth\bitq$ from \eqref{eq:def_q_k} can be obtained by $\nth q_k = \sum_{H\in\comgraphs_k} \nth\mu_H$.

\smallskip
\paragraph*{\textbf{Community degrees and triangles}}
Let us denote the disjoint union of vertices in all community graphs by $\vertices(\comvect)$, that we refer to as the \emph{set of community roles}. 
To a community role $j\in\vertices(\comvect)$, we assign the vector $(d_j^\msc,\tricom{j})$, where $d_j^\msc$ denotes the degree of $j$ in its community graph and $\tricom{j}$ denotes the number of triangles that $j$ is part of within its community graph. 
Let $J_n \distr \Unif[\vertices(\comvect)]$ denote a community role chosen uar.\footnote{Note that the community that $J_n$ is part of is chosen in a \emph{size-biased fashion}, and then a vertex in that community is chosen uniformly at random.} Define the random vector $(D_n^\msc,\tricomrand{n}) := (d_{J_n}^\msc,\tricom{J_n})$, keeping in mind that its coordinates are \emph{dependent}. Define the pmf 
\beq\label{eq:def_rho} \nth \ro_{(k,t)} := \frac1{\he_n} \sum_{j\in\vertices(\comvect)} \ind_{ \{ (d_{j}^\msc,\tricom{j}) = (k,t) \} }
\qquad \nth{\bm\ro} := \bigl( \nth \ro_{(k,t)} \bigr)_{k\in\Z^+,0\leq t\leq \binom{k}{2}}, \eeq
so that $\nn{\bm\varrho}$ describes the joint distribution of $(D_n^\msc,\tricomrand{n})$ as well as the empirical distribution of $(d_{j}^\msc, \tricom{j})_{j\in\vertices(\comvect)}$.

\smallskip
\paragraph*{\textbf{Projected degrees}}
For $v\in\leftpar$, its (random) projected degree, i.e., degree in the $\RIGC$ is by definition given in terms of the edge multiplicities (see \eqref{eq:def_edge_multiplicities}) as
\beq\label{eq:def_deg} d_v^\msp = \pdeg(v) := X(v,v) + \sum_{w\in\leftpar} X(v,w) = 2 X(v,v) + \sum_{w\in\leftpar,w\neq v} X(v,w). \eeq
However, it is more intuitive to look at $\pdeg(v)$ in terms of the community roles taken by $v$. Recall that each community edge incident to some $j$ such that $v \comrole j$ is added between $v$ and some other vertex, thus $j$ contributes $\cdeg(j)$ to the degree of $v$. Then 
\beq\label{eq:pdeg_decompose} \pdeg(v) = \sum_{j: v \comrole j} d_j^\msc. \eeq
Analogously to $D_n^\msl$, with $V_n^\msl \distr \Unif[\leftpar]$ as before, we define
\beq\label{eq:def_Dnproj} D_n^\msp := \pdeg(V_n^\msl). \eeq
Recall that $\pdeg(v)$ is random for each $v\in\leftpar$, due to $\bmatch_n$ being random. Thus, $D_n^\msp$ has two sources of randomness: $V_n^\msl$ and $\bmatch_n$. 
We denote the \emph{random} empirical cumulative distribution function (cdf) of $D_n^\msp$ as
\beq\label{eq:def_Fnproj} F_n^\msp(x) = F_n^\msp(x;\bmatch_n) := \frac{1}{N_n} \sum_{v\in\leftpar} \ind_{ \{ \pdeg(v) \leq x \} } 
=: \P\bigl( D_n^\msp \leq x \bcond \bmatch_n \bigr), \eeq
where $\P(\;\cdot \cond \bmatch_n)$ denotes the conditional probability with respect to (wrt) $\bmatch_n$.

\smallskip
\paragraph*{\textbf{Assumptions}} Recall \eqref{eq:def_ldeg_rdeg}, \eqref{eq:def_pq_pmf} and \eqref{eq:def_mu}. We can now summarize our assumptions on the model parameters, in particular, the conditions under which our results hold:

\begin{assumption}
\label{asmp:convergence} The conditions for the empirical distributions are summarized as follows:
\begin{enumerateA}
\item\label{cond:limit_ldeg} There exists a random variable $D^\msl$ with pmf $\bitp$ s.t.\ $\nth\bitp \to \bitp$ pointwise as $n\to\infty$, i.e.,
\beq D_n^\msl \toindis D^\msl. \eeq
\item\label{cond:mean_ldeg} $\E[D^\msl]$ is finite, and as $n\to\infty$,
\beq \E[D_n^\msl] \to \E[D^\msl]. \eeq
\item\label{cond:limit_com} There exists a probability mass function $\bm\mu$ on $\comgraphs$ such that $\nth{\bm\mu}\to\bm\mu$ pointwise as $n\to\infty$.
\begin{enumeraten}
\item\label{cond:limit_rdeg} Consequently, by $\nth q_k=\sum_{H\in\comgraphs_k} \nth\mu_H$, with the finite set $\comgraphs_k$ from \eqref{eq:def_Hk}, there exists a random variable $D^\msr$ with pmf $\bitq$ such that $\nth\bitq\to\bitq$ pointwise as $n\to\infty$, or equivalently,
\beq D_n^\msr\toindis D^\msr. \eeq
\end{enumeraten}
\item\label{cond:mean_rdeg} $\E[D^\msr]$ is finite, and as $n\to\infty$,
\beq \E[D_n^\msr] \to \E[D^\msr]. \eeq
\end{enumerateA}
\end{assumption}

\begin{remark}[Consequences of \cref{asmp:convergence}]
\label{rem:asmp_consequences}
\normalfont
We note the following:
\begin{enumeratei}
\item\label{cond:partition_ratio} By its definition in \eqref{eq:def_halfedges}, $\he_n = N_n \E[D_n^\msl] = M_n \E[D_n^\msr]$. By \cref{asmp:convergence} (\ref{cond:mean_ldeg},\ref{cond:mean_rdeg}),
\beq\label{eq:mnratio} M_n / N_n = \E[D_n^\msl] / \E[D_n^\msr] \to \E[D^\msl] / \E[D^\msr] =: \gamma \in \R^+. \eeq
\item\label{cond:limit_cdeg} Since $\nth{\bm\ro}$ (see \eqref{eq:def_rho}) can be obtained from $\nth{\bm\mu}$, \cref{asmp:convergence} \cref{cond:limit_com} also implies that there exists a random variable $(D^\msc, \Lambda^\msc)$ with pmf $\bm\ro$ such that $\nth{\bm\ro}\to\bm\ro$ pointwise as $n\to\infty$, or equivalently, $(D_n^\msc,\Lambda_n^\msc) \toindis (D^\msc,\Lambda^\msc)$.
\item\label{cond:dmax} \Cref{asmp:convergence} (\ref{cond:limit_ldeg},\ref{cond:mean_ldeg}) imply\footnote{This implication is proved for a similar setting in \cite[Exercise 6.3]{rvHofRGCNvol1}.} that $\dmax{\msl} := \max_{v\in\leftpar} d_v^\msl = o(\he_n)$, and similarly, conditions (\ref{cond:limit_rdeg},\ref{cond:mean_rdeg}) imply that $\dmax{\msr} := \max_{a\in\rightpar} d_a^\msr = o(\he_n)$. 
\end{enumeratei}
\end{remark}

\begin{remark}[Random parameters]
\label{rem:randomparam} 
The results in \cref{sss:RIGC_results} below remain valid when the sequence of parameters $(\bitd^\msl,\comvect)$ (resp., $(\bitd^\msl,\bitd^\msr)$) is random itself. In this case, we require that $N_n\to\infty$ and $M_n\to\infty$ almost surely, and we replace \cref{asmp:convergence} {(\ref{cond:limit_ldeg}-\ref{cond:mean_rdeg})} (resp., \cref{asmp:convergence} {(\ref{cond:limit_ldeg},\ref{cond:mean_ldeg},\ref{cond:limit_rdeg},\ref{cond:mean_rdeg})}) by the conditions 
$\nth{\bitp}\toinp\bitp$ pointwise, $\E[D_n^\msl \cond \bitd^\msl] \toinp \E[D^\msl]$, $\nth{\bm\mu}\toinp\bm\mu$ pointwise (resp., $\nth{\bitq}\toinp\bitq$) and $\E[D_n^\msr \cond \bitd^\msr] \toinp \E[D^\msr]$, where we assume the limiting pmfs $\bitp$ and $\bm\mu$ (resp., $\bitq$) to be \emph{deterministic}. For a similar setting in the configuration model, see \cite[Remark 7.9]{rvHofRGCNvol1}, where this is spelled out in more detail.
\end{remark}

Note that analogously to \cref{rem:asmp_consequences} \cref{cond:partition_ratio}, under
the conditions of \cref{rem:randomparam}, $M_n/N_n \toinp \gamma$.

\subsection{Results}\label{sss:RIGC_results}
In this section, we state our results on local properties of the $\RIGC$. The main result is the local weak convergence of the $\RIGC$ (defined shortly), which is equivalent to the convergence of subgraph counts (neighborhood counts). Local weak convergence also implies the convergence of degrees and local clustering, and provides some insight into the overlapping structure of communities. We use the following notions throughout this section. Recall that $V_n^\msl\sim\Unif[\leftpar]$ denotes an $\msl$-vertex chosen uar, and $\P(\;\cdot \cond \bmatch_n)$ denotes conditional probability wrt $\bmatch_n$. Let $\E_{V_n^\msl}[\;\cdot \cond \bmatch_n]$ denote the corresponding conditional expectation, that is, empirical averages for a given $\bmatch_n$.

\smallskip
\paragraph*{\textbf{Local weak convergence}} First, we give the brief definition of local weak convergence to state our results, and give a much more detailed introduction to the concept in \cref{ss:LWC_def}.

\begin{definition}[Rooted graph, rooted isomorphism and $r$-neighborhood]
\label{def:rooted_isom_rball}
\begin{enumeratei}
\item\label{stm:def:rootedgraph} We call a pair \allowbreak$(G,\groot)$ a \emph{rooted graph} if $G$ is a locally finite, connected graph and $\groot$ is a distinguished vertex of $G$.
\item\label{stm:def:rootedisom} We say that the rooted graphs $(G_1,\groot_1) \isom (G_2,\groot_2)$, are \emph{rooted isomorphic}, if there exists a graph-isomorphism between $G_1$ and $G_2$ that maps $\groot_1$ to $\groot_2$.
\item\label{stm:def:rball} For some $r\in\N$, we define $B_r(G,\groot)$, the (closed) $r$-ball around $\groot$ in $G$ or \emph{$r$-neighborhood of $\groot$ in $G$}, as the subgraph of $G$ spanned by all vertices of graph distance at most $r$ from $\groot$. We think of $B_r(G,\groot)$ as a rooted graph with root $\groot$.
\end{enumeratei}
\end{definition}

\begin{definition}[Local weak convergence in probability]
\label{def:LWCP}
Let $(G_n)_{n\in\N}$ with size $\abs{G_n} \toinp \infty$ be a sequence of random graphs,\footnote{By 
$\abs{G_n} \toinp \infty$, we mean that for all $K\in\R^+$, $\P(\abs{G_n} \geq K) \to 1$ as $n\to\infty$.} and let $U_n \cond G_n \distr\Unif[\vertices(G_n)]$. Let $(\CR,\groot)$ denote a random element (with arbitrary distribution) of the set of rooted graphs, which we call a \emph{random rooted graph}. We say that \emph{$(G_n,U_n)$ converges to $(\CR,\groot)$ in probability in the local weak convergence sense}, and denote $(G_n,U_n) \toLWCP (\CR,\groot)$, if for any fixed rooted graph $(G,\groot)$ and $r\in\N$,
\beq \bal \P\bigl( B_r(G_n,U_n) \isom B_r(G,\groot) \bcond G_n \bigr) &:= 
\frac{1}{\abs{G_n}} \sum_{u\in\vertices(G_n)} \ind_{\{ B_r(G_n,u) \isom B_r(G,\groot) \}} \\
&\toinp \P\bigl( B_r(\CR,\groot) \isom B_r(G,\groot) \bigr). 
\eal \eeq
We also say that $(\CR,\groot)$ is the \emph{local weak limit in probability} of $(G_n,U_n)$.
\end{definition}

We can now state our first main result on the local weak convergence of the $\RIGC$ model:

\begin{theorem}[Local weak convergence of the $\RIGC$]\label{thm:LWC_RIGC}
Consider $\RIGC_n = \RIGC(\bitd^\msl,\comvect)$ under \cref{asmp:convergence}. Then, with $V_n^\msl\distr\Unif[\leftpar]$, as $n\to\infty$,
\beq\label{eq:LWC_RIGC} \bigl(\RIGC_n,V_n^\msl\bigr) \toLWCP (\rCP,\groot), \eeq
where $(\rCP,\groot)$ is a random rooted graph with distribution specified in \cref{ss:construction_LWClim_RIGC}.
\end{theorem}
The proof of \cref{thm:LWC_RIGC} is completed in \cref{ss:RIGC_LWC_proof}. The construction of the local weak limit relies on the study of the underlying $\BCM$ that we carry out in \cref{s:BCM}, this is why we postpone it. We remark that the limit $(\rCP,o)$ is \emph{not a tree} (under mild conditions on $\bm\mu$ from \cref{asmp:convergence} \cref{cond:limit_com}), however it heavily relies on the locally tree-like structure of the underlying $\BCM$. 
In the following, we present some corollaries of \cref{thm:LWC_RIGC}.

\smallskip
\paragraph*{\textbf{Degrees}}
Recall \eqref{eq:def_Dnproj} and \eqref{eq:def_Fnproj}. 
We define the random variable $D^\msp$ and its distribution function
\beq\label{eq:def_Dproj} D^\msp \eqindis \sum_{i=1}^{D^\msl} D_{(i)}^\msc, \qquad F^\msp(x) := \P\bigl(D^\msp \leq x\bigr), \eeq
with $D^\msl$ from \cref{asmp:convergence} \cref{cond:limit_ldeg}, and $D_{(i)}^\msc$ are independent, identically distributed (iid) copies of $D^\msc$ from \cref{rem:asmp_consequences} \cref{cond:limit_cdeg}.

\begin{corollary}[Degrees in the $\RIGC$]\label{cor:deg} Consider $\RIGC(\bitd^\msl,\comvect)$ under the conditions of \cref{thm:LWC_RIGC}. Then, as $n\to\infty$,
\beq\label{eq:fn_degree_conv} \bigl\lVert F_n^\msp - F^\msp \bigr\rVert_\infty
= \sup_{x\in\R} \, \bigl\lvert F_n^\msp(x) - F^\msp(x) \bigr\rvert \toinp 0, \eeq
and consequently,
\beq\label{eq:degreeconv} D_n^\msp \toindis D^\msp. \eeq
\end{corollary}
In \cref{ss:deg_clust}, we prove \cref{cor:deg} using \cref{thm:LWC_RIGC}. However, \cref{cor:deg} can alternatively be proved independently through a first and second moment method under weaker conditions. In particular, \cref{asmp:convergence} \cref{cond:limit_com} can be replaced by $D_n^\msc\toindis D^\msc$. 
Let us also note that while \eqref{eq:degreeconv} is more intuitive, \eqref{eq:fn_degree_conv} is a stronger statement. Indeed, \eqref{eq:fn_degree_conv} implies that the \emph{random empirical degree distribution}, i.e., 
the observed degree sequence, 
is close to its theoretical limit whp.

\smallskip
\paragraph*{\textbf{Clustering}}
We proceed by studying the clustering in the $\RIGC$, in particular focusing on local clustering. For an arbitrary individual $v\in\leftpar$, let $\triproj(v)$ denote the (random) number of triangles\footnote{We also include degenerate triangles, where one or more vertices are the same, and count triangles with multiplicity, i.e., all possible ways we can choose the three edges.} that $v$ is part of in the $\RIGC$. 
We define the local clustering at $v$ as
\beq\label{eq:def_clustering} \clustering(v) := \frac{\triproj(v)}{\binom{\pdeg(v)}{2}}, \eeq
with the convention that $\clustering(v) := 0$ whenever $\pdeg(v) < 2$. 
Define the empirical local clustering coefficient $\zeta_n := \clustering(V_n^\msl)$ and 
denote its random empirical cdf by
\beq\label{eq:def_clus_empirical} F_n^\zeta(x) = F_n^\zeta(x;\bmatch_n)
:= \frac{1}{N_n} \sum_{v\in\leftpar} \ind_{ \{\clustering(v) \leq x \} } 
= \P\bigl( \zeta_n \leq x \bcond \bmatch_n \bigr). \eeq
We introduce
\beq\label{eq:def_zeta} \zeta \eqindis \biggl( \sum_{i=1}^{D^\msl} \tricomrand{(i)} \biggr)
\Big/ \binom{\sum_{i=1}^{D^\msl} D_{(i)}^\msc}{2}, 
\qquad F^\zeta(x) := \P(\zeta\leq x), \eeq
where $(D_{(i)}^\msc,\tricomrand{(i)})$ are iid copies of the random vector $(D^\msc,\tricomrand{})$ from \cref{rem:asmp_consequences} \cref{cond:limit_cdeg} and are independent of $D^\msl$ (see \cref{asmp:convergence} \cref{cond:limit_ldeg}).

\begin{corollary}[Local clustering in the $\RIGC$]\label{cor:clustering} Consider $\RIGC(\bitd^\msl,\comvect)$ under the conditions of \cref{thm:LWC_RIGC}. Then, as $n\to\infty$,
\beq\label{eq:clustering_conv} \norm{F_n^\zeta - F^\zeta}_\infty 
=  \sup_{x\in\R} \;\bigl\lvert F_n^\zeta(x) - F^\zeta(x) \bigr\rvert \toinp 0. \eeq
In particular, $\zeta_n \toindis \zeta$ and thus the average local clustering converges:
\beq\label{eq:clusteringAVG} \E\bigl[\zeta_n\bigr] \to \E[\zeta]. \eeq
\end{corollary}
We prove \cref{cor:clustering} as a corollary of \cref{thm:LWC_RIGC} in \cref{ss:deg_clust}. However, in fact \cref{cor:clustering} still holds if we replace \cref{asmp:convergence} \cref{cond:limit_com} by the conditions \cref{asmp:convergence} \cref{cond:limit_rdeg} and \cref{rem:asmp_consequences} \cref{cond:limit_cdeg}. The intuition behind \cref{cor:clustering} is that triangles typically arise within one community, that is, triangles containing edges from different communities make a negligible contribution as the model size grows. This is due to the ``locally tree-like'' structure of the underlying $\BCM$ (see \cref{thm:LWC_BCM} below). We remark that under our general conditions, we cannot establish that the local clustering scales inversely with the degree (as in e.g.\ \cite{Bloz13,New2003configRIG}), however, the inverse degree serves as an upper bound for the clustering. In the following, we establish when the model has positive asymptotic clustering.

\begin{corollary}[Condition for positive asymptotic clustering]
\label{cor:clust_pos}
Under the conditions of Corollary \ref{cor:clustering}, the asymptotic average clustering $\E[\zeta]$ is positive if and only if $\P\bigl(\tricomrand{}\geq1\bigr)>0$, with $\tricomrand{}$ 
from \cref{rem:asmp_consequences} \cref{cond:limit_cdeg}. 
\end{corollary}
\begin{proof}[Proof of \cref{cor:clust_pos}]
Note that $\P\bigl(\tricomrand{}\geq1\bigr)>0$ happens exactly when the assigned communities are not $\bm\mu$-almost surely triangle-free with $\bm\mu$ from \cref{asmp:convergence} \cref{cond:limit_com}, i.e., $\mu_H>0$ for at least one $H\in\comgraphs$ such that $H$ contains at least one triangle. Clearly, this is a necessary condition, but also sufficient, as it implies that any vertex has a positive probability to be part of a triangle and have bounded degree at the same time.
\end{proof}

Another measure of clustering is the so-called global clustering coefficient, defined as three times the total number of triangles in the graph divided by the total number of connected triples (paths of length $2$, often called ``wedges''), formally,
\beq\label{eq:def_clust_glob} \clustering_{\mathrm{glob}} 
:= \frac{3 \triproj_{\mathrm{total}}}{\sum_{v\in\leftpar} \binom{\pdeg(v)}{2}} 
= \frac{\sum_{v\in\leftpar} \triproj(v)}{\sum_{v\in\leftpar} \binom{\pdeg(v)}{2}}. \eeq
Note the relation with the local clustering coefficient defined in \eqref{eq:def_clustering} as the ratio of $\triproj(v)$ and $\binom{\pdeg(v)}{2}$; in \eqref{eq:def_clust_glob}, we instead consider the ratio of the \emph{sum} over all individuals of these quantities. 
Also note that we can think of the global clustering coefficient as the \emph{ratio of the averages} of $\triproj(v)$ and $\binom{\pdeg(v)}{2}$:
\beq \clustering_{\mathrm{glob}} = \frac{ \frac{1}{N_n} \sum_{v\in`\leftpar} \triproj(v)}{ \frac{1}{N_n} \sum_{v\in\leftpar} \binom{\pdeg(v)}{2}}, \eeq
while the average local clustering is given by the \emph{average of the ratios} of the same quantities:
\beq \E\bigl[ \clustering(V_n^\msl) \bcond \bmatch_n \bigr] 
= \frac{1}{N_n} \sum_{v\in\leftpar} \clustering(v) 
= \frac{1}{N_n} \sum_{v\in\leftpar} \frac{\triproj(v)}{\binom{\pdeg(v)}{2}}. \eeq
While the global clustering coefficient and average local clustering coefficient grasp similar concepts, their behaviors are different. In this paper we omit the formal study the global clustering coefficient. However, we believe that its convergence requires the stronger condition of $\E\bigl[ (\pdeg(V_n^\msl))^2 \bcond \bmatch_n \bigr] = \E\bigl[ (D_n^\msp)^2 \bcond \bmatch_n \bigr] \toinp \E\bigl[ (D^\msp)^2 \bigr]$, which can be reduced to $\E\bigl[ (D_n^\msl)^2 \bigr] \to \E\bigl[ (D^\msl)^2 \bigr]$ and $\E\bigl[ (D_n^\msc)^2 \bigr] \to \E\bigl[ (D^\msc)^2 \bigr]$. We believe that under these conditions, $\clustering_{\mathrm{glob}}$ converges in probability to the \emph{ratio of expectations} of the numerator and denominator of $\zeta$ in \eqref{eq:def_zeta}, i.e., 
\beq \clustering_{\mathrm{glob}} \toinp \E\biggl[ \sum_{i=1}^{D^\msl} \tricomrand{(i)} \biggr] 
\Big/ \E\biggl[ \binom{\sum_{i=1}^{D^\msl} D_{(i)}^\msc}{2} \biggr], \eeq
which in general is different from the limiting average local clustering $\E[\zeta]$.

\smallskip
\paragraph*{\textbf{The overlapping structure}}
Next, we turn our attention to the overlapping structure of the groups, which is one of the main motivators for the $\RIGC$ model. By an overlap, we mean two (or more) groups having one (or more) individual in common. From this definition, it is clear that the internal structure of the groups do not play a role in the overlapping structure, thus the following discussion applies to the $\RIG$ model as well. 
By the construction of the model, i.e., including individuals in several communities, it is clear that overlaps are present. First, we will study the number of overlaps, and later the typical size of the overlaps as well. Let us introduce some notation. 
For $v\in\leftpar$ and $a\in\rightpar$, we say that $v$ is part of $\com_a$ and denote $v \comrole \com_a$ if $v \comrole j$ for some $j\in\com_a$. 
Let us denote 
the size of overlap between $a,b\in\rightpar,a\neq b$ by
\beq\label{eq:def_intersection_size} \intersection(a,b) := 
\sum_{v\in\leftpar} \ind_{\{v \comrole \com_a\} \cap \{v \comrole \com_b\}}. \eeq
We define the set of communities overlapping with community $a$ as
\beq\label{eq:def_overlap_neighbor} \neighbors(a) := \{b\in\rightpar :\, b\neq a, \intersection(a,b) \geq 1 \}. \eeq
For $k\in\Z^+$, we introduce the set of unordered pairs of (at least) $k$-fold overlapping groups:
\beq\label{eq:def_overlapset}
\overlapset_k = \nn \overlapset_k := \bigl\{ \{a,b\}:\, a,b\in\rightpar, a\neq b, \intersection(a,b) \geq k \bigr\}. \eeq
Note that $\overlapset_k \supseteq \overlapset_{k+1}$ for all $k\in\Z^+$ and $\overlapset_1$ contains all overlapping pairs, regardless of the size of overlap they share. Recall that $V_n^\msr\sim\Unif[\rightpar]$, 
and further recall that $\P(\;\cdot \cond \bmatch_n)$ denotes the conditional probability wrt $\bmatch_n$ and $\E[\;\cdot \cond \bmatch_n]$ denotes the corresponding conditional expectation. We can now state our result on the number of overlaps: 

\begin{proposition}[Number of overlaps]\label{prop:number_of_overlaps}
Consider $\RIGC(\bitd^\msl,\comvect)$ under \cref{asmp:convergence}. In addition, assume that, as $n\to\infty$,
\beq\label{cond:secondmom} \E\bigl[(D_n^\msl)^2\bigr] \to \E\bigl[(D^\msl)^2\bigr] < \infty. \eeq
Then, as $n\to\infty$, the average number of communities overlapping with a ``typical'' one converges: 
\beq\label{eq:number_of_overlaps} 
\frac{ 2 \abs{\overlapset_1}}{M_n} 
= \E\bigl[ \abs{\neighbors(V_n^\msr)} \bcond \bmatch_n \bigr] 
\toinp \E[D^\msr]\E[\wit D^\msl]. \eeq
\end{proposition}
Note that \eqref{cond:secondmom} ensures that $\E[\wit D^\msl]<\infty$, thus the rhs of \eqref{eq:number_of_overlaps} is finite. We prove \cref{prop:number_of_overlaps} in \cref{sss:overlaps_number} using local weak convergence. Intuitively, \eqref{eq:number_of_overlaps} asserts that a typical community $V_n^\msr$ overlaps with constantly many others, and thus the number of overlapping pairs of groups is linear in the total number of groups.

Next, 
we assert that the ``typical'' overlap size is $1$, which we call the \emph{single-overlap property}. There are several ways to interpret what the ``typical overlap'' means, leading to slightly different statements, as follows:

\begin{theorem}[Single-overlap property]\label{thm:overlaps}
Consider $\RIGC(\bitd^\msl,\comvect)$ under \cref{asmp:convergence}, then the single-overlap property holds, in the following ways:
\begin{enumeratei}
\item\label{perspective:vx} \emph{Vertex perspective.} For a uniform individual $V_n^\msl \distr \Unif[\leftpar]$, the communities that $V_n^\msl$ is part of whp \emph{only} overlap at $V_n^\msl$. Formally, 
as $n\to\infty$,
\beq\label{eq:typical_overlap_vx} \P\bigl( \exists \{a,b\} \in\overlapset_2 :\, V_n^\msl \comrole \com_a, V_n^\msl \comrole \com_b \bcond \bmatch_n \bigr) \toinp 0. \eeq
\item\label{perspective:group} \emph{Group perspective.} For a uniform community $V_n^\msr \distr \Unif[\rightpar]$, the communities that $V_n^\msr$ overlaps with whp only share a \emph{single} individual with $V_n^\msr$. Formally, 
as $n\to\infty$,
\beq\label{eq:typical_overlap_group} \P\bigl( \exists b\in\neighbors(V_n^\msr):\, \intersection(V_n^\msr,b) \geq 2 \bcond \bmatch_n \bigr) \toinp 0. \eeq
\item\label{perspective:global} \emph{Global perspective.} Assume additionally condition \eqref{cond:secondmom} and let $\{A_n,B_n\}\sim\Unif[\overlapset_1]$ denote a pair of communities chosen uar among all distinct pairs of overlapping communities. Then, whp their overlap is one individual. Formally, 
as $n\to\infty$,
\beq\label{eq:typical_overlap_global} 
\P\bigl( \intersection(A_n,B_n) \geq 2 \cond \bmatch_n\bigr) 
= \abs{ \overlapset_2 } \big/ \abs{ \overlapset_1 }
\toinp 0. \eeq
\end{enumeratei}
\end{theorem}

We complete the proof in \cref{ss:overlaps} but discuss the statement now. The extra second moment condition \eqref{cond:secondmom} in \eqref{perspective:global} suggests a substantial difference from (\ref{perspective:vx}-\ref{perspective:group}). Indeed, (\ref{perspective:vx}-\ref{perspective:group}) establish local properties and follow directly from local weak convergence, which is not true for \eqref{perspective:global}. The difficulty is in relating the choice of the \emph{pair} $(A_n,B_n)\sim\Unif\bigl[\overlapset_1 \bigr]$ to the choice of a \emph{single} uniform vertex (and further choices in its neighborhood). This problem is nontrivial and further regularity is required. Also note that \cref{prop:number_of_overlaps} requires the same second moment condition for $\E[\wit D^\msl]$ to be finite, which is used in identifying the asymptotics for $\abs{ \overlapset_1 }$, that is the denominator in \eqref{eq:typical_overlap_global}. 
In the underlying $\BCM$ (see \cref{def:bcm}), $\abs{ \overlapset_1 }$ is the number of pairs of groups that are at graph distance $2$; however, the fluctuations of this quantity is an open problem in the case when the variance of the degrees diverges. 

\smallskip
\paragraph*{\textbf{Relation with the ``passive'' random intersection graph}} The overlapping structure may be represented as a graph on $\rightpar$ by adding an edge between a pair of groups for each individual they are both connected to. This leads to a ``dual'' random intersection graph, defined on the communities, that is sometimes referred to as the ``passive model'' in the literature \cite{GodeJaw2003}. Then the size of the overlaps $\intersection(a,b)$ and the number of overlapping pairs $\abs{\overlapset_1}$ can be seen as the edge multiplicities and total number of edges in the passive model, respectively; in particular, $2\abs{ \overlapset_1 }/M_n$ gives the average degree. Note that in this regard, applying \cref{thm:overlaps} with the roles of lhs and rhs reversed (also replacing \eqref{cond:secondmom} by $\E[(D_n^\msr)^2]\to\E[(D^\msr)^2]<\infty$ in \cref{thm:overlaps} \eqref{perspective:global}) provides some insight on the number of multi-edges in the ``active'' $\RIG$ (with complete graph communities) on the $\msl$-vertices. In turn, this provides an upper bound for the number of multi-edges in the $\RIGC$ model as well, but obtaining a lower bound is nontrivial.\footnote{Since not all pairs of community roles are connected by an edge, two individuals being together in several communities does not necessarily mean that they are connected by multiple edges, and finer properties of the measure $\bm\mu$ (see \cref{asmp:convergence} \cref{cond:limit_com}) come into play. It further complicates the situation that if we condition on having several communities that both individuals are part of, we also introduce a bias to the $\msb$-degrees involved.}

\smallskip
\paragraph*{\textbf{The local weak convergence of the underlying BCM}}
Recall the notion of local weak convergence in probability from \cref{def:LWCP}. 
\begin{theorem}[Local weak convergence of the $\BCM$]
\label{thm:LWC_BCM}
Consider $\BCM_n = \BCM(\bitd^\msl,\bitd^\msr)$ under Assumption \ref{asmp:convergence} (\ref{cond:limit_ldeg},\ref{cond:mean_ldeg},\ref{cond:limit_rdeg},\ref{cond:mean_rdeg}). Let $V_n^\msb = V_n^{\msl+\msr} \sim \Unif[\leftpar\cup\rightpar]$. Then, as $n\to\infty$,
\begin{equation}
\label{eq:BCM_LWC} \bigl(\BCM_n,V_n^\msb\bigr) \toLWCP (\BP_\mss,\otroot),
\end{equation}
where $(\BP_\mss,\otroot)$ is a mixture of two branching processes defined in \cref{ss:construction_LWClim_BCM}.
\end{theorem}

We prove \cref{thm:LWC_BCM} in \cref{ss:LWC_BCM}. Note that in particular, \cref{thm:LWC_BCM} asserts that the bipartite configuration model is \emph{locally tree-like}, a property possessed by several random graph models such as the classical configuration model or the Erd\H{o}s-R\'enyi random graph model. We also remark that while \cref{thm:LWC_BCM} and its proof are instrumental to our results on the $\RIGC$, it is also of independent interest.

\subsection{Discussion on the random intersection graph with communities}\label{ss:RIGC_discussion}
In this section\hyphenation{sec-tion}, we discuss the relation of our model to other network models and shed light on possible applications and their limitations. 

\smallskip
\paragraph*{\textbf{Parameter choices}} Working with prescribed parameters provides a wide range of applicability. As \cref{cor:deg,cor:clustering} suggest, the degree distribution and clustering of the RIGC model are tunable to match our observations of real-world networks, however the choice of $\bitd^\msl$ and $\comvect$ is hard to infer. One way of obtaining these parameters explicitly is through community-detection algorithms \cite{Fort10,FortHric16}. For theoretical research, one may be interested in generating the input parameters randomly, of which we give two examples. A simple idea is using iid random variables with distribution $D^\msl$ and $\com$ to generate the sequences $\bitd^\msl$ and $\comvect$, respectively. However, the parameters must satisfy \eqref{eq:def_halfedges}. If both $\Var(D^\msl)<\infty, \Var(D^\msr)<\infty$, we can use the algorithm proposed by Chen and Olvera-Cravioto\footnote{While their algorithm was designed for the directed configuration model, it is straightforwardly applicable to the $\BCM$.} in \cite{ChenOlcr2013} to generate the sequences $\bitd^\msl$, $\comvect$ in such a way that the sum of the $\msl$- and $\msr$-degrees are equal, meanwhile the entries are asymptotically independent.

Our second example is generating a matching pair of $\bitd^\msl$ and $\bitd^\msr$ in a \emph{dependent} way through a bipartite version of the generalized random graph \cite{BriDeiLof06}, or a Norros-Reittu model \cite{NorRei06}. Once $\bitd^\msr$ is given, we have to generate $\comvect$ in a compatible way, i.e., such that the community sizes are indeed the $\msr$-degrees. \Cref{asmp:convergence} \cref{cond:limit_com} implies that there exists a family of conditional measures
\beq\label{eq:def_mu_cond} \mu_{H \vert k} = \P\bigl(\com_a\isom H \cond d_a^\msr = \abs{\com_a} = k\bigr),
\quad \bm\mu_{\cdot\vert k} = (\mu_{H \vert k})_{H\in\comgraphs_k},
\quad (\bm\mu_{\cdot \vert k})_{k\in\Z^+}, \eeq
that describe the conditional distribution of community graphs for each given community size. In fact $\mu_{H \vert k} = \mu_H/q_k$, with $\bm\mu$ and $\bitq$ from \cref{asmp:convergence} \cref{cond:limit_com} and \cref{cond:limit_rdeg}, respectively. (We note that due to this relation, under \cref{asmp:convergence} \cref{cond:limit_rdeg}, the implication is reversible, i.e., the existence of $(\bm\mu_{\cdot \vert k})_{k\in\Z^+}$ implies \cref{asmp:convergence} \cref{cond:limit_com}.) Thus we can generate each $\com_a$ according to the measure $\bm\mu_{\cdot \vert \rdeg(a)}$, independently of each other.

\smallskip
\paragraph*{\textbf{Overlaps}} The motivation behind random intersection graphs is to generate \emph{overlapping} communities, which is clearly satisfied by \cref{prop:number_of_overlaps}. 
However, \cref{thm:overlaps} asserts the single-overlap property of the $\RIGC$ and $\RIG$ graphs, which limits the applicability of these models. For example, they may not be a good fit for scientific collaboration networks, where the same authors often collaborate on several papers and with several other collaborators. However, the $\RIGC$ may be used for social networks when the different groups of the same person tend to be separate: their family members, their colleagues, their sports club friends, etc., typically do not know each other. 

On the other hand, the single-overlap property may be used to optimize community detection; for example, consider the C-finder algorithm based on the clique percolation method \cite{Pal05clique,Pal05Nature}, that we explain briefly. A $k$-clique in a graph is a complete subgraph on $k$ vertices, and we call two $k$-cliques adjacent if they share $k-1$ vertices. A component in $k$-clique percolation is a maximal set of vertices that are connected through a chain of adjacent $k$-cliques. We remark that such components may overlap, as long as the intersection does not contain a $(k-1)$-clique; the simplest case is when the overlap has less than $k-1$ vertices. The C-finder algorithm outputs such components as possibly overlapping communities in the network. Now suppose each community of the $\RIGC$ is $3$-clique connected, i.e., built up from edge-adjacent triangles. Due to the single-overlap property of the $\RIGC$, a typical community will be a component of $3$-clique percolation by itself, i.e., no other communities will be $3$-clique adjacent to it, allowing detection with great accuracy. Thus, such an $\RIGC$ works really well in conjunction with the C-finder algorithm, either as first generating the $\RIGC$ and then detecting its communities, or running C-finder on the dataset for which one wishes to use the $\RIGC$ as a null model.

We believe that we can also use the clique percolation approach to make the $\RIGC$ a better fit than the traditional $\RIG$ for collaboration networks, in particular for scientific collaboration networks of authors and the papers they collaborate on. Rather than considering each paper as its own community, which leads to cliques with a typical overlap size larger than one, we can instead merge cliques with more than a single overlap into one community, which, in fact, uses the components of clique percolation as communities. Then we can think of each community as the collaboration network of a subgroup of authors who often collaborate with one another, and the collaboration network as a network with hierarchical structure.

\smallskip
\paragraph*{\textbf{Multigraphs}} The usual criticism that the configuration model receives is that it may produce a multigraph, and this happens whp in case the degrees have infinite (asymptotic) variance \cite[Chapter 7]{rvHofRGCNvol1}. As the $\RIGC$ uses a bipartite configuration model in its construction, we are bound to deal with multigraphs on the level of group memberships, and possibly on the level of the projection as well. One classical remedy is to condition the graph on simplicity, it is however outside the scope of this paper to study this conditional measure (which we conjecture is non-uniform) or to study whether the simplicity probability remains bounded away from $0$ as the graph size grows. Another classical approach applied to the configuration model is erasure, 
and analogously, we can define the erased $\RIGC$\footnote{Note that using the erased $\BCM$ in the construction does \emph{not} ensure that the resulting $\RIGC$ is a simple graph, multi-edges may still arise due to two individuals being part of two (or more) communities together.} by removing self-loops and collapsing multi-edges into a simple edge, i.e., redefining the edge multiplicites from \eqref{eq:def_edge_multiplicities} as $X_{v,v}'=0$ and $X_{v,w}'=\ind_{ \{X_{v,w}\geq1\} }$. In this paper, we choose to study the $\RIGC$ as a multigraph, and argue that we do not see the effect of this in the local behavior; indeed, subject to \cref{thm:LWC_RIGC}, the local weak limit of the $\RIGC$ is simple (a distribution on rooted simple graphs). This means that a typical individual will whp not see a self-loop or multi-edge in its finite neighborhood. Based on this observation, our results extend to the erased $\RIGC$ without any modification. 

\section{Preliminaries: marked graphs, ordered trees and local weak convergence}
\label{s:prelim}
In order to prove our results, we first introduce the concepts that we rely on in our proof, the most central one being \emph{local weak convergence} (LWC), a notion of convergence for sparse graph sequences. The usefulness of LWC comes from the fact that numerous properties of the finite graph(s) can be determined or approximated based on the limiting object alone \cite{BerBorChaSab14,ChenLitOlcr17}. As its name suggests, LWC describes the graph from a \emph{local} point of view; indeed, in \cref{def:LWCP}, we have defined local weak convergence in probability in terms of convergence of frequencies of graph neighborhoods. In \cref{ss:LWC_def}, we cover some of the theory behind the notion of LWC, in fact in a more general setting of \emph{marked graphs}, which are defined in \cref{ss:LWC_def} as well. 
The theory of LWC presented is partially based on \cite{AldSte04,BenLyoSch15LWC,BenSch01} and \cite[Section 1.4]{rvHofRGCNvol2}, but generalized and tailored to our needs. 
In \cref{ss:neighborhoods_trees}, we introduce some more practical tools for the proofs.

\subsection{Local weak convergence of marked graphs}
\label{ss:LWC_def}
In this section, we introduce marked graphs and the theory of LWC for deterministic and random graphs.

\smallskip
\paragraph*{\textbf{Marked graphs}} \emph{Marks} provide a general framework for indicating additional information on the edges and/or vertices of a (multi)graph, such as edge weights, edge directions, graph coloring, etc. In our case, we use marks to include edge labels of the underlying $\BCM$, as well as indicate the community graphs assigned to each $\msr$-vertex. We formally define marked graphs below.

Let $\graphs$ denote the set of all locally finite (multi)graphs on a countable (finite or countably infinite) vertex set. 
Let the set of marks $\markset$ be an arbitrary countable set that contains the special symbol $\nomark$ which is to be interpreted as ``no mark''. A marked graph is a pair $(G,\markfunc)$, where $\markfunc$ is the mark function that maps elements of $G$ into $\markset$, in particular, for $v\in\vertices(G)$, $\markfunc(v) \in \markset$, and for $e\in\edges(G)$, $\markfunc(e) \in \markset^2$. 
It is common to associate two marks to each edge, with one mark associated to each endpoint, which is often interpreted as separate marks associated to the two directions of a bi-directed edge. Since we work with the bipartite configuration model, it is more useful to think of the marks being associated to the half-edges that form the edge. 
We denote the set of graphs with marks from the mark set $\markset$ by $\graphs(\markset)$. 

We remark that any graph in $\graphs$, that we may refer to as \emph{unmarked graphs} for clarity, can be turned into a marked graph by assigning the ``no mark'' symbol $\nomark$ to each vertex and half-edge; thus results and definitions formulated for marked graphs apply straightforwardly to (unmarked) graphs.

\smallskip
\paragraph*{\textbf{Rooted marked graph, isomorphism and $r$-neighborhood}} 
We now generalize \cref{def:rooted_isom_rball} to marked graphs.
\begin{enumeratei}
\item Choose a vertex $\groot$ in a marked graph $(G,\markfunc)$ to be distinguished as the root; if $G$ is not connected, we restrict ourselves to the connected component of $\groot$, and denote the rooted marked graph by $(G,\markfunc,\groot)$. 
\end{enumeratei} 
Denote the set of rooted marked graphs by $\graphs_\groot(\markset)$. 
We call a random element of $\graphs_\groot(\markset)$ (with an \emph{arbitrary} joint distribution) a random rooted marked graph.
\begin{enumeratei}
\setcounter{enumi}{1}
\item We say that the rooted marked graphs $(G_1,\markfunc_1,\groot_1)$ and $(G_2,\markfunc_2,\groot_2)$ are isomorphic, and denote this by $(G_1,\markfunc_1,\groot_1)\isom(G_2,\markfunc_2,\groot_2)$, if there is a graph-isomorphism between them that also maps root to root and preserves marks. 
\item The (closed) ball $B_r(G,\markfunc,\groot)$ can be defined analogously to the unmarked graph ball (\cref{def:rooted_isom_rball} \cref{stm:def:rball}), by restricting the mark function to the subgraph as well.
\end{enumeratei}

\smallskip
\paragraph*{\textbf{Distance and topology}}
We are now ready to define a metric on $\graphs_\groot(\markset)$. 
For two elements $(G_1,\markfunc_1,\groot_1), (G_2,\markfunc_2,\groot_2)\allowbreak\in\graphs_\groot(\markset)$, we define the largest radius $r$ such that the $r$-neighborhoods of the roots are isomorphic:
\beq\label{eq:def_last_isom_level} \rmax:= \begin{cases}
-1 & \text{if $\markfunc_1(\groot_1)\neq\markfunc_2(\groot_2)$},\\
+\infty & \text{if $(G_1,\markfunc_1,\groot_1)\isom (G_2,\markfunc_2,\groot_2)$},\\
\sup\bigl\{ r\in\N: B_r(G_1,\markfunc_1,\groot_1)\isom B_r(G_2,\markfunc_2,\groot_2) \bigr\} & \text{otherwise}.
\end{cases}
\eeq
Then, we define the distance of the rooted marked graphs as
\beq\label{eq:def_dloc} \dloc\bigl( (G_1,\markfunc_1,\groot_1),(G_2,\markfunc_2,\groot_2)\bigr) := 2^{-\rmax} \in [0,2]. \eeq
The distance $\dloc$ is a metric on the \emph{isomorphism classes} of $\graphs_\groot(\markset)$, which turns this space into a Polish space, i.e., a complete, separable metric space (see \cite{AldSte04} or \cite[Section 1.4]{rvHofRGCNvol2}).

\smallskip
\paragraph*{\textbf{Local weak convergence of deterministic graphs}}
Let $(G_n,\markfunc_n)_{n\in\N}$, $(G_n,\markfunc_n)\in\graphs(\markset)$ be a sequence of (determinisitc) finite marked graphs such that $\abs{G_n}\to\infty$. For each $n$, let $U_n$ be a vertex of $G_n$ chosen uar, and consider the \emph{measures} defined by $(G_n,\markfunc_n,U_n)$ on $(\graphs_\groot(\markset),\dloc)$. We will define the local weak convergence of $(G_n,\markfunc_n)_{n\in\N}$ as the weak convergence of the above measures, which can be defined in the standard way. Let $(\R,\Eudist)$ denote the Polish space of the real numbers equipped with the Euclidean distance, and introduce the set of test functionals
\beq\label{eq:def_testfunctionals} \Phi = \{\varphi:\graphs_\groot(\markset)\to\R \,:\, \varphi \text{ is bounded and continuous} \}. \eeq
We remark that a special case of continuous functionals are those that only depend on a finite neighborhood of the root. 
We say that $(G_n,\markfunc_n,U_n)_{n\in\N}$ converges \emph{in the local weak convergence sense} to a (possibly random) element $(G,\markfunc,\groot) \in \graphs_\groot(\markset)$, denoted by $(G_n,\markfunc_n,U_n) \toLWC (G,\markfunc,\groot)$, if for all $\varphi\in\Phi$, as $n\to\infty$,
\beq\label{eq:def_LWCphi} \E\bigl[ \varphi(G_n,\markfunc_n,U_n) \bigr] 
\to \E\bigl[ \varphi(G,\markfunc,\groot) \bigr]. \eeq
This statement is equivalent (see e.g.\ \cite[Theorem 1.13]{rvHofRGCNvol2}) to the convergence of neighborhood counts, that is, the following statement is an equivalent definition of local weak convergence: for any $r\in\N$ and any fixed $(G',\markfunc',\groot')\in\graphs_\groot(\markset)$, as $n\to\infty$,
\beq\label{eq:def_LWC} 
\P\bigl( B_r(G_n,\markfunc_n,U_n) \isom B_r(G',\markfunc',\groot') \bigr)
\to \P\bigl(B_r(G,\markfunc,\groot) \isom B_r(G',\markfunc',\groot')\bigr). \eeq

\smallskip
\paragraph*{\textbf{Local weak convergence of random graphs}}
We now generalize \cref{def:LWCP} for marked graphs (simultaneously generalizing \eqref{eq:def_LWC} for random graphs). Let $(G_n,\markfunc_n)_{n\in\N}$, $(G_n,\markfunc_n)\in\graphs_\groot(\markset)$ be a sequence of (finite) random marked graphs (with an \emph{arbitrary} joint distribution) such that $\abs{G_n} \toinp \infty$, and let $U_n \cond (G_n,\markfunc_n) \distr \Unif[\vertices(G_n)]$ be a uniformly chosen vertex. Let $\P\bigl(\,\cdot \bcond (G_n,\markfunc_n)\bigr)$ denote conditional probability wrt the marked graph (i.e., the free variable is $U_n$). 
We say that $(G_n,\markfunc_n,U_n)_{n\in\N}$ converges \emph{in probability in the local weak sense} to a (possibly) random element $(G,\markfunc,\groot)\in\graphs_\groot(\markset)$, and denote $(G_n,\markfunc_n,U_n)\toLWCP(G,\markfunc,\groot)$, if the empirical neighborhood counts converge \emph{in probability}, i.e.,
for any fixed $r\in\N$ and fixed $(G',\markfunc',\groot')\in\graphs_\groot(\markset)$, as $n\to\infty$,
\beq\label{eq:def_LWCP} \bal
\P\bigl( B_r(G_n,\markfunc_n,U_n) \isom B_r(G',\markfunc',\groot') \bcond (G_n,\markfunc_n) \bigr) 
&:= \frac{1}{\abs{G_n}} \sum_{u\in\vertices(G_n)} \ind_{\{B_r(G_n,\markfunc_n,U_n) \isom B_r(G',\markfunc',\groot')\}} \\
&\toinp \P\bigl(B_r(G,\markfunc,\groot) \isom B_r(G',\markfunc',\groot')\bigr). \eal \eeq
We can also generalize \eqref{eq:def_LWCphi} for an equivalent definition (again, see e.g.\ \cite[Theorem 1.16]{rvHofRGCNvol2} for a proof of the equivalence) of LWC in probability. Let $\E\bigl[ \,\cdot \bcond (G_n,\markfunc_n) \bigr]$ denote conditional expectation corresponding to the conditional probability measure $\P\bigl(\,\cdot \bcond (G_n,\markfunc_n)\bigr)$. Then, $(G_n,\markfunc_n,U_n) \toLWCP (G,\markfunc,\groot)$ exactly when for all test functionals $\varphi \in \Phi$ (see \eqref{eq:def_testfunctionals}),
\beq\label{eq:def_LWCPphi} \E\bigl[ \varphi\bigl(B_r(G_n,\markfunc_n,U_n)\bigr) \bcond (G_n,\markfunc_n) \bigr] 
\toinp \E\bigl[ \varphi\bigl(B_r(G,\markfunc,\groot)\bigr) \bigr]. \eeq 

\smallskip
\paragraph*{\textbf{Extensions}}
We remark that there exist other notions of LWC for random graphs. \emph{Almost sure} local weak convergence can be defined by replacing the convergence in probability by almost sure convergence in \eqref{eq:def_LWCP}. 
Local weak convergence \emph{in distribution} is defined as
\beq \P\bigl( B_r(G_n,\markfunc_n,U_n) \isom B_r(G',\markfunc',\groot') \bigr) 
\to \P\bigl(B_r(G,\markfunc,\groot) \isom B_r(G',\markfunc',\groot')\bigr), \eeq
where we note the lack of conditioning on the lhs. 
In this paper, we use LWC in probability, as it is not too restrictive while being strong enough to imply \emph{asymptotic independence} of the neighborhoods of two uniformly chosen vertices.

\begin{remark}[Different root distributions]
\label{rem:LWC_nonunif}
In the classical definition of local weak convergence, $U_n$ is chosen \emph{uniformly} at random. However, in certain cases it is meaningful and interesting to study the convergence of subgraph counts around a vertex $W_n$ chosen according to a \emph{different} (non-uniform) distribution, for example size-biased by degree or chosen within a (large enough) subset of vertices. Our motivation is to restrict the choice of the root to one partition of the $\BCM$. 
This motivated us to emphasize the role of $U_n$ in the notation $(G_n,\markfunc_n,U_n) \toLWCP (G,\markfunc,\groot)$. With slight abuse of this notation, we shall write, for a random vertex $W_n$ with an arbitrary distribution on $\vertices(G_n)$, $(G_n,\markfunc_n,W_n) \toLWCP (G,\markfunc,\groot)$ to mean that the neighborhood counts around $W_n$ converge, i.e., for all $r\in\N$ and all $(G',\markfunc',\groot')\in\graphs_\groot(\markset)$, as $n\to\infty$,
\beq\label{eq:def_LWC_nonunif} \P\bigl( B_r(G_n,\markfunc_n,W_n) \isom B_r(G',\markfunc',\groot') \bcond (G_n,\markfunc_n) \bigr)
\toinp \P\bigl(B_r(G,\markfunc,\groot) \isom B_r(G',\markfunc',\groot')\bigr). \eeq
\end{remark}

\subsection{Practical tools: general neighborhoods and ordered trees}
\label{ss:neighborhoods_trees}
In order to calculate neighborhood\hyphenation{neigh-bor-hood} counts and prove local weak convergence of the $\RIGC$ and $\BCM$, we also rely on a few more practical tools and concepts that we introduce below.

\smallskip
\paragraph*{\textbf{General neighborhoods}}
It turns out graph balls are not the right way to look at neighborhoods when relating the $\RIGC$ to the underlying $\BCM$, which is the approach we take. Due to the arbitrary community graphs, graph distances are \emph{substantially} different in the $\RIGC$ and the underlying $\BCM$, thus graph balls in one graph typically do not map to graph balls in the other. Hence we need to generalize the notion of neighborhoods. 
In a rooted marked graph $(G,\markfunc,\groot)$, let $\wih G$ be \emph{any} connected edge-subgraph (i.e., not necessarily spanned subgraph) of $G$ that contains $\groot$. With $\markfunc$ restricted to $\wih G$, we call $(\wih G,\markfunc,\groot)$ a \emph{generalized neighborhood} of $\groot$ in $(G,\markfunc)$. 

\smallskip
\paragraph*{\textbf{Comparing neighborhoods}}
In \cref{thm:LWC_BCM}, we claim that the local weak limit of the $\BCM$ is a mixture of branching processes ($\BP$s). To prove such a statement, we have to compare neighborhoods in the $\BCM$ to $\BP$ family trees through isomorphism. 
Our approach is given by fixing an ordering of the vertices of the family tree according to a breadth-first search algorithm, and constructing the $\BCM$ in an isomorphic fashion by adding one vertex at a time in this fixed order. 
Note that for the $\BP$, the children of each vertex are already ordered, which provides the desired ordering (we explain later how we can recursively define an ordering of all vertices). To make the comparison straightforward, it will be convenient to define an ordering of neighbors of an arbitrary vertex in the $\BCM$ as well. 
Below, we formalize this method of comparison by introducing \emph{ordered trees} and \emph{ordered isomorphism}, specify the  ordering of the neighbors of any vertex in the $\BCM$, and express the usual notion of isomorphism in terms of ordered isomorphism.

\smallskip
\paragraph*{\textbf{Ordered trees}}
We introduce some terminology, inspired by branching processes, to talk about a (locally finite) rooted tree $(T,\groot)$. We define \emph{generation} $r$ as $\boundary_r(T,\groot) := \{ v\in\vertices(T): d(\groot,v) = r \}$, i.e., the set of vertices at graph distance $r$ from the root. We call the neighbors of $v$ further away from the root its \emph{children} and the neighbor of $v$ closer to the root its \emph{parent}. We call $(T,\groot)$ an \emph{ordered tree} if the children of any vertex are ordered. (Such trees are sometimes called planted plane trees or Catalan trees \cite{Drmo09}.) 

\smallskip
\paragraph*{\textbf{Ulam-Harris labeling}}
We can use the above ordering of children to recursively build a labeling of all vertices, called the Ulam-Harris labeling. 
Each label is a sequence or word on the alphabet $\N$, and we start by labeling the root as $\otroot := (0)$. Suppose a vertex in generation $r$ is labeled by the sequence $\vect{v} = (0,v_1,\ldots,v_r)$, then we label its $k\ith$ child by $(0,v_1,\ldots,v_r,k)$. We denote the generation of $\vect{v}$ by $\abs{\vect{v}} = r$ (which equals the length of the sequence minus one). 
The Ulam-Harris labeling provides an ordering of all vertices in the tree, defined as follows: if $\abs{\vect{v}}<\abs{\vect{w}}$, then $\vect{v}<\vect{w}$; if $\abs{\vect{v}}=\abs{\vect{w}}$, then we compare the sequences lexicographically. 
In the following, we reflect ordered trees (which are necessarily always rooted) in the notation by writing $\otroot$ for the root. 

We can now establish the ordering of children in rooted subtrees of the underlying $\BCM$. Recall (see \cref{ss:RIGC_def}) that the $\msl$-half-edges incident to $v\in\leftpar$ are labeled by $(v,i)_{i\in[\ldeg(v)]}$ and the $\msr$-half-edges incident to $a\in\rightpar$ are labeled by $(a,l)_{l\in[\rdeg(a)]}$, further, an edge formed by $\msl$-half-edge $(v,i)$ and $\msr$-half-edge $(a,l)$ is labeled $(i,l)$. We use the labels on the respective edges between the vertex and its children to order the children: for an $\msl$-vertex, we order its children by the \emph{first} (lhs) coordinate of the edge label, and for an $\msr$-vertex, we order its children by the \emph{second} (rhs) coordinate of the edge label. 

\smallskip
\paragraph*{\textbf{Ordered isomorphism and its relation to isomorphism}}
Let $(T_1,\markfunc_1,\otroot_1)$ and $(T_2,\markfunc_2,\otroot_2)$ denote two ordered marked trees, and consider the correspondence between their vertices with identical Ulam-Harris labels. If this correspondence is a rooted marked isomorphism,\footnote{Note that the Ulam-Harris labeling ensures that as long as each vertex has a corresponding vertex in the other tree, the graph structure is the same, and also that the root is mapped to the root. Further, we have to ensure that the corresponding vertices and edges have the same mark.} 
then we say that the trees are ordered isomorphic and denote this by $(T_1,\markfunc_1,\otroot_1) \orderedisom (T_2,\markfunc_2,\otroot_2)$.\footnote{Note the difference between the notation for isomorphism $\isom$ and ordered isomorphism $\orderedisom$.} 

In the following, we show how isomorphism can be expressed in terms of ordered isomorphism. Let $(T,\markfunc,\groot)$ be a \emph{finite} rooted marked tree (without an ordering). Then, there are \emph{finitely} many, say $I\in\Z^+$, ways to equip this tree with an ordering (or equivalently, Ulam-Harris labeling), and denote the set of all possible resulting \emph{ordered} trees by $(T^{(i)},\markfunc^{(i)},\otroot^{(i)})_{i\leq I}$. For short, we denote $\vect{T}^{(i)} := (T^{(i)},\markfunc^{(i)},\otroot^{(i)})$. By the construction of $(\vect{T}^{(i)})_{i \leq I}$, any ordered marked tree $(T',\markfunc',\otroot')$ that is isomorphic to $(T,\markfunc,\groot)$ must be ordered isomorphic to $\vect{T}^{(i)}$ for a unique $i\leq I$; in other words, we have \emph{partitioned} the isomorphism class of $(T,\markfunc,\groot)$ into the ordered isomorphism classes of each $(\vect{T}^{(i)})_{i\leq I}$. Consequently, we can write the event
\beq\label{eq:isom_decomp_event} \{ (T',\markfunc',\otroot') \isom (T,\markfunc,\groot) \} 
= \disjointunion_{i\leq I} \{ (T',\markfunc',\otroot') \orderedisom (T^{(i)},\markfunc^{(i)},\otroot^{(i)}) \} \eeq
as a \emph{disjoint} union. This implies that, if $(T',\markfunc',\otroot')$ is a random ordered marked tree,
\beq\label{eq:isom_decomp_prob} \P\bigl( (T',\markfunc',\otroot') \isom (T,\markfunc,\groot) \bigr) 
= \sum_{i\leq I} \P\bigl( (T',\markfunc',\otroot') \orderedisom (T^{(i)},\markfunc^{(i)},\otroot^{(i)}) \bigr). \eeq

\section{Local weak convergence of the bipartite configuration model}
\label{s:BCM}
In this section, we prove \cref{thm:LWC_BCM}; in fact, we prove local weak convergence for the $\BCM$ where we mark each vertex by $\msl$ or $\msr$ according to its partition.  
In \cref{ss:construction_LWClim_BCM}, we define the mark function on the $\BCM$ as well as the local weak limit of this marked graph, and we prove the LWC in \cref{ss:LWC_BCM}.

\subsection{Describing the local weak limit}
\label{ss:construction_LWClim_BCM}
First, we introduce $(\BP_\mss,\otroot)$, the local weak limit in probability of the $\BCM$. Intuitively, we expect this random rooted graph to describe the neighborhood of a vertex chosen uar over the \emph{entire graph}, while we also expect this neighborhood to look different, depending on whether we choose an $\msl$- or an $\msr$-vertex as the root. However, we have no \emph{direct} way to determine which partition our chosen vertex falls in from observing only its neighborhood. Hence, it will be useful to keep track of the lhs and rhs partitions as marks. Recall from \cref{ss:LWC_def} that we represent edges of the $\BCM$ as the pair of comprising half-edges that receive marks separately. Let $\markset^\msb=\{\msl,\msr,\nomark\}$ be the mark set we use, and mark $\msl$-vertices as $\msl$, $\msr$-vertices as $\msr$, and half-edges as the ``no mark'' symbol $\nomark$. Formally, 
\beq\label{eq:markfunc_BCM}
\markfunc_n^\msb(x) :=
\begin{cases}
\msl & \text{if $x\in\leftpar$},\\
\msr & \text{if $x\in\rightpar$},\\
\nomark & \text{if $x$ is a half-edge}.
\end{cases}
\eeq
Next, we introduce the object $(\BP_\mss,\markfunc^\mss,\otroot)$ that we below (see \cref{prop:LWC_BCM_marked}) establish as the local weak limit of the $\BCM$ equipped with the mark function $\markfunc_n^\msb$; $(\BP_\mss,\otroot)$ in \cref{thm:LWC_BCM} is obtained by omitting the mark function $\markfunc^\mss$. 
Recall $\gamma$ from \eqref{eq:mnratio}. We define $(\BP_\mss,\markfunc^\mss,\otroot)$ as a mixture of two marked $\BP$-trees $(\BP_\msl,\markfunc^\msl,\otroot)$ and $(\BP_\msr,\markfunc^\msr,\otroot)$ defined below, with mixing variable $\mss$:
\begin{subequations}\label{eq:def_BPmixing}
\begin{gather} \P(\mss=\msl) = 1/(1+\gamma), \quad \P(\mss=\msr) = \gamma/(1+\gamma), \\
(\BP_\mss,\markfunc^\mss,\otroot) \eqindis \ind_{\{\mss=\msl\}} (\BP_\msl,\markfunc^\msl,\otroot) + \ind_{\{\mss=\msr\}} (\BP_\msr,\markfunc^\msr,\otroot).
\end{gather}
\end{subequations}
Intuitively, $(\BP_\msl,\markfunc^\msl,\otroot)$ and $(\BP_\msr,\markfunc^\msr,\otroot)$ describe the neighborhood of an $\msl$- and $\msr$-vertex, respectively. With $V_n^{\msl}\sim\Unif[\leftpar]$, $V_n^{\msr}\sim\Unif[\rightpar]$ and the generalized notion of $\toLWCP$ from \cref{rem:LWC_nonunif}, $(\BCM_n,\markfunc^\msb,V_n^\msl)\toLWCP(\BP_\msl,\markfunc^\msl,\otroot)$ and $(\BCM_n,\markfunc^\msb,V_n^\msr)\toLWCP(\BP_\msr,\markfunc^\msr,\otroot)$, as revealed by the proof of \cref{prop:LWC_BCM_marked} (see \cref{ss:LWC_BCM}). Consequently, we can re-interpret the mixing variable $\mss$ as the random mark of the root.

Lastly, we define the marked $\BP$-tree $(\BP_\msl,\markfunc^\msl,\otroot)$, that we think of as a random ordered marked tree. Recall \eqref{eq:def_sizebiasing}, \cref{asmp:convergence} \cref{cond:limit_ldeg,cond:limit_rdeg}. We consider a discrete-time $\BP$, and the offspring of any two individuals are independent. 
We mark individuals in even and odd generations respectively by $\msl$ and $\msr$, and edges are ``unmarked'' (marked by $(\nomark,\nomark)$).  Generation $0$ consists of the single \emph{root} $\otroot$ that has offspring distributed as $D^\msl$; further individuals marked $\msl$ and $\msr$ have offspring distributed as $\wit D^\msl$ and $\wit D^\msr$, respectively. This concludes the dynamics of the $\BP$. 
We define $(\BP_\msr,\markfunc^\msr,\otroot)$ as the corresponding $\BP$-tree when we reverse the roles of $\msl$ and $\msr$.

\smallskip
\paragraph*{\textbf{Neighborhoods in the limit}}
To prepare for proving local weak convergence, we study the probability of observing a certain tree in a finite neighborhood (of radius $r\in\Z^+$) of the root $\otroot$ in $\BP_\mss$. Let us denote the set of ordered marked trees $(T,\markfunc_T,\otroot)$ that are possible $\BP$ family trees by $\vect{\supp}(\BP_\mss,\markfunc^\mss,\otroot)$ (the support of the distribution). Recall that vertices are marked $\msl$ or $\msr$ in an alternating fashion, i.e., marks are chosen based solely on the parity of the generation. 
Thus given the rooted tree $(T,\otroot)$, the function $\markfunc_T$ is uniquely determined by the single value $m_{\otroot} := \markfunc_T(\otroot)$. Recall that the random mark of the root in $\BP_\mss$ is denoted by $\mss=\markfunc^\mss(\otroot)$. Consequently,  $B_{r}(\BP_\mss,\markfunc^\mss,\otroot) \orderedisom B_{r}(T,\markfunc_T,\otroot)$ holds exactly when $\mss = m_{\otroot}$ and $B_{r}(\BP_\mss,\otroot) \orderedisom B_{r}(T,\otroot)$, formally,
\beq\label{eq:BPdistr} \bal \P\bigl( B_{r}(\BP_\mss,\markfunc^\mss,\otroot) \orderedisom B_{r}(T,\markfunc_T,\otroot) \bigr) 
&= \P\bigl(B_{r}(\BP_\mss,\otroot) \orderedisom B_{r}(T,\otroot) \bcond \mss = m_{\otroot} \bigr) \cdot \P(\mss = m_{\otroot} ) \\
&= \P\bigl(B_{r}(\BP_{m_{\otroot}},\otroot) \orderedisom B_{r}(T,\otroot) \bigr) \cdot \P(\mss = m_{\otroot}). \eal \eeq
By the construction of $(\BP_\mss,\markfunc^\mss,\otroot)$, 
\beq\label{eq:random_side_distr} \P(\mss = m_{\otroot}) =
\begin{cases}
\P(\mss = \msl) = 1/(1+\gamma)& \text{for $m_{\otroot} = \msl$},\\
\P(\mss = \msr) = \gamma/(1+\gamma)& \text{for $m_{\otroot} = \msr$}.
\end{cases}
\eeq
In the following, we focus on the case $m_{\otroot} = \msl$ and further expand the factor 
$\P\bigl(B_r(\BP_\msl,\otroot) \orderedisom B_r(T,\otroot) \bigr)$ from \eqref{eq:BPdistr} using characteristics of the ordered tree $(T,\otroot)$. Recall (see \cref{ss:neighborhoods_trees}) that we denote the generation of a vertex $\vect{v}$ by $\abs{\vect{v}}$, and further denote its degree by $d(\vect{v})$. Note that for $\vect{v} \neq \otroot$, $d(\vect{v})$ equals the number of its children plus one. 
Recall $\bitp$ and $\bitq$ from \cref{asmp:convergence} \cref{cond:limit_ldeg} and \cref{cond:limit_rdeg} respectively, further recall \eqref{eq:def_sizebiasing} and fix $r\in\Z^+$. 
Due to how ordered isomorphism is defined (see \cref{ss:neighborhoods_trees}), the probability of observing a given (unmarked) neighborhood in the $\BP$ equals the probability of observing the given sequence of degrees:
\beq\label{eq:BPdistr_sides} \bal
\P\bigl(B_r(\BP_\msl,\otroot) \orderedisom B_r(T,\otroot) \bigr) 
&= \P\bigl( D^\msl = d(\otroot) \bigr)
\!\!\! \prod_{\substack{\vect{v}\in T\\ 0<\abs{\vect{v}}<r\\ \abs{\vect{v}} \text{ odd}}} \!\!\! 
\P\bigl( \wit D^\msr = d(\vect{v})-1 \bigr)
\!\!\! \prod_{\substack{\vect{w}\in T\\ 0<\abs{\vect{v}}<r\\ \abs{\vect{w}} \text{ even}}} \!\!\! 
\P\bigl( \wit D^\msl = d(\vect{w})-1 \bigr) \\
&= p_{d(\otroot)}
\prod_{\substack{\vect{v}\in T\\ 0<\abs{\vect{v}}<r\\ \abs{\vect{v}} \text{ odd}}} \frac{d(\vect{v})\cdot q_{d(\vect{v})}}{\E[D^\msr]}
\prod_{\substack{\vect{w}\in T\\ 0<\abs{\vect{v}}<r\\ \abs{\vect{w}} \text{ even}}} \frac{d(\vect{w})\cdot p_{d(\vect{w})}}{\E[D^\msl]}.
\eal \eeq
Note that we only check the degrees up to generation $r-1$, since generation $r$ in the $r$-neighborhood in any tree are just leaves. 
The case $m_{\otroot} = \msr$ can easily be obtained by reversing the roles of $\msl$ and $\msr$, and consequently the roles of $\bitp$ and $\bitq$. Then combining (\ref{eq:BPdistr}-\ref{eq:BPdistr_sides}) yields
\beq\label{eq:neighborhood_lim_expand} \bal 
\P\bigl( B_r(\BP_\mss,&\markfunc^\mss,\otroot) \orderedisom B_r(T,\markfunc_T,\otroot) \bigr) \\
&= {\begin{cases}
\displaystyle
\frac{1}{1+\gamma} \; p_{d(\otroot)} 
\prod_{\substack{\vect{v}\in T\\ 0<\abs{\vect{v}}<r\\ \abs{\vect{v}} \text{ odd}}} \frac{d(\vect{v})\cdot q_{d(\vect{v})}}{\E[D^\msr]} 
\prod_{\substack{\vect{w}\in T\\ 0<\abs{\vect{w}}<r\\ \abs{\vect{w}} \text{ even}}} \frac{d(\vect{w})\cdot p_{d(\vect{w})}}{\E[D^\msl]}, 
& \text{for $m_{\otroot} = \markfunc_T(\otroot) = \msl$}; \\
\displaystyle
\frac{\gamma}{1+\gamma} \; q_{d(\otroot)} 
\prod_{\substack{\vect{v}\in T\\ 0<\abs{\vect{v}}<r\\ \abs{\vect{v}} \text{ odd}}} \frac{d(\vect{v})\cdot p_{d(\vect{v})}}{\E[D^\msl]} 
\prod_{\substack{\vect{w}\in T\\ 0<\abs{\vect{w}}<r\\ \abs{\vect{w}} \text{ even}}} \frac{d(\vect{w})\cdot q_{d(\vect{w})}}{\E[D^\msr]}, 
& \text{for $m_{\otroot} = \markfunc_T(\otroot) = \msr$}.
\end{cases} }
\eal \eeq

\subsection{Proof of local weak convergence}\label{ss:LWC_BCM} In this section, we prove the following LWC of the marked $\BCM$:
\begin{proposition}[Local weak convergence of the \emph{marked} $\BCM$]
\label{prop:LWC_BCM_marked}
Consider $\BCM_n = \BCM(\bitd^\msl,\bitd^\msr)$ under \cref{asmp:convergence} (\ref{cond:limit_ldeg},\ref{cond:mean_ldeg},\ref{cond:limit_rdeg},\ref{cond:mean_rdeg}). Recall the mark function $\markfunc_n^\msb$ (encoding the partition of each vertex) from \eqref{eq:markfunc_BCM} and let $V_n^\msb = V_n^{\msl+\msr} \distr \Unif[\leftpar\cup\rightpar]$. Then, as $n\to\infty$,
\beq \bigl( \BCM_{n}, \markfunc_n^\msb, V_n^\msb \bigr) \toLWCP (\BP_\mss,\markfunc^\mss,\otroot). \eeq
\end{proposition}

Note that \cref{thm:LWC_BCM} asserts the same convergence without the mark function, thus subject to \cref{prop:LWC_BCM_marked}, \cref{thm:LWC_BCM} immediately follows.

\begin{proof}[Proof of \cref{prop:LWC_BCM_marked}]
Recall that $\bigl( \BCM_{n},\markfunc_n^\msb,V_n^\msb \bigr)$ has two sources of randomness: the graph realization determined by the bipartite matching $\bmatch_n$ and the choice of $V_n^\msb$, and that
$\P(\,\cdot \cond \bmatch_n)$ denotes the conditional probability wrt the graph realization. 
Further recall that $\vect{\supp}(\BP_\mss,\markfunc^\mss,\otroot)$ denotes the set of all possible ordered family trees produced by $\BP_\mss$. We claim that for all $r\in\N$ and all \emph{ordered} family trees $(T,\markfunc_T,\otroot) \in \vect{\supp}(\BP_\mss,\markfunc^\mss,\otroot)$, as $n\to\infty$,
\beq\label{eq:cond:neighborhoods_marked_ordered} \bal
&\P\bigl( B_r(\BCM_n,\markfunc_n^\msb,V_n^\msb) \orderedisom B_r(T,\markfunc_T,\otroot) \bcond (\BCM_n,\markfunc_n^\msb) \bigr) \\
& := \frac{1}{\abs{\leftpar\cup\rightpar}} \sum_{u\in\leftpar\cup\rightpar} \ind_{\{ B_r(\BCM_n,\markfunc_n^\msb,u) \orderedisom B_r(T,\markfunc_T,\otroot) \}} 
\toinp \P\bigl( B_r(\BP_\mss,\markfunc^\mss,\otroot) \orderedisom B_r(T,\markfunc_T,\otroot) \bigr). \eal \eeq
Below, we prove this statement via a first and second moment method, but first we show why this is sufficient for completing the proof of \cref{prop:LWC_BCM_marked}. Let $\supp(\BP_\mss,\markfunc^\mss,\otroot)$ denote the set of \emph{unordered} versions of the family trees in $\vect{\supp}(\BP_\mss,\markfunc^\mss,\otroot)$. Subject to \eqref{eq:cond:neighborhoods_marked_ordered}, it follows by (\ref{eq:isom_decomp_event}-\ref{eq:isom_decomp_prob}) that for all unordered family trees $(T,\markfunc_T,\groot) \in \supp(\BP_\mss,\markfunc^\mss,\otroot)$ and all $r\in\N$, as $n\to\infty$,
\beq\label{eq:cond:neighborhoods_marked} \bal 
&\P\bigl( B_r(\BCM_n,\markfunc_n^\msb,V_n^\msb) \isom B_r(T,\markfunc_T,\groot) \bcond (\BCM_n,\markfunc_n^\msb) \bigr) 
\toinp \P\bigl( B_r(\BP_\mss,\markfunc^\mss,\otroot) \isom B_r(T,\markfunc_T,\groot) \bigr). \eal \eeq
Then, for any $r\in\N$, it also immediately follows by completeness of measure (for a more detailed argument, see \cite[Theorem 1.17]{rvHofRGCNvol2}) that for a rooted marked graph $(G,\markfunc,\groot)$ \emph{not} in $\supp(\BP_\mss,\markfunc^\mss,\otroot)$, its frequency as a neighborhood must converge to $0$ in probability. That is, \eqref{eq:cond:neighborhoods_marked} follows for \emph{all} rooted marked graphs $(G,\markfunc,\groot) \in \graphs_\groot(\markset^\msb)$, which is how we defined LWC in probability in \eqref{eq:def_LWCP}. This concludes the proof of \cref{prop:LWC_BCM_marked} subject to \eqref{eq:cond:neighborhoods_marked_ordered}. In the following, we prove \eqref{eq:cond:neighborhoods_marked_ordered}.

\subsubsection{First moment}\label{sss:LWC_firstmom} 
Let us fix an arbitrary ordered family tree $\vect{T} := (T,\markfunc_T,\otroot)\in\supp(\BP_\mss,\markfunc^\mss,\otroot)$ and an integer $r\in\N$. For convenience, we introduce the event, for some $v\in \vertices^\msb := \leftpar\cup\rightpar$, 
\beq\label{eq:neighborhood_event} \ballevent_{\vect{T}}^r(v) := \bigl\{ B_r(\BCM_n,\markfunc^\msb,v) \orderedisom B_r(T,\markfunc_T,\otroot) \bigr\}, \eeq
which is the event on the lhs of \eqref{eq:cond:neighborhoods_marked_ordered}. 
We compute the expected neighborhood frequency
\beq\label{eq:neighborhood_1stmoment} \bal \E\bigl[ \P\bigl( \ballevent_{\vect{T}}^r(V_n^\msb) \bcond \bmatch_n \bigr) \bigr]
= \P\bigl( \ballevent_{\vect{T}}^r(V_n^\msb) \bigr), \eal \eeq
where $\P$ on the rhs denotes total probability, i.e., wrt the product measure of $\bmatch_n$ and $V_n^\msb$. We compute this probability analogously to the neighborhood probabilities in the limit in \eqref{eq:neighborhood_lim_expand}, taking advantage of the notion of ordered trees and ordered isomorphism introduced in \cref{ss:neighborhoods_trees}). There, we have also turned tree neighborhoods in the underlying $\BCM$ into ordered trees using edge labels, and due to \cref{rem:pairingalgo}, we can construct the $\BCM$ in the order prescribed by $\vect{T}$. Thus we can compute the probability in \eqref{eq:neighborhood_1stmoment} simply by checking the degree and mark of the vertex added in each step as we construct the $\BCM$.

Again, we only study the case $\markfunc_T(\otroot)=\msl$ in detail, as the case $\markfunc_T(\otroot)=\msr$ is analogous. Denote the set of vertices with a given mark and degree as
\beq\label{eq:def_Vk} \leftpar_k := \{v\in\leftpar: \ldeg(v)=k\}, 
\qquad \rightpar_k := \{a\in\rightpar: \rdeg(a)=k\}. \eeq
Further, recall (see \cref{ss:neighborhoods_trees}) that the Ulam-Harris labels give an ordering ($<$) of all vertices in $B_r(\vect{T}) := B_r(T,\markfunc_T,\otroot)$, that $\boundary_r(\vect{T})$ denotes generation $r$ in $\vect{T}$ and that the degree of $\vect{v}\in\vect{T}$ is denoted by $d(\vect{v})$. 
Recall that the mark function $\markfunc_n^\msb$ on the $\BCM$ from \eqref{eq:markfunc_BCM} is constant, thus in particular, it does \emph{not} depend on the matching. We can hence assume that the marks are already present while constructing the matching. 
As we construct a neighborhood in the $\BCM$, we want it to be isomorphic to $B_r(\vect{T})$, thus we cannot choose the same vertices and half-edges again, which reduces the number of available objects with the desired properties, i.e., mark and degree. This is called the depletion-of-points-and-half-edges effect, 
and we can quantify it in terms of $B_r(\vect{T})$, since we construct the neighborhood in the $\BCM$ to be isomorphic to this graph; the following notation serves this purpose. We denote the total number of vertices in $B_r(\vect{T})$ and the number of vertices preceding some vertex $\vect{v} \in B_r(\vect{T})$ in the ordering respectively by
\beq\label{eq:def_depletion_constants1} 
\otreesize := \abs{B_r(\vect{T})}, 
\qquad \otreebefore{\vect{v}} := \bigl\lvert \{ \vect{w}\in B_r(\vect{T}):\, \vect{w} < \vect{v} \} \bigr\rvert. \eeq
Since $\vect{T}$ is a tree, $\otreesize-1$ equals the total number of edges in $B_r(\vect{T})$, and $\otreebefore{\vect{v}}-1$ equals the number of edges created before choosing the vertex corresponding to $\vect{v}$. 
Further, let $\vect{w} \isom \vect{v}$ denote the event that $\vect{w}$ is ``similar'' to $\vect{v}$, in the sense of having the same mark and degree; formally,
\beq\label{eq:def_similar_vertices} \{ \vect{w} \isom \vect{v} \} 
:= \{ \markfunc_T(\vect{w}) = \markfunc_T(\vect{v}) \} \cap \{ d(\vect{w}) = d(\vect{v}) \}. \eeq
Denote the number of vertices similar to $\vect{v}$, in total in $B_r(\vect{T})$ and preceding $\vect{v}$, respectively, by 
\beq\label{eq:def_depletion_constants2} 
\otreetype{\vect{v}} := \bigl\lvert \{ \vect{w}\in B_r(\vect{T}):\, \vect{w} \isom \vect{v} \} \bigr\rvert, 
\qquad \otreetypebefore{\vect{v}} := \bigl\lvert \{ \vect{w}\in B_r(\vect{T}):\, \vect{w} < \vect{v},\, \vect{w} \isom \vect{v} \} \bigr\rvert. 
\eeq
Intuitively, $\otreetypebefore{\vect{v}}$ counts the number of vertices with the desired properties that are already used up in the construction of the neighborhood in the $\BCM$ when we pick the vertex corresponding to $\vect{v}\in B_r(\vect{T})$, while $\otreetype{\vect{v}}$ is the corresponding quantity after constructing the entire neighborhood isomorphic to $B_r(\vect{T})$.  
Recall \eqref{eq:def_halfedges}, \eqref{eq:def_pq_pmf}, \eqref{eq:mnratio} and \eqref{eq:neighborhood_event}, and denote the positive part of expression $f$ by $f_+:=\max\{f,0\}$. 
Let $\wih \ballevent_{\vect{T}}^{r-}(V_n^\msb)$ denote the event that the neighborhood of $V_n^\msb$ has the desired tree structure up to generation $r-1$, i.e., $\ballevent_{\vect{T}}^{r-1}(V_n^\msb)$ happens, and further, generation $r-1$ in the neighborhood of $V_n^\msb$ has the desired degree sequence $(d(\vect{v}))_{\vect{v}\in\boundary_{r-1}(\vect{T})}$. 
Keeping the intuitive meaning of the quantities defined in \eqref{eq:def_depletion_constants1} and \eqref{eq:def_depletion_constants2} in mind, as well as that creating one edge requires one $\msl$-half-edge and one $\msr$-half-edge, we calculate 
\begin{subequations}\label{eq:neighborhood_1stmom_expand} 
\begin{align} 
\P\bigl( \ballevent_{\vect{T}}^r(V_n^\msb) \bigr) 
&= \frac{\abs{\vertices^{\msl}_{d(\otroot)}}}{N_n+M_n}
\prod_{\substack{\vect{i}\in T\\ 0<\abs{\vect{i}}<r\\ \abs{\vect{i}}\text{ odd}}} 
\frac{ d(\vect{i}) \bigl(\abs{\rightpar_{d(\vect{i})}} - \otreetypebefore{\vect{i}} \bigr)_+ }{\he_n-\bigl(\otreebefore{\vect{i}}-1\bigr)} 
\prod_{\substack{\vect{j}\in T\\ 0<\abs{\vect{j}}<r\\ \abs{\vect{j}}\text{ even}}} 
\frac{ d(\vect{j}) \bigl(\abs{\leftpar_{d(\vect{j})}} - \otreetypebefore{\vect{j}} \bigr)_+ }{\he_n-\bigl(\otreebefore{\vect{j}}-1\bigr)} \\
\label{eq:neighborhood_1stmom_expand_condprob} &\phantom{{}={}} \times \P\bigl( \text{no cycle closed in generation $r$} \bcond \wih \ballevent_{\vect{T}}^{r-}(V_n^\msb) \bigr). \end{align}
\end{subequations}
Note that once again, the products only include degrees of vertices up to generation $r-1$, as the generation $r$, which is the \emph{last generation} in the $r$-ball, consists of leaves: while they may have further neighbors in the complete tree $\vect{T}$, those potential neighbors are not part of the $r$-ball. To ensure that the vertices in generation $r$ of the $\BCM$ are also \emph{leaves in the $r$-ball}, we have to make sure they do not create any cycle in generation $r$, and the probability \eqref{eq:neighborhood_1stmom_expand_condprob} accounts for this, conditionally on the past of the construction $\wih \ballevent_{\vect{T}}^{r-}(V_n^\msb)$. 
Due to the bipartite structure, odd cycles (edges between two vertices in generation $r$) are impossible, thus we only have to make sure the vertices chosen as ``leaves'' do not coincide. Denote the desired number of vertices (``leaves'') in generation $r$ by $L := \abs{\boundary_r(\vect{T})}$, and let $\dmax{\msb} := \max\{\dmax{\msl},\dmax{\msr}\}$. We can bound the complement probability of \eqref{eq:neighborhood_1stmom_expand_condprob} using the union bound as
\beq\label{eq:probcycle} \bal &\P\bigl( \text{there is a cycle closed in generation $r$} \bcond \wih \ballevent_{\vect{T}}^{r-}(V_n^\msb) \bigr) \\
&\leq \sum_{i=2}^{L} \P(\text{first cycle created by $i\ith$ ``leaf''})
\leq \sum_{i=2}^{L} \frac{(i-1)(\dmax{\msb}-1)}{\he_n - (\otreesize - 1) - (i - 1)},
\eal \eeq
where the denominator is the exact number of available half-edges (considering the bipartite structure of the graph), while the numerator is an upper bound on the available half-edges incident to vertices already chosen in generation $r$. For fixed $r$ and $\vect{T}$, $\otreesize$ and $L$ are fixed constants, and by \cref{rem:asmp_consequences} \cref{cond:dmax}, $\dmax{\msb} = o(\he_n)$, thus we have finitely many $o(1)$ terms. Consequently,
\beq\label{eq:neighborhood_1stmom_1} \P\bigl( \text{no cycle closed in generation $r$} \bcond \wih \ballevent_{\vect{T}}^{r-}(V_n^\msb) \bigr) 
= 1-o(1). \eeq
By \cref{asmp:convergence} \cref{cond:limit_ldeg} and \cref{rem:asmp_consequences} \cref{cond:partition_ratio}, the first factor
\beq\label{eq:neighborhood_1stmom_3} \frac{\abs{\leftpar_{d(\otroot)}}}{N_n+M_n} 
= \frac{N_n}{N_n+M_n} \frac{\abs{\leftpar_{d(\otroot)}}}{N_n} 
\to \frac{1}{1+\gamma} p_{d(\otroot)}. \eeq
We now look at factors in the product over $\vect{i}$. 
By (\ref{eq:def_depletion_constants1}-\ref{eq:def_depletion_constants2}), $0\leq \otreetypebefore{\vect{i}} \leq \otreetype{\vect{i}} \leq \otreesize$, and $\he_n\geq\he_n-(\otreebefore{\vect{i}}-1)>\he_n-\otreesize$ (since $\vect{i}\neq\otroot$). Then, for fixed $r$ and $\vect{T}$, by \cref{asmp:convergence} \cref{cond:limit_rdeg}, \cref{cond:mean_rdeg} and \cref{rem:asmp_consequences} \cref{cond:partition_ratio}, as $n\to\infty$,
\beq\label{eq:neighborhood_1stmom_2} \frac{ d(\vect{i}) \bigl(\abs{\rightpar_{d(\vect{i})}} - \otreetypebefore{\vect{i}} \bigr)_+ }{\he_n-\bigl(\otreebefore{\vect{i}}-1\bigr)} 
= \frac{d(\vect{i}) \bigl(M_n \cdot \nth q_{d(\vect{i})}-O(1)\bigr) }{M_n\E[D_n^\msr]-O(1)}
\to \frac{d(\vect{i}) \cdot q_{d(\vect{i})}}{\E[D^\msr]}. \eeq
It is analogous to show convergence for the factors in the product over $\vect{j}$, using \cref{asmp:convergence} (\ref{cond:limit_ldeg}-\ref{cond:mean_ldeg}) and \cref{rem:asmp_consequences} \cref{cond:partition_ratio}. 
From (\ref{eq:neighborhood_1stmom_1}-\ref{eq:neighborhood_1stmom_2}), we conclude that, since there are only finitely many factors, \eqref{eq:neighborhood_1stmom_expand} converges to \eqref{eq:neighborhood_lim_expand} as $n\to\infty$. That is, as required,
\beq\label{eq:neighborhood_1stmom_conv} \E\bigl[ \P\bigl( \ballevent_{\vect{T}}^r(V_n^\msb) \bcond \bmatch_n \bigr) \bigr]
\to \P\bigl( B_r(\BP_\mss,\markfunc^\mss,\otroot) \orderedisom B_r(\vect{T}) \bigr). \eeq

\subsubsection{Second moment}\label{sss:LWC_secondmom}
Let $V_n^\msb, U_n^\msb \sim \Unif(\vertices^\msb)$, with $\vertices^\msb = \leftpar \cup \rightpar$, denote two independent, uniformly chosen vertices of the $\BCM$. To show that the variance of the neighborhood counts converge to $0$, we compute the second moment
\beq\label{eq:neighborhood_2ndmom_1} \bal 
\E\Bigl[ \P\bigl( \ballevent_{\vect{T}}^r(V_n^\msb) \bcond \bmatch_n \bigr)^2 \Bigr] 
&= \E\Bigl[ \frac{1}{(N_n+M_n)^2} \sum_{v,w \in \vertices^\msb} \ind_{ \ballevent_{\vect{T}}^r(v) } \cdot \ind_{\ballevent_{\vect{T}}^r(w) } \Bigr] \\
&= \frac{1}{(N_n+M_n)^2} \sum_{v,w \in \vertices^\msb} \P\bigl( \ballevent_{\vect{T}}^r(v) \cap \ballevent_{\vect{T}}^r(w) \bigr) 
= \P\bigl( \ballevent_{\vect{T}}^r(V_n^\msb) \cap \ballevent_{\vect{T}}^r(U_n^\msb ) \bigr). \eal \eeq
We analyze this probability in parts, noting that 
$\ind_{\{ U_n^\msb = V_n^\msb \}} + \ind_{\{ 0 < \dist(U_n^\msb ,V_n^\msb) \leq 2r \}} + \ind_{\{ \dist(U_n^\msb ,V_n^\msb) > 2r \}}$ equals $1$ almost surely, where $\dist(u,v)$ denotes the (random) graph distance of $u,v\in\vertices^\msb$ in the $\BCM$. First, on the event $\{ U_n^\msb = V_n^\msb \}$, we have that 
\beq\label{eq:neighborhood_2ndmom_3} 
\P\bigl( \ballevent_{\vect{T}}^r(V_n^\msb) \cap \ballevent_{\vect{T}}^r(U_n^\msb ) \cap \{V_n^\msb = U_n^\msb \} \bigr) 
\leq \P\bigl( V_n^\msb = U_n^\msb \bigr) 
= \frac{1}{M_n+N_n} \to 0 
\eeq
as $n\to\infty$. 
Next, we consider the event $\{ 0 < d(U_n^\msb ,V_n^\msb) \leq 2r \}$, where $V_n^\msb$ and $U_n^\msb $ are distinct, but their $r$-neighborhoods intersect. Let $K$ denote the largest degree in $B_r(\vect{T})$. 
On the event $\ballevent_{\vect{T}}^r(V_n^\msb) \cap \ballevent_{\vect{T}}^r(U_n^\msb ) \cap \{ 0 < d(U_n^\msb ,V_n^\msb) \leq 2r \}$, there must exist a path between $V_n^\msb$ and $U_n^\msb $ that is fully contained in the union of their $r$-neighborhoods. That is, the path consists of $j\leq 2r$ vertices, each of which has degree at most $K$. Let us denote $v_0 := V_n^\msb, v_j := U_n^\msb $. By relaxing the conditions on the path and taking a union bound,
\beq\label{eq:neighborhood_2ndmom_6} \begin{split} &\P\bigl( \ballevent_{\vect{T}}^r(V_n^\msb) \cap \ballevent_{\vect{T}}^r(U_n^\msb ) \cap \{0<d(V_n^\msb,U_n^\msb )\leq 2r\} \bigr) \\
&\leq 
\P\bigl( \bdeg(v_0), \bdeg(v_1)\leq K \bigr) \cdot \P\bigl( (v_0,v_1)\in\edges(\BCM) \bcond \bdeg(v_0), \bdeg(v_1)\leq K \bigr) \\
&\phantom{{}\leq{}}+ \sum_{j=2}^{2r} \, \sum_{v_1,...,v_{j-1}\in\vertices^\msb} 
\P\bigl( \forall\, 0\leq i\leq j,\, \bdeg(v_i)\leq K \bigr) \\
&\phantom{{}\leq{}+{}}\times \P\bigl( \forall\, 0\leq i\leq j-1,\, (v_i,v_{i+1})\in\edges(\BCM_n) \bcond \forall\, 0\leq i\leq j,\, \bdeg(v_i)\leq K \bigr).
\end{split} \eeq
The first factor in the first term and the first factor in the sum are trivially bounded by $1$. Since by construction, half-edges are paired uniformly, and we consider vertices with at most $K$ half-edges attached, we can upper bound \eqref{eq:neighborhood_2ndmom_6} by
\beq\label{eq:neighborhood_2ndmom_8} \sum_{j=1}^{2r} (N_n+M_n)^{j-1} 
\frac{K \cdot \bigl(K(K-1)\bigr)^{j-1} \cdot K}{(2\he_n-2j)^j}. \eeq
Note that for $r$ and $\vect{T}$ fixed, $K$ is a fixed, finite constant, while $M_n$ and $\he_n$ grow linearly with $N_n$ by \cref{rem:asmp_consequences} \cref{cond:partition_ratio}. Thus the bound we obtained in \eqref{eq:neighborhood_2ndmom_8} is of order $O(N_n^{-1})$, which implies that
\beq\label{eq:neighborhood_2ndmom_9} 
\P\bigl( \ballevent_{\vect{T}}^r(V_n^\msb) \cap \ballevent_{\vect{T}}^r(U_n^\msb ) \cap \{0<d(V_n^\msb,U_n^\msb )\leq 2r\} \bigr)
= O(N_n^{-1}) \to 0 \eeq
as $n\to\infty$. 
Finally, we restrict ourselves to the event $\{ d(U_n^\msb ,V_n^\msb) > 2r \}$, when the $r$-balls of $V_n^\msb$ and $U_n^\msb $ are disjoint. Analogously to before, we calculate the probability of interest through constructing neighborhoods in the $\BCM$ in the fashion prescribed by the ordered family tree $\vect{T}$. In particular, we first construct the $r$-ball around $V_n^\msb$ avoiding $U_n^\msb $, then construct the $r$-ball around $U_n^\msb $ avoiding the $r$-ball of $V_n^\msb$, leading to slight changes compared to \eqref{eq:neighborhood_1stmom_expand}. 
Again, we carry out the calculation in the case $\markfunc_T(\otroot) = \msl$; the case $\markfunc_T(\otroot) = \msr$ can be studied analogously. 
Recall \eqref{eq:def_halfedges}, \eqref{eq:def_pq_pmf}, \eqref{eq:mnratio}, $f_+:=\max\{f,0\}$, (\ref{eq:def_depletion_constants1}-\ref{eq:def_depletion_constants2}), and the event $\wih \ballevent_{\vect{T}}^{r-}(V_n^\msb)$ defined above \eqref{eq:neighborhood_1stmom_expand}. 
Similarly to \eqref{eq:neighborhood_1stmom_expand}, we compute and explain below
\begin{subequations}\label{eq:neighborhood_2ndmom_expand} 
\begin{align}
\label{eq:neighborhood_2ndmom_expand_UV} &\P\bigl( \ballevent_{\vect{T}}^r(V_n^\msb) \cap \ballevent_{\vect{T}}^r(U_n^\msb ) \cap \{d(V_n^\msb,U_n^\msb ) > 2r\} \bigr) 
= \frac{\abs{\leftpar_{d(\otroot)}}}{N_n+M_n} \cdot \frac{\abs{\leftpar_{d(\otroot)}}-1}{N_n+M_n} \\
&\label{eq:neighborhood_2ndmom_expand_Vneighborhood} \phantom{{}={}}\times 
\prod_{\substack{\vect{i}\in T\\ 0<\abs{\vect{i}}<r\\ \abs{\vect{i}} \text{ odd}}} 
\frac{ d(\vect{i}) \bigl( \abs{\rightpar_{d(\vect{i})}} - \otreetypebefore{\vect{i}} - \ind_{\{\vect{i} \isom \otroot\}} \bigr)_+ }{\he_n-\bigl(\otreebefore{\vect{i}}-1\bigr)} 
\prod_{\substack{\vect{j}\in T\\ 0<\abs{\vect{j}}<r\\ \abs{\vect{j}} \text{ even}}}
\frac{ d(\vect{j}) \bigl(\abs{\leftpar_{d(\vect{j})}} - \otreetypebefore{\vect{j}} - \ind_{\{\vect{j} \isom \otroot\}} \bigr)_+ }{\he_n-\bigl(\otreebefore{\vect{j}}-1\bigr)} \\
\label{eq:neighborhood_2ndmom_expand_Uneighborhood} &\phantom{{}={}}\times \prod_{\substack{\vect{k}\in T\\ 0<\abs{\vect{k}}<r\\ \abs{\vect{k}} \text{ odd}}} 
\frac{ d(\vect{k}) \bigl(\abs{\rightpar_{d(\vect{k})}} - \otreetype{\vect{k}} - \otreetypebefore{\vect{k}} \bigr)_+ }{\he_n-(\otreesize-1)-\bigl(\otreebefore{\vect{k}}-1\bigr)} 
\prod_{\substack{\vect{l}\in T\\ 0<\abs{\vect{l}}<r\\ \abs{\vect{l}} \text{ even}}}
\frac{ d(\vect{l}) \bigl(\abs{\leftpar_{d(\vect{l})}} - \otreetype{\vect{l}} - \otreetypebefore{\vect{l}} \bigr)_+ }{\he_n-(\otreesize-1)-\bigl(\otreebefore{\vect{l}}-1\bigr)} \\
\label{eq:neighborhood_2ndmom_expand_Vnocycle} &\phantom{{}={}}\times \P\bigl(\text{no cycle in $B_r(V_n^\msb)$ \& not connecting to $U_n^\msb $} \bcond \wih \ballevent_{\vect{T}}^{r-}(V_n^\msb) \bigr) \\
\label{eq:neighborhood_2ndmom_expand_Unocycle} &\phantom{{}={}}\times \P\bigl(\text{no cycle in $B_r(U_n^\msb)$ \& not connecting to $B_r(V_n^\msb)$} \bcond \ballevent_{\vect{T}}^{r}(V_n^\msb), \wih \ballevent_{\vect{T}}^{r-}(U_n^\msb) \bigr).
\end{align}
\end{subequations}
The rhs of \eqref{eq:neighborhood_2ndmom_expand_UV} corresponds to the choice of $V_n^\msb$ and $U_n^\msb$, and both factors separately converge to $\tfrac{1}{1+\gamma} p_{d(\otroot)}$, by \cref{asmp:convergence} \cref{cond:limit_ldeg} and \cref{rem:asmp_consequences} \cref{cond:partition_ratio}. 
The line \eqref{eq:neighborhood_2ndmom_expand_Vneighborhood} arises from the construction of the neighborhood of $V_n^\msb$, and only differs from the products in \eqref{eq:neighborhood_1stmom_expand} in the terms $\ind_{\{ \vect{i}\isom\otroot \}}$ and $\ind_{\{ \vect{j}\isom\otroot \}}$, which ensure avoiding $U_n^\msb $ in the neighborhood of $V_n^\msb$. By the same argument as in \cref{sss:LWC_firstmom}, each factor converges to the corresponding factor in \eqref{eq:neighborhood_lim_expand}. 
The line \eqref{eq:neighborhood_2ndmom_expand_Uneighborhood} arises from the construction of the neighborhood of $U_n^\msb $, where we have to avoid the neighborhood of $V_n^\msb$, thus we further exclude $(\otreesize-1)$ paired $\msl$-and $\msr$-half-edges and $\otreetype{\vect{v}}$ vertices ``similar'' to $\vect{v}$ when choosing the vertex corresponding to $\vect{v}$. As these are finite corrections, it still holds that each factor converges to the corresponding factor in \eqref{eq:neighborhood_lim_expand}, which we demonstrate for factors from the product over $\vect{k}$. 
Note that for $\vect{T}$ and $r$ fixed, $\otreesize$ is a finite constant, and it serves as an upper bound for $\otreetype{\vect{k}}$, $\otreebefore{\vect{k}}$ and $\otreetypebefore{\vect{k}}$, see \eqref{eq:def_depletion_constants1} and \eqref{eq:def_depletion_constants2}. 
Then, by \cref{asmp:convergence} \cref{cond:limit_rdeg}, \cref{cond:mean_rdeg} and \cref{rem:asmp_consequences} \cref{cond:partition_ratio}, as $n\to\infty$, 
\beq\label{eq:neighborhood_2ndmom_5} \frac{ d(\vect{k}) \bigl(\abs{\rightpar_{d(\vect{k})}} - \otreetype{\vect{k}} - \otreetypebefore{\vect{k}} \bigr)_+ }{\he_n-(\otreesize-1)-\bigl(\otreebefore{\vect{k}}-1\bigr)}
= \frac{ d(\vect{k}) \bigl(M_n \cdot \nn q_{d(\vect{k})} - O(1) \bigr)_+ }{ M_n \E[D_n^\msr] - O(1)} \to \frac{d(\vect{k}) \cdot q_{d(\vect{k})}}{\E[D^\msr]}. \eeq
Analogous results hold for the factors in the product over $\vect{l}$, using \cref{asmp:convergence} (\ref{cond:limit_ldeg}-\ref{cond:mean_ldeg}) and \cref{rem:asmp_consequences} \cref{cond:partition_ratio}. 
Note that each factor in the products in \eqref{eq:neighborhood_lim_expand} now appears as the limit of \emph{two} factors in \eqref{eq:neighborhood_2ndmom_expand}, one in \eqref{eq:neighborhood_2ndmom_expand_Vneighborhood} and one in \eqref{eq:neighborhood_2ndmom_expand_Uneighborhood}. 
By calculations analogous to \eqref{eq:probcycle}, the conditional probabilities in (\ref{eq:neighborhood_2ndmom_expand_Vnocycle}-\ref{eq:neighborhood_2ndmom_expand_Unocycle}) are again $1-o(1)$. 
Combining \eqref{eq:neighborhood_2ndmom_expand} and the arguments below it with \eqref{eq:neighborhood_lim_expand} yields that, as $n\to\infty$, 
\beq\label{eq:neighborhood_2ndmom_mainterm} \bal 
\P\bigl( \ballevent_{\vect{T}}^r(V_n^\msb) \cap \ballevent_{\vect{T}}^r(U_n^\msb ) \cap \{d(V_n^\msb,U_n^\msb )> 2r\} \bigr) 
\to \P\bigl( B_r(\BP_\mss,\markfunc^\mss,\otroot) \orderedisom B_r(\vect{T}) \bigr)^2. \eal \eeq
Combining \eqref{eq:neighborhood_2ndmom_3}, \eqref{eq:neighborhood_2ndmom_9} and \eqref{eq:neighborhood_2ndmom_mainterm}, we conclude that as $n\to\infty$,
\beq\label{eq:neighborhood_2ndmom_total} \E\bigl[ \P( \ballevent_{\vect{T}}^r(V_n^\msb) \cond \bmatch_n )^2 \bigr]  
\to \P\bigl( B_r(\BP_\mss,\markfunc^\mss,\otroot) \orderedisom B_r(\vect{T}) \bigr)^2. \eeq
Thus, by \eqref{eq:neighborhood_1stmom_conv}, it follows that, as $n\to\infty$,
\beq\label{eq:neighborhood_2ndmom_10} \Var \bigl( \P( \ballevent_r(V_n^{\msl+\msr}) \cond \bmatch_n ) \bigr) \to 0. \eeq
By \eqref{eq:neighborhood_1stmom_conv} and \eqref{eq:neighborhood_2ndmom_10}, Chebyshev's inequality yields that \eqref{eq:cond:neighborhoods_marked_ordered} holds for arbitrary $r\in\N$ and $\vect{T}=(T,\markfunc_T,\otroot)\in\vect{\supp}(\BP_\mss,\markfunc^\mss,\otroot)$. 
Since at the beginning of \cref{ss:neighborhoods_trees}, we have reduced \cref{prop:LWC_BCM_marked} to this statement, this concludes the proof of \cref{prop:LWC_BCM_marked}.
\end{proof}
The calculation of the second moment, in particular  the result \eqref{eq:neighborhood_2ndmom_total}, asserts that neighborhoods of two independently and uniformly chosen vertices are \emph{asymptotically independent}.

\section{Proof of results on the random intersection graph with communities}
\label{s:proof_local}
In this section, we provide the proofs of our results on the local properties of the $\RIGC$ model. We introduce the local weak limit of the $\RIGC$ in \cref{ss:construction_LWClim_RIGC} and formally prove \cref{thm:LWC_RIGC} on the local weak convergence in \cref{ss:RIGC_LWC_proof}. Finally, we prove the consequences of local weak convergence for the degrees and local clustering coefficient as well as the overlapping structure in \cref{ss:deg_clust,ss:overlaps}, respectively.

\subsection{The local weak limit of the RIGC}
\label{ss:construction_LWClim_RIGC}
In this section, we construct the random rooted graph $(\rCP,\groot)$, that is the local weak limit in probability of the $\RIGC$. The notation is inspired by the fact that $(\rCP,\groot)$ is the ``community projection'' (see \cref{ss:RIGC_def}) of a random rooted marked tree $(\BP_\msl,\markfunc^\msp,\otroot)$ defined below, in the same way that the $\RIGC$ is the ``community projection'' of the underlying $\BCM$. It is then not surprising that $(\BP_\msl,\markfunc^\msp,\otroot)$ is the local weak limit of the underlying $\BCM$, including the community graphs, and that it is obtained from the $\BP$-tree $(\BP_\msl,\otroot)$, introduced in \cref{ss:construction_LWClim_BCM}, by equipping it with a \emph{new} mark function $\markfunc^\msp$ defined below. In the following, we give a formal definition of these objects, starting from the marked graph representation of the underlying $\BCM$.

\smallskip
\paragraph*{\textbf{The pre-image: the community-marked BCM}}
We introduce the new mark function $\markfunc^\msc$ on $\BCM$ to encode not only the partition of each vertex, but also the community graphs and the assignment of community roles. Recall the set of possible community graphs $\comgraphs$ and the ``no mark'' symbol $\nomark$. Let the set of marks be $\markset^\msp := \comgraphs\cup\Z^+\cup\{\nomark,\msl\}$. We mark each $v\in\leftpar$ by $\msl =: \markfunc^\msc(v)$ and each $a\in\rightpar$ by its community graph $\com_a =: \markfunc^\msc(a)$. Recall that an edge of the underlying $\BCM$ formed by $\msl$-half-edge $(v,i)$ and $\msr$-half-edge $(a,l)$ is labeled by $(i,l)$; we also mark this edge by the tuple $(i,l)$. Now $(\BCM,\markfunc^\msc)$, the community-marked $\BCM$, encodes all information necessary for constructing the $\RIGC$. The community graphs are given as the marks of $\msr$-vertices, and edge-marks encode the assigned community roles: if $\msl$-vertex $v$ is connected to $\msr$-vertex $a$ by an edge marked $(i,l)$, we know that $v$ takes on the community role of the vertex with label $l$ in $\com_a$. Thus 
the community projection operator \cref{ss:RIGC_def} can be naturally redefined as $\wih{\SP}: (\BCM,\markfunc^\msc) \mapsto \RIGC$. For $v\in\leftpar$, we write $\wih{\SP}: (\BCM,\markfunc^\msc,v) \mapsto (\RIGC,v)$ for the rooted version of the projection.

\smallskip
\paragraph*{\textbf{Constructing the local weak limit of the RIGC}}
Since we define $(\rCP,\groot)$ as the community-projection $\wih \SP$ of $(\BP_\msl,\markfunc^\msp,\otroot)$, we now introduce this marked $\BP$-family tree. 
Recall $(\BP_\msl,\otroot)$ from \cref{ss:construction_LWClim_BCM}; conditionally on this (possibly infinite) ordered tree, we now define the random mark function $\markfunc^\msp$, using the set of marks $\markset^\msp$ from above. 
We mark vertices in even generations by $\msl$, and vertices in odd generations by some $H\in\comgraphs$, determined as follows. Recall the family of conditional measures $(\bm\mu_{\cdot \countin k})_{k\in\Z^+}$ from \eqref{eq:def_mu_cond}, and that we denote the degree of $\vect{a}\in\BP_\msl$ by $d(\vect{a})$. Independently of everything else, we mark $\vect{a}$ according to the measure $\bm\mu_{\cdot \countin d(\vect{a})}$. 
We mark each edge $e$ by a tuple $(i,l)\in(\Z^+)^2$, and we determine $i$ and $l$ separately. Denote the endpoint of $e$ in an even generation by $\vect{v}$ and the endpoint in an odd generation by $\vect{a}$, that intuitively correspond to an $\msl$- and $\msr$-vertex, respectively. 
We think of $i$ and $l$ as the marks of the $\msl$- and $\msr$-half-edges incident to $\vect{v}$ and $\vect{a}$, respectively. We mark families of half-edges incident to the same vertex $\vect{u}$ jointly, so that each mark in $[d(\vect{u})]$ is used once, but independently of all other families. (In particular, the two coordinates in each edge mark are independent.) For $\vect{u} \neq \otroot$, we first mark the half-edge that is part of the edge connecting $\vect{u}$ to its parent, by a uniform mark $K\distr\Unif[d(\vect{u})]$. Recall that the family tree is ordered, thus we also have an ordering of half-edges incident to $\vect{u}$ that are part of edges connecting $\vect{u}$ to its children. We mark these half-edges by $[d(\vect{u})] \setminus \{K\}$ in increasing order. For the root, we mark all its half-edges by $[d(\otroot)]$ in increasing order, analogously. 

This defines the law of $\markfunc^\msp$ conditional on $(\BP_\msl,\otroot)$, and consequently the joint law $(\BP_\msl,\markfunc^\msp,\otroot)$, as well as the law of the $\wih\SP$-projection $(\rCP,\groot)$. 
It follows from the construction that $(\rCP,\groot)$ is a simple, locally finite rooted graph with countable (possibly infinite) vertex set $\vertices(\rCP) = \{\vect{v}\in\BP_\msl, \abs{\vect{v}} \text{ even}\}$. 
We obtain the following insight on the overlapping structure of the communities: each vertex $\vect{v}\in\vertices(\rCP)$ is part of exactly $d(\vect{v})$ communities, however, by the tree structure of $\BP_\msl$, any two of these communities only share $\vect{v}$ as a common vertex, i.e., the proposed local weak limit $\rCP$ has the \emph{single-overlap} property.

\subsection{The local weak convergence of the RIGC}
\label{ss:RIGC_LWC_proof}
Since $(\RIGC,V_n^\msl)$ is defined as the $\wih\SP$-projection of $(\BCM_n,\markfunc^\msc,V_n^\msl)$, and $(\rCP,\groot)$ is defined as the $\wih\SP$-projection of $(\BP_\msl,\markfunc^\msp,\otroot)$, it is a natural idea to prove LWC of the $\RIGC$ through the LWC of the underlying $\BCM$. However, as argued before, balls in the $\BCM$ generally do not map to balls in the $\RIGC$, thus we have introduced generalized neighborhoods in \cref{ss:neighborhoods_trees}. To obtain ball neighborhoods in the $\RIGC$, we show convergence of frequencies of \emph{generalized neighborhoods} in the underlying $\BCM$. 
To formalize such results, we generalize the notation $\ballevent_{\vect{T}}^r(v) = \{ B_r(\BCM_n,\markfunc^\msb,v) \orderedisom B_r(T,\markfunc_T,\otroot) \}$ from \eqref{eq:neighborhood_event}. 

Recall from \cref{ss:neighborhoods_trees} that ordered trees are defined by having an ordering of the children of any vertex. This can in fact be generalized for non-tree rooted graphs, if we allow vertices that close cycles to also have a second (third, etc) parent; they still obtain their Ulam-Harris label via the first parent. We call such graphs ordered graphs. Further recall, again from \cref{ss:neighborhoods_trees}, the ordering of the underlying $\BCM$ defined by the edge labels. 
Let ${\vect{H}} := (H,\markfunc_H,\otroot)$ be a \emph{finite} ordered marked graph with marks from $\markset^\msp$, and consider the correspondence between vertices of $\vect{H}$ and $(\BCM,\markfunc^\msc,v)$ with the same Ulam-Harris labels; if this correspondence is a rooted marked isomorphism, 
we say that $v$ has an $\vect{H}$-neighborhood and denote this event by $\ballevent_{\vect{H}}(\BCM_n,\markfunc^\msc,v)$. 
We define the corresponding event $\ballevent_{\vect{H}}(\BP_\msl,\markfunc^\msp,\otroot)$ for $(\BP_\msl,\markfunc^\msp,\otroot)$ analogously. 
We can now state the convergence of neighborhood frequencies in the $\BCM$ for generalized neighborhoods, as follows, with $V_n^\msl\distr\Unif[\leftpar]$ and $\P(\,\cdot \cond \bmatch_n)$ denoting conditional probability wrt $\bmatch_n$. 

\begin{lemma}[Convergence of general neighborhoods]
\label{lem:LWC_BCM_generalized}
Consider $(\BCM_n,\markfunc^\msc)$ under Assumption \ref{asmp:convergence}, and let $\vect{H} := (G,\markfunc_H,\otroot)$ denote an ordered graph marked from $\markset^\msp$. Then, as $n\to\infty$,
\beq\label{eq:BCM_raggedneighborhoods} \P\bigl( \ballevent_{\vect{H}}(\BCM_n,\markfunc^\msc,V_n^\msl) \bcond \bmatch_n \bigr) 
\toinp \P\bigl( \ballevent_{\vect{H}}(\BP_\msl,\markfunc^\msp,\otroot) \bigr). \eeq
\end{lemma}

\begin{proof}
The proof of \cref{lem:LWC_BCM_generalized} is analogous to \cref{prop:LWC_BCM_marked} and follows a first and second moment method.
\end{proof}

\Cref{lem:LWC_BCM_generalized} also implies, with the generalized meaning of the notion $\toLWCP$ from \cref{rem:LWC_nonunif},
\beq\label{eq:BCM_commarked_neighborhoods} 
(\BCM_n,\markfunc^\msc,V_n^\msl) \toLWCP (\BP_\msl,\markfunc^\msp,\otroot).
\eeq
The statement follows by applying \cref{lem:LWC_BCM_generalized} to the special case of ball neighborhoods, and completing the argument with a similar reasoning as in the proof of \cref{prop:LWC_BCM_marked} (see \cref{ss:LWC_BCM}). The convergence of frequencies of ordered trees implies convergence of frequencies of unordered trees. Since the support of the limiting measure only contains trees, by completeness of measure the convergence must hold for any neighborhood.

We are now ready to prove \cref{thm:LWC_RIGC} that asserts the local weak convergence in probability of the random intersection graph with communities: 

\begin{proof}[Proof of \cref{thm:LWC_RIGC}]
By \eqref{eq:def_LWCP}, we have to prove that for any $r\in\N$ and $(H,\groot)\in\graphs_\groot$,
\beq \P\bigl( B_r(\RIGC_n,V_n^\msl)\isom B_r(H,\groot) \bcond \bmatch_n \bigr) 
\toinp \P\bigl( B_r(\rCP,\groot)\isom B_r(H,\groot) \bigr). \eeq
As discussed above, we rely on the LWC of the $\BCM$, more precisely, \cref{lem:LWC_BCM_generalized}, to prove the above statement. To make a connection between neighborhoods in the $\RIGC$ and the underlying $\BCM$, we define the \emph{pre-images} of $B_r(H,\groot)$: all possible \emph{ordered} graphs $\vect{U_i} := (U_i,\markfunc_i,\otroot)$, $i\in \CI$, for some index set $\CI$, that are mapped into $B_r(H,\groot)$ by the community projection $\wih\SP$. 
Since pre-images contain \emph{all} information necessary to determine the $r$-neighborhood in the projection, we can decompose the following events as \emph{disjoint} unions:
\begin{gather} \bigl\{ B_r(\RIGC_n,V_n^\msl)\isom B_r(H,\groot) \bigr\} 
= \disjointunion_{i\in\CI} \,\ballevent_{\vect{U_i}}(\BCM_n,\markfunc^\msc,V_n^\msl), \\
\bigl\{ B_r(\rCP,\groot)\isom B_r(H,\groot) \bigr\} 
= \disjointunion_{i\in\CI} \,\ballevent_{\vect{U_i}}(\BP_\msl,\markfunc^\msp,\otroot).
\end{gather}

In the following, we present an intuitive partitioning of these unions, for which we need to understand the pre-images better. Note that when we only observe the $\RIGC$ (or $\rCP$) graph, the communities are not known, thus we consider each possibility for the communities to reconstruct every pre-image. By the properties of the projection, each edge belongs to a unique community, thus the communities intersecting $B_r(H,\groot)$ determine a \emph{partition}\footnote{A partition of a set is a family of subsets such that any two subsets are disjoint, and their union is the complete set. We refer to the subsets in the family as partition blocks.}
of all edges in this neighborhood (see \cref{fig:preimage} for an illustration). 
\begin{figure}[hbt]
	\centering
	\begin{subfigure}[t]{0.45\textwidth}
		\centering
		\includegraphics[width=0.9\textwidth]{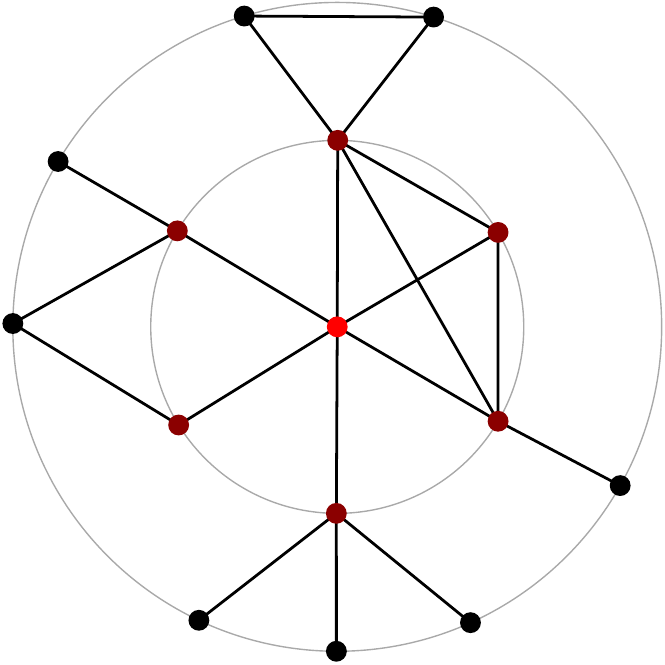}
		\caption{\emph{A $2$-neighborhood in the $\RIGC$}\\The circles and colors represent distance from the root (central red vertex).}
	\end{subfigure}\hspace{5px}
	\begin{subfigure}[t]{0.52\textwidth}
		\centering
		\includegraphics[width=0.8\textwidth]{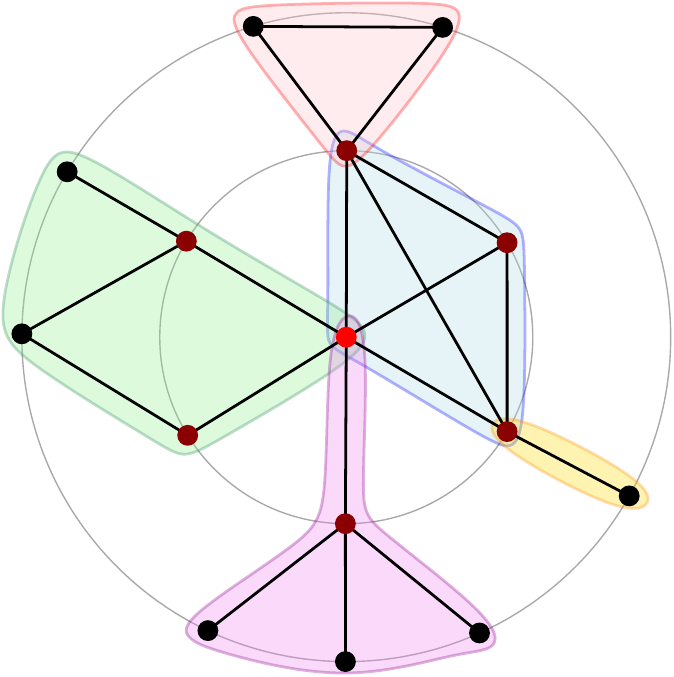}
		\caption{\emph{A possible edge-partition}\\ Partition blocks containing vertices on the boundary (black vertices) may correspond to ``unfinished'' communities that extend beyond this neighborhood.}
	\end{subfigure}\\
	\begin{subfigure}[b]{0.8\textwidth}
		\centering
		\includegraphics[width=0.8\textwidth]{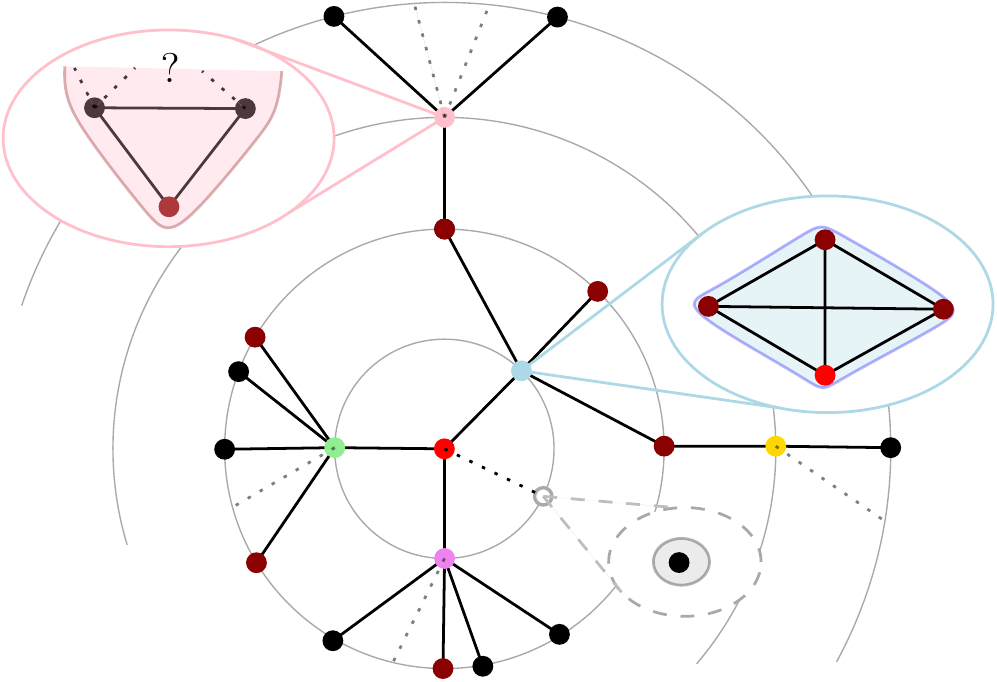}
		\caption{\emph{The ``pre-image'' corresponding to the above partition}\\ We preserved the color of each individual, and more or less its angular direction from the root. We can observe the bipartite structure and the change in graph distances. The pre-image is not unique: the community graph of ``unfinished'' communities (e.g.\ on the top left) is only partially known; individuals may be part of ``invisible'' $1$-member communities (e.g.\ on the bottom right).}
	\end{subfigure}
	\caption{A neighborhood and a possible pre-image}
	\label{fig:preimage}
\end{figure}

Let $\partitions$ denote the set of all edge-partitions of $B_r(H,\groot)$, which is \emph{finite}, since the total number of edges in $B_r(H,\groot)$ is finite. Note that for each pre-image $\vect{U_i}$, there is a unique corresponding edge-partition $\partt$; denote the (possibly empty) index set of pre-images corresponding to the edge-partition $\partt\in\partitions$ by $\CI(\partt)\subseteq\CI$. 
Thus
\begin{gather} 
\P\bigl( B_r(\RIGC_n,V_n^\msl)\isom B_r(H,\groot) \bcond \bmatch_n \bigr) 
= \sum_{\partt\in\partitions} \P\bigl( \disjointunion_{i\in\CI(\partt)} \ballevent_{\vect{U_i}}(\BCM_n,\markfunc^\msc,V_n^\msl) \bcond \bmatch_n \bigr), \\
\P\bigl( B_r(\rCP,\groot)\isom B_r(H,\groot) \bigr) 
= \sum_{\partt\in\partitions} \P\bigl( \disjointunion_{i\in\CI(\partt)} \ballevent_{\vect{U_i}}(\BP_\msl,\markfunc^\msp,\otroot) \bigr).
\end{gather}
Then, by $\abs{\partitions}<\infty$, it is sufficient to prove that for any fixed $\partt\in\partitions$,
\beq\label{eq:preimages_onepartition_conv} \P_{V_n^\msl}\bigl( \disjointunion_{i\in\CI(\partt)} \SB_{\vect{U}_i}(\BCM_n,\markfunc^\msp,V_n^\msl) \bcond \bmatch_n \bigr) 
\toinp \P\bigl( \disjointunion_{i\in\CI(\partt)} \,\SB_{\vect{U}_i}(\BP_\msl,\markfunc^\msp,\otroot) \bigr). \eeq
Clearly, this convergence is trivial if $\CI(\partt)$ is empty, thus in the following we consider edge-partitions $\partt$ such that $\CI(\partt)$ is not empty. We argue why the set $(\vect{U_i})_{i\in\CI(\partt)}$ contains several, in fact possibly infinitely many elements. To construct any pre-image (see \cref{fig:preimage}), more information is necessary, that is captured \emph{neither} in the neighborhood $B_r(H,\groot)$ \emph{nor} in the partition $\partt$. First, one-member communities do not produce edges, and thus remain ``invisible'' in the community-projection. Second, each partition block $E\in\partt$ containing an edge adjacent to a vertex in generation $r$ may correspond to an ``unfinished community'' that intersects the ball $B_r(H,\groot)$, but  is not fully contained in it.

We rely on a truncation argument so that we can focus on a finite subset of $\CI(\partt)$. We first show that each pre-image $\vect{U_i}$ is contained in a ball of radius $2r+1$. Note that distances in the pre-images are the largest possible when each edge forms a partition block by itself, then an $\msr$-vertex representing an edge between generation $r$ individuals can reach the maximum distance $2r + 1$. Now we can use that degrees in a finite ball are \emph{tight}, 
by the local weak convergence of the community-marked $\BCM$ \eqref{eq:BCM_commarked_neighborhoods}.\footnote{Since $\ind_{\{ \text{maximal degree in $r$-ball of $\BCM$} > K \}}$ is a bounded and continuous functional on $(\graphs_\groot,\dloc)$ (see \cref{ss:LWC_def}), by \eqref{eq:def_LWCphi}, $\P\bigl( \text{maximal degree in $r$-ball of $\BCM$} > K \bcond \bmatch_n \bigr) \toinp \P\bigl( \text{maximal degree in $r$-ball of $\BP_\msl$} > K \bigr)$, which vanishes as $K\to\infty$, since the degrees in $\BP_\msl$ that follow distributions $D^\msl$, $\wit D^\msr$ and $\wit D^\msl$ are tight.} 
Thus, for any $\eps>0$, there exists $K=K(\eps)<\infty$ such that
\begin{subequations}\label{eq:preimage_truncation}
\begin{gather} \P\bigl( \max \{ \bdeg(v): v\in B_{2r+1}(\BP_\msl,\markfunc^\msp,\otroot) \} > K \bigr) < \eps/6, \\
\P\bigl( \max \{ \bdeg(v): v\in B_{2r+1}(\BCM_n,\markfunc^\msc,V_n^\msl) \} > K \bcond \bmatch_n \bigr) < \eps/3 \text{ whp as } n\to\infty.
\end{gather}
\end{subequations}
Define the index set $\CI\bigl(\partt,{\leq}K\bigr) = \bigl\{i\in\CI(\partt): \max\{\bdeg(v): v\in\vertices(\vect{U_i})\} \leq K \bigr\}$. As each $\vect{U_i}$ for $i\in\CI(\partt,{\leq}K)$ has depth (maximal degree from the root) bounded by $2r+1$ and degree bounded by $K$, necessarily $\CI\bigl(\partt,{\leq}K\bigr)$ is finite. Denote $\CI(\partt,{>}K) := \CI(\partt)\setminus\CI(\partt,{\leq}K)$. 
By the triangle inequality,
\beq\label{eq:RIGC_LWC_1} \bal &\Bigl\lvert \P\bigl( {\disjointunion_{i\in\CI(\partt)}} \ballevent_{\vect{U_i}}(\BCM_n,\markfunc^\msp,V_n^\msl) \bcond \bmatch_n \bigr) 
- \P\bigl( {\disjointunion_{i\in\CI(\partt)}} \,\ballevent_{\vect{U_i}}(\BP_\msl,\markfunc^\msp,\otroot) \bigr) \Bigr\rvert \\
&\leq \P\bigl( {\disjointunion_{i\in\CI(\partt,{>}K)}} \ballevent_{\vect{U_i}}(\BCM_n,\markfunc^\msp,V_n^\msl) \bcond \bmatch_n \bigr) 
+ \P\bigl( {\disjointunion_{i\in\CI(\partt,{>}K)}} \,\ballevent_{\vect{U_i}}(\BP_\msl,\markfunc^\msp,\otroot) \bigr) \\
&\phantom{{}\leq{}}+ \sum_{i\in\CI(\partt,{\leq}K)} \bigl\lvert \P\bigl( \ballevent_{\vect{U_i}}(\BCM_n,\markfunc^\msp,V_n^\msl) \bcond \bmatch_n \bigr) 
- \P\bigl( \ballevent_{\vect{U_i}}(\BP_\msl,\markfunc^\msp,\otroot) \bigr) \bigr\rvert.
\eal \eeq
We study the finite sum first. By \cref{lem:LWC_BCM_generalized}, for each $i\in\CI(\partt,{\leq}K)$, whp
\beq\label{eq:RIGC_LWC_2} \bigl\lvert \P\bigl( \ballevent_{\vect{U_i}}(\BCM_n,\markfunc^\msp,V_n^\msl) \bcond \bmatch_n \bigr) 
- \P\bigl( \ballevent_{\vect{U_i}}(\BP_\msl,\markfunc^\msp,\otroot) \bigr) \bigr\rvert
\leq \frac{\eps}{2 \abs{\CI(\partt,{\leq}K)}}. \eeq
Now, we look at the first two terms on the rhs of \eqref{eq:RIGC_LWC_1}. By the definition of the set $\CI(\partt,{>}K)$,
\begin{subequations}\label{eq:RIGC_LWC_3}
\begin{align}
\begin{split} &\P\bigl( \disjointunion_{i\in\CI(\partt,{>}K)} \ballevent_{\vect{U_i}}(\BCM_n,\markfunc^\msp,V_n^\msl) \bcond \bmatch_n \bigr) \\
&\phantom{{}={}}\leq \P\bigl( \max \{ \bdeg(v): v\in B_{2r}(\BCM_n,\markfunc^\msc,V_n^\msl) \} > K \bcond \bmatch_n \bigr), \end{split} \\
\begin{split} \P\bigl( \disjointunion_{i\in\CI(\partt,{>}K)} \,\ballevent_{\vect{U_i}}(\BP_\msl,\markfunc^\msp,\otroot) \bigr) 
\leq \P\bigl( \max \{ \bdeg(v): v\in B_{2r}(\BP_\msl,\markfunc^\msp,\otroot) \} > K \bigr), \end{split}
\end{align}
\end{subequations}
which both are whp smaller than $\eps/4$ by \eqref{eq:preimage_truncation}. 
Combining (\ref{eq:preimage_truncation}-\ref{eq:RIGC_LWC_3}), we obtain that for any $\eps>0$, whp 
\beq \Bigl\lvert \P\bigl( \disjointunion_{i\in\CI(\partt)} \ballevent_{\vect{U_i}}(\BCM_n,\markfunc^\msp,V_n^\msl) \bcond \bmatch_n \bigr) 
- \P\bigl( \disjointunion_{i\in\CI(\partt)} \,\ballevent_{\vect{U_i}}(\BP_\msl,\markfunc^\msp,\otroot) \bigr) \Bigr\rvert \leq \eps, \eeq
which is equivalent to \eqref{eq:preimages_onepartition_conv}. Since we have previously reduced \cref{thm:LWC_RIGC} to this statement, this concludes the proof of \cref{thm:LWC_RIGC}.
\end{proof}

\FloatBarrier

\subsection{Degrees and clustering}
\label{ss:deg_clust}

Recall the definition of $(\rCP,\groot)$, the local weak limit of the $\RIGC$, as the $\wih\SP$-projection of $(\BP_\msl,\markfunc^\msp,\otroot)$ from \cref{ss:construction_LWClim_RIGC}. By this construction, it is clear that $D^\msp$ (see \eqref{eq:def_Dproj}) and $\zeta$ (see \eqref{eq:def_zeta}) describe the degree and local clustering coefficient of $\groot \in \rCP$, respectively. Further recall the empirical degree $D_n^\msp$ (see (\ref{eq:def_Dnproj}-\ref{eq:def_Fnproj})) and empirical local clustering $\zeta_n$ (see (\ref{eq:def_clus_empirical}-\ref{eq:def_zeta})). 
By $(\RIGC_n,V_n^\msl)\toLWCP(\rCP,\groot)$, it is intuitive that $D_n^\msp\toindis D^\msp$ and $\zeta_n \toindis \zeta$. We complete the formal proof of the stronger statements \eqref{eq:fn_degree_conv} and \eqref{eq:clustering_conv} below.

\begin{proof}[Proof of \cref{cor:deg,cor:clustering}]
Recall that $\P(\,\cdot \cond \bmatch_n)$ denotes conditional probability wrt $\bmatch_n$ and $\E[\,\cdot \cond \bmatch_n]$ denotes the corresponding conditional expectation. 
Further, denote by $\P_\groot$ and $\E_\groot$ the probability measure of $(\rCP,\groot)$ and the corresponding expectation. For arbitrary \emph{fixed} $x\in\R$, we define the functionals
\beq \bal \varphi_x, \psi_x, \overline{\psi}_x: \graphs_\groot\to\{0,1\},
\quad &\varphi_x(G,\groot) := \ind_{\{\deg(\groot)\leq x\}}, \\
&\psi_x(G,\groot):= \ind_{\{\clustering(\groot)\leq x\}},
\quad \overline{\psi}_x(G,\groot):= \ind_{\{\clustering(\groot) < x\}}.
\eal \eeq
Clearly, all three functionals are bounded, and only depend on a finite neighborhood of $\groot$, thus they are continuous in the metric space $(\graphs_\groot,\dloc)$ (see \cref{ss:LWC_def}). 
Note that we can express the (empirical) cdfs from \eqref{eq:def_Fnproj}, \eqref{eq:def_Dproj}, \eqref{eq:def_clus_empirical} and \eqref{eq:def_zeta} respectively as
\beq\bal F_n^\msp(x) &= \E\bigl[ \varphi_x(\RIGC_n,V_n^\msl) \bcond \bmatch_n \bigr],
& F^\msp(x) &= \E_\groot \bigl[ \varphi_x(\rCP,\groot) \bigr], \\
F_n^\zeta(x) &= \E\bigl[ \psi_x(\RIGC_n,V_n^\msl) \bcond \bmatch_n \bigr],
& F^\zeta(x) &= \E_\groot \bigl[ \psi_x(\rCP,\groot) \bigr].
\eal\eeq
Let us denote $f(x-):=\lim_{\eps\searrow0}f(x-\eps)$. Then also
\begin{gather} F^\zeta(x-) 
= \lim_{\eps\searrow0} \P_\groot \bigl( \zeta \leq x-\eps \bigr)
= \P_\groot \bigl( \zeta < x \bigr) 
= \E_\groot \bigl[ \overline{\psi}_x \bigl((\rCP,\groot)\bigr) \bigr], \\
F_n^\zeta(x-) = \E\bigl[ \overline{\psi}(\RIGC_n,V_n^\msl) \bcond \bmatch_n \bigr]. 
\end{gather}
\Cref{thm:LWC_RIGC} asserts $(\RIGC_n,V_n^\msl) \toLWCP (\rCP,\groot)$, thus, using the equivalent definition of LWC in probability \eqref{eq:def_LWCPphi}, for any \emph{fixed} $x\in\R$, as $n\to\infty$,
\beq\label{eq:cor_degclust_1} F_n^\msp(x) \toinp F^\msp(x), 
\quad F_n^\zeta(x) \toinp F^\zeta(x), 
\quad F_n^\zeta(x-) \toinp F^\zeta(x-). \eeq
That is, we have established pointwise convergence of the cdfs. In the following, we show that it in fact implies convergence in sup-norm as well, by a truncation and discretization argument for the degrees and clustering, respectively, starting with the degrees. As $D_n^\msp$ and $D^\msp$ are $\N$-valued random variables,
\beq \sup_{x\in\R}\, \bigl\lvert F_n^\msp(x) - F^\msp(x) \bigr\rvert 
= \sup_{k\in\N}\, \bigl\lvert F_n^\msp(k) - F^\msp(k) \bigr\rvert. \eeq
Choose $K=K(\eps)\in\N$ minimal such that $F^\msp(K)>1-\tfrac13\eps$. Then by \eqref{eq:cor_degclust_1}, $F_n^\msp(K) > 1-\tfrac23\eps$ whp. By the triangle inequality and the monotonicity of distribution functions, whp for all $k\geq K$,
\beq \bal \bigl\lvert F_n^\msp(k) - F^\msp(k) \bigr\rvert 
\leq 1 - F_n^\msp(k) + 1 - F(k) 
\leq 1 - F_n^\msp(K) + 1 - F(K) 
< \tfrac23\eps + \tfrac13\eps = \eps. \eal \eeq
That is, $\max_{k \geq K} \bigl\lvert F_n^\msp(k) - F^\msp(k) \bigr\rvert \leq \eps$ whp. By \eqref{eq:cor_degclust_1}, clearly the \emph{finite} maximum $\max_{k < K} \bigl\lvert F_n^\msp(k) - F^\msp(k) \bigr\rvert \leq \eps$ whp as well. Combining the above, we conclude that 
\beq \norm{F_n^\msp - F^\msp}_\infty 
= \max\Bigl\{ \max_{k<K}\, \bigl\lvert F_n^\msp(k) - F(k) \bigr\rvert ,\; \max_{k<K}\, \bigl\lvert F_n^\msp(k) - F(k) \bigr\rvert \Bigr\}
\leq \eps \text{ whp}. \eeq
This is equivalent to the convergence in probability in \eqref{eq:fn_degree_conv}, and concludes the proof of \cref{cor:deg}. We move on to study the distribution of the local clustering. As $\zeta_n$ and $\zeta$ take potentially all rational values in $[0,1]$, a different approach is required. First, we write
\beq \sup_{x\in\R}\, \bigl\lvert F_n^\zeta(x) - F^\zeta(x) \bigr\rvert 
= \sup_{x\in[0,1)}\, \bigl\lvert F_n^\zeta(x) - F^\zeta(x)\bigr\rvert, \eeq
and in the following, we discretize this supremum. Since $F^\zeta$ is a cdf, consequently non-decreasing and taking values between $0$ and $1$, there must exist $K=K(\eps)<\infty$ and a \emph{finite} sequence $0=z_0<z_1<\ldots<z_K=1$ such that for all $k=0,1,\ldots,K-1$,
\beq\label{eq:cor_degclust_2} \bigl\lvert F^\zeta(z_{k+1}-) - F^\zeta(z_k) \bigr\rvert < \eps/3. \eeq
Define the ``good event'' as
\beq\label{eq:cor_degclust_goodevent} \goodevent_n := \bigcap_{k=0}^{K-1} \bigl( \bigl\{ \lvert F_n^\zeta(z_k) - F^\zeta(z_k) \rvert < \eps/3 \bigr\} 
\cap \bigl\{ \lvert F_n^\zeta(z_{k+1}-) - F^\zeta(z_{k+1}-) \rvert < \eps/3 \bigr\} \bigr). \eeq
By \eqref{eq:cor_degclust_1}, each event on the rhs happens whp, thus the finite intersection $\goodevent_n$ also happens whp. 
\emph{On the event $\goodevent_n$,} using (\ref{eq:cor_degclust_2}-\ref{eq:cor_degclust_goodevent}) and that $F_n^\zeta$ is non-decreasing, we bound the empirical cdf for any $k$ and all $x\in[z_k,z_{k+1})$ as
\begin{subequations}\label{eq:cor_degclust_3}
\begin{gather}  
F_n^\zeta(x) 
\geq F_n^\zeta(z_k) 
> F^\zeta(z_k) - \tfrac13\eps, \\
F_n^\zeta(x) 
\leq F_n^\zeta(z_{k+1}-) 
< F^\zeta(z_{k+1}-) + \tfrac13\eps
< F^\zeta(z_k) + \tfrac23\eps.
\end{gather}
\end{subequations}
Using \eqref{eq:cor_degclust_2} and that $F^\zeta$ is non-decreasing, we bound the limiting cdf for any $k$ and all $x\in[z_k,z_{k+1})$ as 
\beq\label{eq:cor_degclust_4} F^\zeta(x) 
\geq F^\zeta(z_k), 
\qquad F^\zeta(x) 
\leq F^\zeta(z_{k+1}-) 
< F^\zeta(z_k) + \tfrac13\eps. \eeq
Combining (\ref{eq:cor_degclust_3}-\ref{eq:cor_degclust_4}) through the triangle inequality yields that, on the event $\goodevent_n$, 
\beq \sup_{x\in[z_k,z_{k+1})} \, \bigl\lvert F_n^\zeta(x) - F^\zeta(x) \bigr\rvert < \eps, \eeq
for all $k=0,1,\ldots,K-1$. Recall that $\cup_{k=0}^{K-1} [z_k,z_{k+1}) = [0,1)$. Consequently on the event $\goodevent_n$, which happens whp,
\beq \norm{F_n^\zeta - F^\zeta}_\infty 
= \sup_{x\in[0,1)}\, \bigl\lvert F_n^\zeta(x) - F^\zeta(x) \bigr\rvert < \eps, \eeq
which is equivalent to \eqref{eq:clustering_conv}. This concludes the proof of \cref{cor:clustering}.
\end{proof}

\subsection{The overlapping structure}
\label{ss:overlaps}
In this section, we prove \cref{prop:number_of_overlaps} and \cref{thm:overlaps} on the typical number and size of overlaps in the $\RIGC$ model. First, we prove \cref{thm:overlaps} (\ref{perspective:vx}-\ref{perspective:group}), that follow directly from \cref{thm:LWC_RIGC}, then prove \cref{prop:number_of_overlaps} and \cref{thm:overlaps} \cref{perspective:global}, which also require the second moment condition \eqref{cond:secondmom} and a slightly different approach. We make use of the following notation. 
Recall that $V_n^\msl \sim \Unif[\leftpar]$, $V_n^\msr \sim \Unif[\rightpar]$ and $V_n^\msb \sim \Unif[\leftpar\cup\rightpar]$. Further recall that $\P(\,\cdot \cond \bmatch_n)$ denotes conditional probability wrt $\bmatch_n$ (i.e., conditionally on the graph realization), and $\E[\,\cdot \cond \bmatch_n]$ denotes the corresponding conditional expectation (i.e., partial average over the choice of the uniform vertex). 

\subsubsection{Proof of Theorem \ref{thm:overlaps} (\ref{perspective:vx}-\ref{perspective:group})}
\label{sss:overlaps_local}
An overlap of size (at least) two happens in the $\RIGC$ exactly when there are two individuals that are part of two groups together. In the underlying $\BCM$, these two individuals and two groups form a $K_{2,2}$ complete bipartite graph, which we can also look at as a $4$-cycle. 
Thus in the following, we study $4$-cycles through \emph{typical}, i.e., uniformly chosen, vertices in the $\BCM$. 
Recall the notion of local weak convergence from \cref{ss:LWC_def}, and in particular the set $\graphs_\groot$ of rooted graphs and the metric $\dloc$ defined on it. We define the functional $\indc4$ on $\graphs_\groot$ as the indicator that there is a $4$-cycle containing the root. Note that $\indc4 \in \Phi$ (see \eqref{eq:def_testfunctionals}): it is clearly bounded, and since it only depends on the $2$-neighborhood of the root, also continuous. 
\Cref{thm:LWC_BCM} 
the LWC in probability of the $\BCM$, thus by the equivalent definition \eqref{eq:def_LWCPphi},
\beq\label{eq:overlaps_proof_1} \E\bigl[ \indc4(\BCM_n,V_n^\msb) \bcond \bmatch_n\bigr] 
\toinp \E\bigl[ \indc4(\BP_\mss,\otroot) \bigr] = 0. \eeq
Recall (\ref{eq:def_intersection_size}-\ref{eq:def_overlapset}). We can rewrite the lhs of \eqref{eq:typical_overlap_vx} and \eqref{eq:typical_overlap_group} respectively as
\begin{gather}
\P\bigl( \exists \{a,b\} \in\overlapset_2 :\, V_n^\msl \comrole \com_a, V_n^\msl \comrole \com_b \bcond \bmatch_n \bigr) 
= \E\bigl[ \indc4 (\BCM_n,V_n^\msl) \bcond \bmatch_n \bigr], \\
\P\bigl( \exists b\in\neighbors(V_n^\msr):\, \intersection(V_n^\msr,b) \geq 2 \bcond \bmatch_n \bigr) 
= \E\bigl[ \indc4 (\BCM_n,V_n^\msr) \bcond \bmatch_n \bigr].
\end{gather}
By the definition of the partial average,
\beq\label{eq:partialaverage_bound} \bal &\E\bigl[ \indc4 (\BCM_n,V_n^\msl) \bcond \bmatch_n \bigr]
= \frac{1}{N_n} \sum_{v\in\leftpar} \indc4 (\BCM_n,v) \\
&\leq \frac{N_n+M_n}{N_n} \frac{1}{N_n+M_n}
\sum_{v\in\leftpar\cup\rightpar} \indc4 (\BCM_n,v) 
= \frac{N_n+M_n}{N_n} \cdot
\E\bigl[ \indc4 (\BCM_n,V_n^\msb) \bcond \bmatch_n \bigr],
\eal \eeq
and analogously,
\beq \E\bigl[ \indc4 (\BCM_n,V_n^\msr) \bcond \bmatch_n \bigr] 
\leq \frac{N_n+M_n}{M_n} \,
\E\bigl[ \indc4 (\BCM_n,V_n^\msb) \bcond \bmatch_n \bigr]. \eeq
By \cref{rem:asmp_consequences} \cref{cond:partition_ratio}, as $n\to\infty$,
\beq\label{eq:overlaps_proof_2} (N_n+M_n)/N_n \to 1+\gamma < \infty, 
\quad (N_n+M_n)/M_n \to (1+\gamma)/\gamma < \infty. \eeq
Combining (\ref{eq:overlaps_proof_1}-\ref{eq:overlaps_proof_2}) yields \eqref{eq:typical_overlap_vx} and \eqref{eq:typical_overlap_group}, as required. This concludes the proof of \cref{thm:overlaps} (\ref{perspective:vx}-\ref{perspective:group}).
\proofends

\subsubsection{Proof of Proposition \ref{prop:number_of_overlaps}}
\label{sss:overlaps_number}
As before, we want to reduce \cref{prop:number_of_overlaps} to local weak convergence. Thus, we define the functional $\varphi$ on $\graphs_\groot$ (see \cref{ss:LWC_def} for the notation) that counts the number of vertices at graph distance $2$ from the root, i.e., for $(G,\groot)\in\graphs_\groot$,
\beq \varphi(G,\groot) := \abs{\boundary_2(G,\groot)}. \eeq
Recall (\ref{eq:def_intersection_size}-\ref{eq:def_overlap_neighbor}). We can rewrite the lhs of \eqref{eq:number_of_overlaps} as
\beq\label{eq:numberoverlaps_1} \begin{split} \frac{2 \abs{ \overlapset_1 }}{M_n} 
= \E\bigl[ \abs{\neighbors(V_n^\msr)} \bcond \bmatch_n \bigr] 
= \E \bigl[ \varphi(\BCM_n,V_n^\msr) \bcond \bmatch_n \bigr]. \end{split} \eeq
Recall $(\BP_\msr,\otroot)$ from \cref{ss:construction_LWClim_BCM} and note that
\beq\label{eq:numberoverlaps_2} \E\bigl[ \varphi(\BP_\msr,\otroot) \bigr] = \E[D^\msr] \E[\wit D^\msl], \eeq
which is exactly the proposed limit of \eqref{eq:numberoverlaps_1}. 
It is tempting to conclude the result by \cref{prop:LWC_BCM_marked} and \eqref{eq:def_LWCphi} as before, however, \eqref{eq:def_LWCphi} is not applicable, since $\varphi$ is \emph{not} in $\Phi$ (see \eqref{eq:def_testfunctionals}). While $\varphi$ only depends on a finite neighborhood of the root and is necessarily continuous, it is \emph{not} bounded. 
Instead, we rely on a truncation argument to establish 
convergence in probability, 
in the following form: 
for any fixed $\eps,\delta>0$ and $n$ large enough (possibly depending on $\eps$ and $\delta$),
\beq\label{eq:numberoverlaps_3} \P\Bigl( \bigl\lvert\, \E \bigl[ \varphi(\BCM_n,V_n^\msr) \bcond \bmatch_n \bigr] - \E\bigl[ \varphi(\BP_\msr,\otroot) \bigr] \,\bigr\rvert > \eps \Bigr) < \delta. \eeq
With some $K=K(\eps,\delta)\in\N$ to be specified later, we decompose
\begin{subequations}
\begin{align} \nonumber &\P\Bigl( \bigl\lvert\, \E \bigl[ \varphi(\BCM_n,V_n^\msr) \bcond \bmatch_n \bigr] - \E\bigl[\varphi(\BP_\msr,\otroot)\bigr] \,\bigr\rvert > \eps \Bigr) \\
\label{eq:overlaps_decomp1} &\leq \P\Bigl( \bigl\lvert\, \E \bigl[ \varphi(\BCM_n,V_n^\msr) \bcond \bmatch_n \bigr] - \E \bigl[ \varphi(\BCM_n,V_n^\msr) \midmin K \bcond \bmatch_n \bigr] \,\bigr\rvert > \eps/3 \Bigr) \\
\label{eq:overlaps_decomp2} &\phantom{{}\leq{}}+ \P\Bigl( \bigl\lvert\, \E \bigl[ \varphi(\BCM_n,V_n^\msr) \midmin K \bcond \bmatch_n \bigr] - \E\bigl[\varphi(\BP_\msr,\otroot) \midmin K \bigr] \,\bigr\rvert > \eps/3 \Bigr) \\
\label{eq:overlaps_decomp3} &\phantom{{}\leq{}}+ \P\Bigl( \bigl\lvert\, \E\bigl[\varphi(\BP_\msr,\otroot) \midmin K \bigr]  - \E\bigl[\varphi(\BP_\msr,\otroot)\bigr] \,\bigr\rvert > \eps/3 \Bigr).
\end{align}
\end{subequations}
We study \eqref{eq:overlaps_decomp2} first. 
Recall $\markset^\msb=\{\msl,\msr\}$ and denote $x\midmin y := \min\{x,y\}$. We define the \emph{bounded} functional $\varphi_K$ on $\graphs_\groot(\markset^\msb)$ as 
\beq\label{eq:def_neighbors_truncation} \varphi_K(G,\markfunc,\groot) := \ind_{\{\markfunc(\groot)=\msr\}} \cdot 
\bigl(\varphi(G,\groot) \midmin K\bigr). \eeq
Note that $\varphi_K$ is also continuous, as it only depends on a finite neighborhood of the root. 
By properties of conditional expectation, we can now rewrite
\beq\label{eq:overlaps_proof_3} \bal \E \bigl[ \varphi_K(\BCM_n,\markfunc^\msb,V_n^\msb) \bcond \bmatch_n \bigr] 
&= \P\bigl( V_n^\msb \in \rightpar \bigr) \E\bigl[ \varphi(\BCM_n,V_n^\msb) \bcond \bmatch_n, V_n^\msb\in\rightpar \bigr] \\
&= \frac{M_n}{N_n+M_n} \E \bigl[ \varphi(\BCM_n,V_n^\msr) \midmin K \bcond \bmatch_n \bigr], \eal \eeq
and analogously,
\beq\label{eq:overlaps_proof_4} \E\bigl[ \varphi_K(\BP_\mss,\markfunc^\mss,\otroot) \bigr] 
= \frac{\gamma}{1+\gamma} \E\bigl[ \varphi(\BP_\msr,\otroot)) \midmin K \bigr]. \eeq
By \cref{prop:LWC_BCM_marked} and \eqref{eq:def_LWCPphi}, the lhs of \eqref{eq:overlaps_proof_3} converges in probability to the lhs of \eqref{eq:overlaps_proof_4} and $M_n/(N_n+M_n)\to\gamma/(1+\gamma)$ by \cref{rem:asmp_consequences} \cref{cond:partition_ratio}. Necessarily
$\E \bigl[ \varphi(\BCM_n,V_n^\msr) \midmin K \bcond \bmatch_n \bigr] 
\toinp \E\bigl[ \varphi(\BP_\msr,\otroot) \midmin K \bigr]$, 
or equivalently, for any $\eps,\delta>0$ fixed and $n$ large enough,
\beq\label{eq:overlaps_proof_term1} \P\Bigl( \bigl\lvert\, \E\bigl[ \varphi(\BCM_n,V_n^\msr) \midmin K \bcond \bmatch_n \bigr] 
- \E\bigl[ \varphi(\BP_\msr,\otroot) \midmin K \bigr] \,\bigr\rvert 
> \eps/3 \Bigr) < \delta/2. \eeq
Next, we study \eqref{eq:overlaps_decomp3}. Recall $D^\msl$ and $D^\msr$ from \cref{asmp:convergence} \cref{cond:limit_ldeg} and \cref{cond:limit_rdeg}, and recall \eqref{eq:def_sizebiasing}. 
By the definition of $\BP_\msr$ in \cref{ss:construction_LWClim_BCM}, $\varphi(\BP_\msr,\otroot) \eqindis \sum_{i=1}^{D^\msr} \wit D_{(i)}^\msl$, where $\wit D_{(i)}^\msl$ are iid copies of $\wit D^\msl$. Under the second moment condition \eqref{cond:secondmom}, $\E[\wit D^\msl]<\infty$, thus $\E\bigl[ \varphi(\BP_\msr,\otroot) \bigr] < \infty$. We now choose and fix $K=K(\eps,\delta)$ large enough so that
\beq\label{eq:overlaps_proof_term2} 
0 \leq \E\bigl[ \varphi(\BP_\msr,\otroot) \bigr] - \E\bigl[ \varphi(\BP_\msr,\otroot) \midmin K \bigr] < (\eps/3) \midmin (\eps\delta/18). \eeq
Consequently, the probability in \eqref{eq:overlaps_decomp3} equals $0$. 
Finally, we prove below that for large enough $n$, 
\beq\label{eq:overlaps_proof_term3} \P\Bigl( 
\bigl\lvert\, \E\bigl[ \varphi(\BCM_n,V_n^\msr) \midmin K \bcond \bmatch_n \bigr] 
- \E\bigl[ \varphi(\BCM_n,V_n^\msr) \bcond \bmatch_n \bigr] \,\bigr\rvert 
> \eps/3 \Bigr) < \delta/2. \eeq
Combining \eqref{eq:overlaps_proof_term1}, \eqref{eq:overlaps_proof_term2} and \eqref{eq:overlaps_proof_term3} yields \eqref{eq:numberoverlaps_3}, which concludes the proof of \cref{prop:number_of_overlaps} subject to \eqref{eq:overlaps_proof_term3}. 
We now prove \eqref{eq:overlaps_proof_term3} using a first moment method. The advantage of this method is that taking expectation removes the conditioning on $\bmatch_n$, and we better understand the distribution of $\varphi(\BCM_n,V_n^\msr)$ with \emph{both} sources of randomness, i.e., $V_n^\msr$ and $\bmatch_n$. Using that 
$\varphi(\BCM_n,V_n^\msr) \midmin K \leq \varphi(\BCM_n,V_n^\msr)$, we compute
\beq\label{eq:overlaps_proof_5} \bal &\E\Bigl[ \bigl\lvert\, \E\bigl[ \varphi(\BCM_n,V_n^\msr) \midmin K \bcond \bmatch_n \bigr] 
- \E\bigl[ \varphi(\BCM_n,V_n^\msr) \bcond \bmatch_n \bigr] \,\bigr\rvert \Bigr] \\
&= \E\Bigl[ \E\bigl[ \varphi(\BCM_n,V_n^\msr) \bcond \bmatch_n \bigr] 
- \E\bigl[ \varphi(\BCM_n,V_n^\msr) \midmin K \bcond \bmatch_n \bigr] \Bigr] \\
&= \E \bigl[ \varphi(\BCM_n,V_n^\msr) \bigr] 
- \E \bigl[ \varphi(\BCM_n,V_n^\msr) \midmin K \bigr]. \eal \eeq
Under the joint measure of $\bmatch_n$ and $V_n^\msr$, 
the following \emph{stochastic domination},\footnote{We do not make this argument explicit, but note that $\abs{\boundary_2(\BCM_n,V_n^{\msr})}$ is the largest possible when all community roles of the uniform community $V_n^\msr$ are taken by different individuals, and all further memberships of these individuals are taken in different communities.} denoted by $\preceq$, holds:
\beq\label{eq:overlaps_proof_6} \varphi(\BCM_n,V_n^\msr) 
= \abs{\boundary_2(\BCM_n,V_n^\msr)} 
\preceq \sum_{i=1}^{D_n^\msr} (d_{\pi(i)}^\msl - 1), \eeq
where $D_n^\msr$ was defined in \eqref{eq:def_ldeg_rdeg} and $(d_{\pi(i)}^\msl)_{i\leq N_n}$ denotes a \emph{size-biased reordering} (defined below) of $\bitd^{\msl}$, independently of $D_n^\msr$. We define the size-biased reordering by the random permutation $(\pi(i))_{i\leq N_n}$ as follows. Denote the set of already chosen indices by $\Pi_1^0 := \{\,\}$, $\Pi_1^i := \{ \pi(1),\ldots,\pi(i) \}$ for $i>0$, then for $i=0,1,\ldots,N_n-1$,
\beq\label{eq:def_sizebiased_reordering} \P\bigl( \pi(i+1) = k \bcond \Pi_1^i \bigr) = 
\begin{cases}
0 & \text{for $k\in\Pi_1^i$},\\
\displaystyle \frac{d_k^\msl}{\sum_{j\not\in\Pi_1^i} d_j^\msl} & \text{otherwise}.
\end{cases}
\eeq
In the following, we bound the expectation of the rhs of \eqref{eq:overlaps_proof_6}. Clearly, $D_n^\msr \leq \dmax{\msr}$ almost surely, and by \cref{rem:asmp_consequences} \cref{cond:dmax}, $\dmax{\msr} = o(\he_n)$, hence for any $\delta'>0$, for $n$ large enough, $D_n^\msr \leq \dmax{\msr} \leq \delta' N_n$ almost surely. We will choose $\delta'$ later, and now study $\E\bigl[ d_{\pi(i)}^\msl \bigr]$ for $i\leq \delta' N_n$. Let $d_{(k)}^\msl$ denote the $k\ith$ largest element of $\bitd^\msl$ (with ties broken arbitrarily). 
\begin{gather}
\label{eq:overlaps_proof_7} \E\bigl[ d_{\pi(1)}^\msl \bigr] 
= \frac{\sum_{v\in[N_n]} (d_v^\msl)^2}{\sum_{v\in[N_n]} d_v^\msl }
= \E\bigl[ D_n^{\msl,\star} \bigr]
= \E\bigl[ \wit D_n^\msl \bigr] + 1 < \infty, \\
\label{eq:overlaps_proof_8} \E\bigl[ d_{\pi(i+1)}^\msl \bigr] 
= \E\Bigl[ \E\bigl[ d_{\pi(i+1)}^\msl \bcond \Pi_1^i \bigr] \Bigr]
= \E\biggl[ \frac{\sum_{v\not\in\Pi_1^i} (d_v^\msl)^2}{\sum_{v\not\in\Pi_1^i} d_v^\msl} \biggr]
\leq \frac{\sum_{v\in[N_n]} (d_v^\msl)^2}{\sum_{v\in[N_n]} d_v^\msl - \sum_{k\in[i]} d_{(k)}^\msl}
\end{gather}
almost surely, by taking the worst-case scenario. 
We claim that for any $1/2>\eps'>0$ and $i \leq \delta' N_n$ with an appropriate $\delta' = \delta'(\eps)$ and $n$ large enough, 
\beq\label{eq:overlaps_proof_10} \E\bigl[ d_{\pi(i+1)}^\msl \bigr] 
\leq \frac{1}{1-\eps'} \E\bigl[ d_{\pi(1)}^\msl \bigr] 
= \frac{1}{1-\eps'} \E[D_n^{\msl,\star}]
\leq (1+2\eps') \E[D_n^{\msl,\star}]. \eeq
Comparing \eqref{eq:overlaps_proof_7} and \eqref{eq:overlaps_proof_8}, clearly it is sufficient to show that $\sum_{k\in[i]} d_{(k)}^\msl \leq \eps' \he_n$, with $\he_n = \sum_{v\in[N_n]} d_v^\msl$ (see \eqref{eq:def_halfedges}) and $i \leq \delta' N_n$. To choose an appropriate $\delta'$, first note that $\E[D^\msl] < \infty$ by \cref{asmp:convergence} \cref{cond:mean_ldeg}, thus we can choose $K' = K'(\eps)$ so that $\E[D^\msl \ind_{\{D^\msl>K'\}}] \leq (\eps'/2)\, \E[D^\msl]$. 
Now define $\delta':= \P(D^\msl > K')/2$, so that 
for $n$ large enough, $\P(D_n^\msl>K')>\delta'$; equivalently, $d_{(\floor{\delta' N_n})}^\msl > K'$. Thus,
\beq\label{eq:overlaps_proof_9} \sum_{k\in[i]} d_{(k)}^\msl 
\leq \sum_{k \leq \delta' N_n} d_{(k)}^\msl 
= \sum_{k \leq \delta' N_n} d_{(k)}^\msl \ind_{\{d_{(k)}^\msl > K'\}} 
\leq \sum_{v\in[N_n]} d_v^\msl \ind_{\{d_v^\msl > K'\}} 
= N_n \E[D_n^\msl \ind_{\{D_n^\msl > K'\}}]. \eeq
By \cref{asmp:convergence} \cref{cond:mean_ldeg}, the collection $(D_n^\msl)_{n\in\N}$ is uniformly integrable, thus $\E[D_n^\msl \ind_{\{D_n^\msl > K'\}}] \to \E[D^\msl \ind_{\{D^\msl > K'\}}] \leq (\eps'/2)\, \E[D^\msl]$ as $n\to\infty$. Further, by \cref{rem:asmp_consequences} \cref{cond:partition_ratio}, $N_n/\he_n \to 1/\E[D^\msl]$. That is, $N_n \E[D_n^\msl \ind_{\{D_n^\msl > K'\}}] / \he_n \to \eps'/2$, hence \eqref{eq:overlaps_proof_9} implies that
for $n$ large enough, $\sum_{k\in[i]} d_{(k)}^\msl \leq \eps' \he_n$ for all $i\leq \delta' N_n$, as required, which concludes the proof of \eqref{eq:overlaps_proof_10}. 

We now combine the above results. 
Recall that $D_n^\msr$ is independent from the size-biased reordering $(d_{\pi(i)}^\msl)_{i\leq N_n}$, and that $D_n^\msr \leq \dmax{\msr} \leq \delta' N_n$ almost surely for $n$ large enough. Taking expectation in \eqref{eq:overlaps_proof_6} and using \eqref{eq:overlaps_proof_10}, we obtain
\beq\label{eq:overlaps_proof_11} \begin{split} 
&\E\bigl[ \varphi(\BCM_n,V_n^\msr) \bigr] 
\leq \E\Bigl[ \sum_{i=1}^{D_n^\msr} \bigl( d_{\pi(i)}^\msl - 1 \bigr) \Bigr] 
= \E\Bigl[ \E\Bigl[ \sum_{i=1}^{D_n^\msr} \bigl( d_{\pi(i)}^\msl - 1 \bigr) \bcond D_n^\msr \Bigr] \Bigr] \\
&= \E\Bigl[ \sum_{i=1}^{D_n^\msr} \bigl( \E\bigl[ d_{\pi(i)}^\msl \bigr] - 1 \bigr) \Bigr] 
\leq \E\Bigl[ \sum_{i=1}^{D_n^\msr} \bigl( (1+2\eps') \E[D_n^{\msl,\star}] - 1] \bigr) \Bigr] 
= \E[D_n^\msr] \bigl( \E[\wit D_n^\msl] + 2\eps'\E[D_n^{\msl,\star}] \bigr), 
\end{split} \eeq
since $\E[\wit D_n^\msl] = \E[D_n^{\msl,\star}] - 1$. By \cref{asmp:convergence} \cref{cond:mean_rdeg}, $\E[D_n^\msr]\to\E[D^\msr]<\infty$, and \eqref{cond:secondmom} ensures that $\E[\wit D_n^\msl] \to \E[\wit D^\msl] < \infty$, as well as that $(\E[D_n^{\msl,\star}])_{n\in\N}$ is bounded. Thus, for any $\eps,\delta>0$, with $\eps' = \eps'(\eps,\delta)$ above chosen appropriately, for $n$ large enough,
\beq\label{eq:overlaps_proof_12} \E\bigl[ \varphi(\BCM_n,V_n^\msr) \bigr] 
\leq \E[D^\msr] \E[\wit D^\msl] + \eps\delta/18 
= \E\bigl[\varphi(\BP_\msr,\otroot)\bigr] + \eps\delta/18. \eeq
That is, we have obtained a bound on $\E\bigl[ \varphi(\BCM_n,V_n^\msr) \bigr]$ uniformly in $n$. However, this is \emph{not} enough to obtain a bound on $\E \bigl[ \varphi(\BCM_n,V_n^\msr) \bigr] - \E \bigl[ \varphi(\BCM_n,V_n^\msr) \midmin K \bigr]$ uniformly in $n$; such a statement requires uniform integrability. Instead, we rely on another triangle inequality and our previous results. 
Recall that by the choice of $K$ in \eqref{eq:overlaps_proof_term2},
\beq 0 \leq \E\bigl[ \varphi(\BP_\msr,\otroot) \bigr] - \E\bigl[ \varphi(\BP_\msr,\otroot) \midmin K \bigr] < \eps\delta/18. \eeq
Since $\E\bigl[ \varphi(\BCM_n,V_n^\msr) \midmin K \bcond \bmatch_n \bigr]$ is bounded by $K$, and $\E\bigl[ \varphi(\BP_\msr,\otroot) \midmin K \bigr]$ is a constant, the convergence in probability in \eqref{eq:overlaps_proof_term1} implies convergence of mean. Thus, for $n$ large enough,
\beq\label{eq:overlaps_proof_14} \bigl\lvert\, \E\bigl[ \varphi(\BCM_n,V_n^\msr) \midmin K \bigr] 
- \E\bigl[ \varphi(\BP_\msr,\otroot) \midmin K \bigr] \,\bigr\rvert < \eps\delta/18. \eeq
Noting that $\E\bigl[ \varphi(\BCM_n,V_n^\msr) \midmin K \bigr] \leq \E\bigl[ \varphi(\BCM_n,V_n^\msr) \bigr]$ and combining (\ref{eq:overlaps_proof_12}-\ref{eq:overlaps_proof_14}) via the triangle inequality, we obtain that for $n$ large enough,
\beq \E\bigl[ \varphi(\BCM_n,V_n^\msr) \bigr] 
- \E\bigl[ \varphi(\BCM_n,V_n^\msr) \midmin K \bigr] 
\leq 3 \eps\delta/18 = \eps\delta/6. \eeq
Then \eqref{eq:overlaps_proof_term3} follows by Markov's inequality. Since we have proved \cref{prop:number_of_overlaps} subject to \eqref{eq:overlaps_proof_term3}, this concludes the proof of \cref{prop:number_of_overlaps}. 
\proofends

\subsubsection{Proof of Theorem \ref{thm:overlaps} (\ref{perspective:global})}
\label{sss:overlaps_global}
Recall $\intersection(a,b)$ and $\overlapset_k$ from (\ref{eq:def_intersection_size}-\ref{eq:def_overlapset}). By \cref{prop:number_of_overlaps}, $\abs{\overlapset_1}$ is of order $M_n$, thus to show that $\abs{\overlapset_2}/\abs{\overlapset_1} = o_{\sss\P}(1)$, it is sufficient to prove that $\abs{\overlapset_2} = o_{\sss\P}(M_n)$, which we carry out via a first moment method. We compute
\beq 2\, \E\bigl[ \abs{\overlapset_2} \bigr] 
= \E\Bigl[ \sum_{\substack{a,b\in\rightpar\\a\neq b}} \ind_{\{\intersection(a,b)\geq 2\}} \Bigr] 
= \sum_{\substack{a,b\in\rightpar\\a\neq b}} \P\bigl( \intersection(a,b)\geq2 \bigr). \eeq
With some $K$ to be chosen later, we split the sum
\beq\label{eq:overlapsglob_split} \sum_{\substack{a,b\in\rightpar\\a\neq b}} \P\bigl( \intersection(a,b)\geq2 \bigr) 
= \sum_{\substack{a,b\in\rightpar\\a\neq b\\d_a^\msr\leq K}} \P\bigl( \intersection(a,b)\geq2 \bigr) 
+ \sum_{\substack{a,b\in\rightpar\\a\neq b\\d_a^\msr> K}} \P\bigl( \intersection(a,b)\geq2 \bigr). \eeq
We start by bounding the first term. Recall that $v \comrole \com_a$ denotes the event that $v$ takes a community role in $\com_a$. For individuals $v_1,\ldots,v_k$ and communities $a_1,\ldots,a_l$, denote the event that all $k$ individuals are in all $l$ communities by
\beq \bigl\{ \{v_1,\ldots,v_k\} \comroles \{\com_{a_1},\ldots,\com_{a_l}\} \bigr\}
:= \cap_{i\leq k} \cap_{j\leq l} \{ v_i \comrole \com_{a_j} \}. \eeq
Further recall $d_v^\msl = \ldeg(v)$ and $d_a^\msr = \rdeg(a)$ from \cref{ss:RIGC_def} and $\he_n$ from \eqref{eq:def_halfedges}. By the union bound, 
\beq\label{eq:overlapsglob_1} \bal \P\bigl( \intersection(a,b)\geq2 \bigr) 
&= \P\bigl(\exists v,w\in\leftpar, v< w: \{v,w\} \comroles \{\com_a,\com_b\} \bigr) \\
&\leq \frac12 \sum_{\substack{v,w\in\leftpar\\v\neq w}} \P\bigl( \{v,w\} \comroles \{\com_a,\com_b\} \bigr) \\
&\leq \sum_{\substack{v,w\in\leftpar\\v\neq w}} \frac{d_a^\msr(d_a^\msr-1)d_b^\msr(d_b^\msr-1)d_v^\msl(d_v^\msl-1)d_w^\msl(d_w^\msl-1)}{2\cdot\he_n(\he_n-1)(\he_n-2)(\he_n-3)}. \eal \eeq
Using \eqref{eq:def_sizebiasing}, \eqref{eq:def_ldeg_rdeg}, and that $\he_n = \E[D_n^\msl] N_n$ by \cref{rem:asmp_consequences} \cref{cond:partition_ratio}, 
\beq\label{eq:overlapsglob_Etilde} \sum_{v\in\leftpar} \frac{d_v^\msl(d_v^\msl-1)}{\he_n} 
= \frac{1}{N_n} \sum_{v\in\leftpar} \frac{d_v^\msl(d_v^\msl-1)}{\E[D_n^\msl]} 
= \frac{\E\bigl[ D_n^\msl (D_n^\msl-1) \bigr]}{\E[D_n^\msl]} 
= \E\bigl[ \wit D_n^\msl \bigr]. \eeq
Since $\he_n\to\infty$ as $n\to\infty$, we have that $2\he_n(\he_n-1)(\he_n-2)(\he_n-3) \geq \he_n^4$ for $n$ large enough, thus combining (\ref{eq:overlapsglob_1}-\ref{eq:overlapsglob_Etilde}), we obtain
\beq \bal \P\bigl( \intersection(a,b) \geq 2 \bigr) 
&\leq \frac{d_a^\msr(d_a^\msr-1)d_b^\msr(d_b^\msr-1)}{\he_n^2}
\sum_{v,w\in\leftpar} \frac{d_v^\msl(d_v^\msl-1)}{\he_n} 
\frac{d_w^\msl(d_w^\msl-1)}{\he_n} \\
&\leq \frac{d_a^\msr(d_a^\msr-1)d_b^\msr(d_b^\msr-1)}{\he_n^2} \bigl(\E[\wit D_n^\msl]\bigr)^2.
\eal \eeq
Then, using the condition $d_a^\msr \leq K$, the definition of $\dmax{\msr}$ from  \cref{rem:asmp_consequences} \cref{cond:dmax} and that $\he_n = \sum_{b\in\rightpar} d_b^\msr$ by definition,
\beq\label{eq:overlapsglob_bounded} \bal \sum_{\substack{a,b\in\rightpar\\a\neq b\\d_a^\msr\leq K}} \P\bigl( \intersection(a,b) \geq 2 \bigr) 
&\leq \bigl(\E[\wit D_n^\msl]\bigr)^2 \sum_{\substack{a,b\in\rightpar\\a\neq b\\d_a^\msr\leq K}} \frac{d_a^\msr(d_a^\msr-1)d_b^\msr(d_b^\msr-1)}{\he_n^2} \\
&< \bigl(\E[\wit D_n^\msl]\bigr)^2 K^2 M_n \sum_{b\in\rightpar} \frac{d_b^\msr(\dmax{\msr}-1)}{\he_n^2} 
\leq \bigl(\E[\wit D_n^\msl]\bigr)^2 K^2 M_n \frac{\dmax{\msr}}{\he_n}. \eal \eeq
We continue by bounding the second term in \eqref{eq:overlapsglob_split}, where $d_a^\msr>K$. Using Markov's inequality, we obtain an alternative bound for the probability
\beq\label{eq:overlapsglob_2} \P\bigl( \intersection(a,b) \geq 2 \bigr) \leq \E\bigl[ \intersection(a,b) \bigr]/2. \eeq
Taking expectation in \eqref{eq:def_intersection_size} and again using \eqref{eq:overlapsglob_Etilde},
\beq\label{eq:overlapsglob_3} \E\bigl[ \intersection(a,b) \bigr] 
= \sum_{v\in\leftpar} \P\bigl( v \comroles \{\com_a,\com_b\} \bigr) 
\leq \sum_{v\in\leftpar} \frac{d_a^\msr d_v^\msl (d_v^\msl-1) d_b^\msr}{\he_n (\he_n-1)} 
= \frac{d_a^\msr d_b^\msr}{\he_n-1} \E[\wit D_n^\msl]. \eeq
Combining (\ref{eq:overlapsglob_2}-\ref{eq:overlapsglob_3}), and using that $\sum_{b\in\rightpar}d_b^\msr = \he_n \leq 2(\he_n-1)$ for $n$ large enough,
\beq\label{eq:overlapsglob_tail} \sum_{\substack{a,b\in\rightpar\\a\neq b\\d_a^\msr> K}} \P\bigl( \intersection(a,b)\geq2 \bigr) 
\leq \frac{\E[\wit D_n^\msl]}{2} \sum_{b\in\rightpar} \frac{d_b^\msr}{\he_n-1} \sum_{\substack{a\in\rightpar\\d_a^\msr> K}} d_a^\msr
\leq \E[\wit D_n^\msl] \sum_{a\in\rightpar} d_a^\msr \ind_{\{d_a^\msr>K\}}. \eeq
Combining \eqref{eq:overlapsglob_split}, \eqref{eq:overlapsglob_bounded} and \eqref{eq:overlapsglob_tail},
\beq \bal \frac{2\, \E\bigl[ \abs{\overlapset_2} \bigr]}{M_n} 
&\leq \bigl(\E[\wit D_n^\msl]\bigr)^2 K^2 \frac{\dmax{\msr}}{\he_n}
+ \E[\wit D_n^\msl] \frac{1}{M_n} \sum_{a\in\rightpar} d_a^\msr \ind_{\{d_a^\msr>K\}} \\
&= \bigl(\E[\wit D_n^\msl]\bigr)^2 K^2 \frac{\dmax{\msr}}{\he_n}
+ \E[\wit D_n^\msl] \E[D_n^\msr \ind_{\{D_n^\msr>K\}}]. \eal \eeq
We show that $\E\bigl[ \abs{\overlapset_2} \bigr] / M_n \to 0$ by showing that it can be made arbitrarily small for $n$ large enough. Fix an arbitrary $\eps>0$, and we will choose first $K$ then $n$ so that the obtained upper bound is smaller than $\eps$. Under the second moment condition \eqref{cond:secondmom}, $\E[\wit D_n^\msl] \to \E[\wit D^\msl] < \infty$, thus $(\E[\wit D_n^\msl])_{n\in\N}$ is bounded. By \cref{asmp:convergence} \cref{cond:mean_rdeg}, $(D_n^\msr)_{n\in\N}$ is uniformly integrable, thus we can choose $K=K(\eps)$ large enough so that for all $n$ large enough,
\beq \E[\wit D_n^\msl] \E[D_n^\msr \ind_{\{D_n^\msr>K\}}] \leq \eps. \eeq
Again using that $(\E[\wit D_n^\msl])_{n\in\N}$ is bounded, further that $K$ is now fixed and $\dmax{\msr}/\he_n\to0$ by \cref{rem:asmp_consequences} \cref{cond:dmax}, we conclude that for $n$ large enough,
\beq \bigl(\E[\wit D_n^\msl]\bigr)^2 K^2 \frac{\dmax{\msr}}{\he_n} \leq \eps. \eeq
We conclude that for $n$ large enough, $\E\bigl[ \abs{\overlapset_2} \bigr]/M_n \leq \eps$, which is equivalent to $\E\bigl[ \abs{\overlapset_2} \bigr] = o(M_n)$. By Markov's inequality, $\abs{\overlapset_2} = o_{\sss\P}(M_n)$, which combined with \cref{prop:number_of_overlaps} implies $\abs{\overlapset_2}/\abs{\overlapset_1} = o_{\sss\P}(1)$. This concludes the proof of \cref{thm:overlaps} \cref{perspective:global}.
\proofends

\FloatBarrier

\section*{Acknowledgements}

This work is supported by the Netherlands Organisation for Scientific Research (NWO) through VICI grant 639.033.806 (RvdH), VENI grant 639.031.447 (JK), the Gravitation {\sc Networks} grant 024.002.003 (RvdH), and TOP grant 613.001.451 (VV). VV thanks Lorenzo Federico and Clara Stegehuis for helpful discussions throughout the project.

\printbibliography

\end{document}